\documentclass{amsart}
\usepackage{graphicx}
\usepackage{amsfonts,amsmath,amssymb,amsthm}
\usepackage{mathtools}
\usepackage[dvipsnames]{xcolor}
\usepackage[margin=1in]{geometry}
\usepackage{faktor}
\usepackage{url}
\usepackage{hyperref}
\usepackage{cleveref}
\usepackage[normalem]{ulem}

\newcommand{\op}[1]{\operatorname{#1}}
\newcommand{\Qp}{\mathbb{Q}_p}
\newcommand{\Qpbar}{\overline{\mathbb{Q}}_p}

\newtheorem{theorem}{Theorem}

\newtheorem{corollary}[theorem]{Corollary}
\newtheorem{example}[theorem]{Example}
\newtheorem{lemma}[theorem]{Lemma}
\newtheorem{setup}[theorem]{Setup}
\newtheorem{notation}[theorem]{Notation}
\newtheorem*{notation*}{Notation}
\newtheorem*{example*}{Example}
\newtheorem{proposition}[theorem]{Proposition}

\theoremstyle{definition}
\newtheorem{definition}[theorem]{Definition}
\newtheorem{remark}[theorem]{Remark}

\newcommand{\Z}{\mathbb{Z}}
\newcommand{\Q}{\mathbb{Q}}
\newcommand{\C}{\mathbb{C}}
\newcommand{\F}{\mathbb{F}}

\newcommand{\Gal}{\operatorname{Gal}}

\newcommand{\GL}{\operatorname{GL}}
\newcommand{\SL}{\operatorname{SL}}

\newcommand{\Fheight}{\operatorname{h}_\mathcal{F}}

\newcommand{\NRep}[1]{\operatorname{Im}\rho_{E,#1}}
\newcommand{\pnRep}[2]{\operatorname{Im}\rho_{E,#1^#2}}

\newcommand{\BaseField}{F}
\newcommand{\GoodReductionExtension}{K}
\newcommand{\Dbasis}{b}

\newcommand{\isom}{\tilde{w}}
\newcommand{\Frob}{\varphi}

\numberwithin{theorem}{section}
\numberwithin{equation}{section}

\title{Explicit $p$-adic Hodge theory for elliptic curves and non-split Cartan images}

 \author{Matthew Bisatt}
 \address{Faculty of Mathematics and Computer Science, Adam Mickiewicz University, 61-614 Pozna\'{n}, Poland}
 \email{matthew.bisatt@amu.edu.pl}

 \author{Lorenzo Furio}
 \address{Sorbonne Université, CNRS, IMJ-PRG. 4, place Jussieu, 75252 Paris, France}
 \email{lorenzo.furio@imj-prg.fr}

 \author{Davide Lombardo}
 \address{Dipartimento di Matematica, Universit\`a di Pisa, Largo Bruno Pontecorvo 5, 56127 Pisa, Italy}
 \email{davide.lombardo@unipi.it}

\date{}

\begin{document}

\begin{abstract}
    Let $E/\mathbb{Q}_p$ be an elliptic curve whose mod $p$ Galois image is contained in the normaliser of a non-split Cartan. We classify the possible $p$-adic images of $E$ using tools from $p$-adic Hodge theory via a careful analysis of the local Galois structure of the $p$-power torsion of $E$. We pay special attention to the case where $E$ has potentially supersingular reduction, where we give an algorithm to determine the corresponding filtered $(\varphi,\op{Gal}(K/\Qp))$-module from a Weierstrass model (which appears to be novel), and introduce alternative division polynomials that may be of independent interest.

    We deduce global consequences for elliptic curves $E/\Q$: when the mod $p$ representation of $E$ has non-split Cartan image and $E$ doesn't have CM, 
    the $p$-adic image must be the full preimage of the normaliser of a mod $p^n$ non-split Cartan for some $n \geq 1$. As an application, we sharpen existing bounds on the adelic image in terms of the Weil height of the $j$-invariant. 
\end{abstract}

\maketitle

\setcounter{tocdepth}{1}
\tableofcontents

\section{Introduction}

Galois representations attached to elliptic curves $E$ over fields $F$ of arithmetic significance are a cornerstone of modern number theory, connecting to many central themes, both theoretical and computational. Of particular interest are the mod $p^n$ and $p$-adic representations $\rho_{E,p^n}$ and $\rho_{E,p^\infty}$, which encode the Galois action on the $p^n$-torsion subgroup $E[p^n]$ and the Tate module $T_pE$ respectively (after fixing bases, we always regard $\rho_{E, p^n}$ and $\rho_{E, p^\infty}$ as taking values in $\GL_2(\Z/p^n\Z)$ and $\GL_2(\Z_p)$, respectively). In this paper, we develop tools to analyse the images of these representations over $\Q_p$, and apply them to obtain global results for elliptic curves over $\Q$.

When $F=\Q$, Serre proved the images of $\rho_{E,p}$ and $\rho_{E,p^{\infty}}$ are intertwined: if $p>3$ then one is surjective if and only if the other is \cite[IV-23, Lemma 3]{serre-abrep}. In his landmark open image theorem \cite{MR387283}, he further showed that $\rho_{E,p^{\infty}}$ is surjective when $p$ is sufficiently large so long as $E$ does not have potential complex multiplication (CM).  In the same paper, he then posed a natural uniformity question: does there exist a constant $N \geq 3$, such that $\rho_{E,p}$ (and hence also $\rho_{E,p^{\infty}}$) is surjective for every prime $p>N$, uniformly in non-CM curves $E/\Q$? This became known as \textit{Serre's uniformity question}. Despite substantial progress, the problem remains open, though it is conjectured to have an affirmative answer with optimal constant $N=37$.

By the work of many authors \cite{serre81, mazur78, biluparent11, biluparent13, zywina15, lefournlemos21, furiolombardo23}, it is known that for a non-CM elliptic curve $E/\Q$ and a prime $p>37$, the image of the representation $\rho_{E,p}$ is either $\GL_2(\F_p)$ or the normaliser of a non-split Cartan, which we henceforth denote by $C_{ns}^+(p)$ (and $C_{ns}^+(p^n)$ for the mod $p^n$ versions). Recent work has also led to significant progress towards the finer classification of possible $p$-adic images, which is an important first step in Mazur's \textit{Program B}. When $p>7$, Rouse--Sutherland--Zureick-Brown \cite[Theorem 1.1.6]{rszb22} have classified the possible $p$-adic images unless $\op{Im}\rho_{E,p} \subseteq C_{ns}^+(p)$. In the non-split Cartan case, much less is known about the images of $\rho_{E, p^n}$ for $n \geq 2$. Our first global result bridges this gap to describe the full $p$-adic tower:

\begin{theorem}[=\Cref{thm: no sharp groups over Q}]\label{introthm: no sharp groups over Q}
Let $E/\Q$ be an elliptic curve without CM and let $p > 7$ be a prime. Suppose that $\op{Im} \rho_{E, p} \subseteq C_{ns}^+(p)$. Then there exists $n \geq 1$ such that
\[
\op{Im} \rho_{E, p^\infty} = \pi_n^{-1} \left(C_{ns}^+(p^n)\right),
\]
where $\pi_n : \GL_2(\Z_p) \to \GL_2(\Z/p^n\Z)$ is the canonical projection.
\end{theorem}

We note that if Serre's uniformity question has a positive answer, then the containment $\op{Im}\rho_{E,p} \subseteq C_{ns}^+(p)$ only occurs for finitely many primes $p$, but this question remains wide open. 
The refined classification of \Cref{introthm: no sharp groups over Q} precludes the existence of images that had not previously been ruled out, thereby allowing us to strengthen previous adelic image bounds \cite{MR3437765, furio24}. 
We state here a simplified version of our result in terms of the $j$-invariant; a fully explicit version is given in \Cref{thm: adelic bound}. We note that the dependence on the height of $E$ is essentially optimal for the available methods.

\begin{theorem}\label{introthm: adelic bound}
    Let $E/\Q$ be an elliptic curve without complex multiplication and let $\operatorname{h}(j(E))$ be the absolute logarithmic Weil height of its $j$-invariant. For every $\varepsilon > 0$ there exists an effective constant $C_\varepsilon$ such that the index of the adelic Galois representation of $E$ is bounded by $C_\varepsilon \cdot \max\{1, \operatorname{h}(j(E))\}^{2+\varepsilon}$.
\end{theorem}

\Cref{introthm: no sharp groups over Q} is a consequence of a more explicit local one, where the hypotheses enable us to suppose that $E$ has potentially supersingular reduction. Recall that the semistability defect $e$ of an elliptic curve $E/\Qp$ is the minimal degree of an extension $L/\Qp$ such that the base change $E_L$ is semistable; when $p>3$, one has $e \in \{1,2,3,4,6\}$.

\begin{theorem}[=\Cref{lemma: p-adic image in the twisted crystalline case} + \Cref{cor: explicit p-adic image}]\label{introthm: more precise description padic Galreps}
    Let $p>7$ be a prime and let $E/\Q_p$ be an elliptic curve with semistability defect $e$, and suppose that $\operatorname{Im}\rho_{E,p} \subseteq C_{ns}^+(p)$. Then $E$ has potentially good supersingular reduction and the following hold.
    \begin{enumerate}
        \item If $e\leq 2$ then $\operatorname{Im}\rho_{E,p^n} = C_{ns}^+(p^n)$ for all $n \geq 1$.
        \item Suppose $e \geq 3$. Let $j$ be the $j$-invariant of $E$ and let $n_0$ be the positive integer
    \[
    n_0 := \begin{cases}
    \begin{array}{lll}
    \!\left\lfloor \frac{1}{3}v(j) \right\rfloor & \quad & \text{if $e \in \{3,6\}$} \\
    \!\left\lfloor \frac{1}{2}v(j-1728) \right\rfloor & \quad & \text{if $e = 4$.}
    \end{array}
    \end{cases}
    \]
    Let $\pi_{n_0} : \GL_2(\Z_p) \to \GL_2(\Z/p^{n_0}\Z)$ be the canonical projection. If $p>\sqrt{n_0+1}$, then 
\[
\operatorname{Im}\rho_{E,p^{n_0}} \subseteq C_{ns}^+(p^{n_0}) \text{ with index dividing $3$,} \qquad\qquad \text{and} \,\, \operatorname{Im}\rho_{E,p^\infty} = \pi_{n_0}^{-1}(\operatorname{Im}\rho_{E,p^{n_0}}).
\]
Moreover, if $e=4$ or $p \not\equiv 2,5 \bmod 9$, then $\operatorname{Im}\rho_{E,p^{n_0}} = C_{ns}^+(p^{n_0})$. 
    \end{enumerate}
\end{theorem}

\begin{remark}
    As explained in \Cref{rmk: index 3}, even when $e \in \{3,6\}$ and $p \equiv 2,5 \bmod 9$ one can easily decide whether the index $[C_{ns}^+(p^{n_0}) : \operatorname{Im}\rho_{E,p^{n_0}}]$ is $1$ or $3$. In particular, for all such $e$ and $p$ there exist elliptic curves as in the statement for which the index is $3$.
\end{remark}

The local classification problem has seen considerably less attention than the global version. If $E/\Qp$ has potentially good supersingular reduction and doesn't have \textit{potential formal CM}, a result of Serre \cite[Théorème 5]{MR242839} shows that $\op{Im} \rho_{E,p^{\infty}}$ contains an open subgroup of $\GL_2(\Z_p)$, but we are unaware of any bounds on the index of this subgroup. Our work doesn't just bound this index, but precisely describes $\op{Im} \rho_{E,p^{\infty}}$ in this setup.

\subsection{Explicit $p$-adic Hodge theory approach}
We now give more details on our method for proving \Cref{introthm: more precise description padic Galreps}, where the novelty in our work lies in the systematic and computationally accessible utilisation of $p$-adic Hodge theory to study local Galois representations of $E/\Qp$ in the supersingular setting.

If $E$ has good supersingular reduction over $\Qp$ or a quadratic extension, then it is known that there are only two possible isomorphism classes for $T_pE$ and the $p$-adic image is known (see \cite[top of p.~34]{Volkov} and \Cref{lemma: p-adic image in the twisted crystalline case}). We can then suppose that $E$ has semistability defect $e \geq 3$ and becomes semistable over $K=\Qp(\sqrt[e]{-p},\zeta_e)$. For $p>7$, we show that in this context $T_pE$ is irreducible as a Galois module and hence its Galois structure is uniquely determined by its rational counterpart $V_pE:=T_pE \otimes_{\Z_p} \Q_p$.

Known results in the literature, due in particular to work of Volkov, imply that $V_pE$ (which is potentially crystalline) can be encoded by a single deformation parameter $\alpha \in \mathbb{P}^1(\Qp)$, corresponding to a filtered $(\varphi,\Gal(K/\Qp))$-module $D_{\alpha}$ (\Cref{def: D alpha}). We are able to extract from $D_\alpha$ concrete arithmetic information about the torsion representations $E[p^k]$ and thereby construct the following \textit{alternative ``division'' polynomials} for $E$.

\begin{theorem}[=\Cref{thm: p2 torsion points correspond to roots of g} + \Cref{rmk: twisted Galois action 2}]\label{introthm: alternative division polynomials}
    Let $p>7$ and let $E/\Qp$ be an elliptic curve with potentially supersingular reduction, semistability defect $e\geq 3$ and deformation parameter $\alpha$. For $k \geq 1$, define 
    \[
    g_k:=x^{p^{2k}} + \sum_{n=1}^k (-1)^np^n(x^{p^{2k-2n}} + \alpha^{-1}(\sqrt[e]{-p})^2x^{p^{2k+1-2n}}).
    \]
    If $v(\alpha^{-1}) \geq 0$ and $p> \max\{\sqrt{k}, k-v(\alpha^{-1})-1\}$, then there is a $\op{Gal}(\overline{\Q}_p/\Qp)$-equivariant bijection between $E[p^k]$ and the roots $\mathcal{R}_k$ of $g_k$ (for a suitable twisted Galois action on $\mathcal{R}_k$). Moreover, the fields $K(E[p^k])$ and $K(\mathcal{R}_k)$ are equal.
\end{theorem}

The recursive structure of the polynomials $g_k$ and their linear dependence on $\alpha^{-1}$ emphasise their versatility in comparison to the classical (but considerably more complicated) division polynomials; see \Cref{prop: shifted Newton polygon}, \Cref{rmk: bijection is sensible}, and \S\ref{sect: construction of Phik}. We expect that these polynomials $g_k$ will be useful in other problems involving local $p$-adic representations of elliptic curves, and note that they may also be used when the semistability defect is $1$ or $2$ (\Cref{rmk: alternative division polynomials for small e}).

We briefly remark on the hypothesis that $v(\alpha^{-1}) \geq 0$. Replacing $E$ by a ramified quadratic twist changes the parameter $\alpha$ to $\pm p^2\alpha^{-1}$ (\Cref{rmk: alpha role}(i)). Thus, whenever $v(\alpha^{-1}) \ne -1$, one can still construct a Galois-equivariant bijection between $\mathcal{R}_k$ and $E[p^k] \otimes \chi$ for a suitable quadratic character $\chi$. The case $v(\alpha^{-1})=-1$ corresponds precisely to $E$ having a canonical subgroup of order $p$ (\Cref{rmk: alpha role}(ii)), and will be excluded from our applications by the Cartan image assumption (\Cref{cor: supersingular and no canonical subgroup}); moreover, in this case the Tate module is no longer irreducible, so different methods are necessary.

The proof of \Cref{introthm: more precise description padic Galreps} is then based on two main pillars. On the one hand, we establish a dichotomy in the behaviour of the roots of $g_k$ depending on the valuation of the deformation parameter $\alpha$: when $k$ is sufficiently small compared to $v(\alpha^{-1})$, we show that the image of the Galois action on $\mathcal{R}_k$ is contained in $C_{ns}^+(p^k)$; conceptually, this happens because $E$ is `sufficiently close' to having complex multiplication. Conversely, for $k$ large we see maximal growth, namely $\op{Im} \rho_{E,p^k}$ is the full inverse image of $\op{Im} \rho_{E,p^{k-1}}$ inside $\op{GL}_2(\Z/p^k\Z)$. Proving this dichotomy gives a version of \Cref{introthm: more precise description padic Galreps} that may be expressed in terms of the deformation parameter $\alpha$ (see \Cref{thm: more precise description padic Galreps}).

The second pillar in the proof of \Cref{introthm: more precise description padic Galreps} is a connection between $\alpha$ and the $j$-invariant of $E$, which we believe to be an interesting result in its own right.
Specifically, Volkov's results give an abstract link between the set of $E/\Q_p$ with potentially good supersingular reduction and semistability defect $e$ and the $(\varphi,\Gal(K/\Qp))$-modules $D_{\alpha}$, but the $\alpha$ corresponding to such a curve is not immediately computable. We solve this problem by connecting Volkov's description directly with Weierstrass models and formal logarithms.

\begin{theorem}[=\Cref{prop: beta algorithm,prop: valuation beta,prop: beta parameter,prop: epsilon via discriminant}]\label{introthm: formula for beta}
Let $p>3$ be a prime. Suppose $E/\Qp$ is an elliptic curve having potentially good supersingular reduction, with $j$-invariant $j$, minimal discriminant $\Delta$, and semistability defect $e \in \{3,4,6\}$ satisfying $e<p-1$. Set $\pi_e=\sqrt[e]{-p}$ and let $\log_{\hat{E}}(t)=\sum_{n \ge 0} d_nt^n$ be the formal logarithm attached to a short Weierstrass model of $E/\Qp(\pi_e)$ with good reduction. Define
\[
\beta = \lim_{k \to \infty} -\frac{p}{\pi_e} \frac{d_{p^{2k+1}}}{d_{p^{2k}}}, \qquad \text{which has valuation } v(\beta)=\begin{cases}
\begin{array}{lll}
\frac{1}{3}v(j) - v(\pi_e) & \, \text{ if $e \in \{3,6\}$}, \\
\frac{1}{2}v(j-1728) - v(\pi_e) & \,  \text{ if $e =4$.}
\end{array}
\end{cases}
\]

The deformation parameter $\alpha \in \mathbb{P}^1(\Qp)$ corresponding to $E/\Qp$ is
\[
\alpha = \begin{cases}
\begin{array}{cc}
    -p \cdot (\beta/\pi_e)^{-1} & \text{ if } v(\Delta)<6 \\
    -p \cdot \beta/\pi_e^{e-3} & \text{ if } v(\Delta)>6.
    \end{array}
\end{cases}
\]
\end{theorem}

Although \Cref{introthm: formula for beta} gives only a limit formula for $\beta$, we discuss in \Cref{subsect: beta algorithm} how to make it computable, in the sense of being able to determine $\beta$ to any given $p$-adic precision. We also note that the ratios $d_{p^{2k+1}}/d_{p^{2k}}$ may be interpreted as higher analogues of the classical Hasse invariant of $E$, see \Cref{rmk: relation to Hasse invariant}.

Having a formula for $\beta$ then allows us to get access to the filtered $(\varphi, \op{Gal}(K/\Q_p))$-module of any given elliptic curve $E$ as above, enabling the use of tools from $p$-adic Hodge theory which, while available in principle, were previously not accessible in practice. We also emphasise that \Cref{introthm: formula for beta} has greater applicability than \Cref{introthm: more precise description padic Galreps}, since it holds for any potentially supersingular curve with semistability defect $e \geq 3$.

\subsection{Relation to existing literature}

\subsubsection*{Previous classification of $p$-adic representations}
\Cref{introthm: no sharp groups over Q} builds on previous work of the second author \cite[Theorem 1.6]{furio24}, who gave a similar classification, showing that either: $\operatorname{Im}\rho_{E,p^\infty}$ is the inverse image of $C_{ns}^+(p^n)$ for some $n \geq 1$; or $\operatorname{Im}\rho_{E,p^\infty} = \pi_2^{-1}(G_{ns}^\#(p^2))$, where $G_{ns}^\#(p^2)$ is an explicit group of level $p^2$ which we describe in \Cref{def: ciuffetto}.

One of the stepping stones towards \Cref{introthm: more precise description padic Galreps} -- namely \Cref{thm: no sharp group over Qp} -- removes this group-theoretic obstruction, namely it shows that $\rho_{E,p^2}(\Gal(\overline{\Q}_p/\Q_p))$ cannot be contained in $G_{ns}^{\#}(p^2)$ when $p>7$. This can be interpreted as a $p$-adic obstruction to the existence of rational points on the corresponding modular curve $X^{\#}_{ns}(p^2)$, confirming the calculation of Rouse--Sutherland--Zureick-Brown when $p=11$ \cite[Section 8.7, label \texttt{121.605.41.1}]{rszb22}.
When $p=7$, the situation is substantially different: $X_{ns}^\#(7^2)$ has a $\Q$-rational point, hence in particular (by smoothness) infinitely many $\Q_7$-rational points. The elliptic curves representing the points on $X_{ns}^\#(7^2)$ have potentially ordinary reduction, in contrast to the curves studied in this paper, for which the reduction is potentially supersingular. The case $p=7$ is examined in detail in \cite{furio20257adicgaloisrepresentationselliptic}.

We also remark that \Cref{introthm: more precise description padic Galreps} lies beyond the reach of purely group-theoretic arguments, since the image at level $p$ does not determine the higher-level images. In particular, group theory does not provide information about the level at which the $p$-adic tower stabilises. The additional information given by the crystalline theory is therefore essential in our analysis.

\subsubsection*{Galois representations of elliptic curves in families}

It is well-known that the mod $p^n$ Galois representation of an elliptic curve $E/\Qp: y^2=x^3+Ax+B$ is locally constant (see for example \cite[Theorem~5.1(1)]{Kisin} for a much more general statement); i.e., there exists a (typically non-explicit) $p$-adic neighbourhood of the coefficients $A, B$ where the Galois representation remains unchanged. The novel virtue of our approach is that we are able to determine the precise neighbourhood when $E$ is as in \Cref{introthm: more precise description padic Galreps} via a delicate examination of division polynomials, coupled with Krasner's lemma. In particular, the value $n_0$ in \Cref{introthm: more precise description padic Galreps} is maximal such that the mod $p^{n_0}$ representation of $E$ coincides with the mod $p^{n_0}$ representation of a certain CM curve $E_{\textrm{CM}}/\Qp$ with $j$-invariant $0$ (resp. $1728$) if $e=3,6$ (resp. $e=4$), and is therefore contained in $C_{ns}^+(p^{n_0})$. 

Our work is also related to the study of ramification in division fields for the family of elliptic curves with potentially supersingular reduction without a canonical subgroup of order $p$. Specifically \cite[Theorem 1.1]{smith23} (see also \cite{MR3501021}) proves that the ramification degree of $\Q_p(E[p^k])$ is divisible by $(p^2-1)p^{2k-2}$ whereas our proof of \Cref{thm: valuation 0} crucially improves this to $(p^2-1)p^{2k}$ (under appropriate assumptions). This additional factor is pivotal in proving the maximal growth condition $[\Q_p(E[p^{k}]) : \Q_p(E[p^{k-1}])] = p^4$ when $k$ is sufficiently large.

\subsubsection*{Other versions of $p$-adic Hodge theory}
We point out that -- despite the fact that we are mainly interested in integral and torsion representations -- we use a \textit{rational}, rather than \textit{integral}, version of $p$-adic Hodge theory. While several flavours of integral $p$-adic Hodge theory are now available (based for example on Breuil-Kisin modules \cite{MR1971512, MR2263197} -- see also \cite{MR2191528} for the case of torsion representations -- or Zink's displays \cite{MR1922825}), these theories currently do not seem to be amenable to fully explicit calculations. For example, to the best of the authors' knowledge, the Breuil-Kisin machinery relies in part on approximation statements for Robba rings which do not appear to admit a systematic treatment in families. In our setting, we can use the rational version of the theory, as developed by Fontaine \cite{Fontaine-p-divisible, Fontaine82, MR1293972} and Volkov \cite{Volkov}, because of the easy but important fact that the rational $p$-adic Tate module determines the integral one, see \Cref{lemma: any lattice will do}.

A further comment is related to the restriction $p>\sqrt{n_0+1}$ that appears in \Cref{introthm: more precise description padic Galreps}. This doesn't come from the application of $p$-adic Hodge theory: the connection we establish between the $j$-invariant of $E$ and the $p$-adic Hodge-theoretic invariants $\alpha$ and $\beta$ holds in much greater generality (in particular, for all $p>7$, see \Cref{introthm: formula for beta}); the difficulty lies in the $p$-adic analysis that underpins the Galois-equivariant bijection between $E[p^k]$ and the roots of $g_k(x)$. It would be interesting to remove this restriction from \Cref{introthm: more precise description padic Galreps}, which we believe also holds without it.

\subsubsection*{Other approaches to computing filtered $\varphi$-modules}
Finally, we remark on another aspect which we believe to be important, namely, the use of formal logarithms to describe the deformation parameter $\beta$. In the case of elliptic curves with potentially ordinary reduction, the $p$-adic Hodge-theoretic invariants of $V_pE$ can essentially be read off the trace of Frobenius, so there is no continuous family of deformations \cite[p.~111]{Volkov}.
In the case of potentially multiplicative reduction, one does have continuous deformations, parameterised by Tate's $q$-invariant, and the $p$-adic Hodge-theoretic invariants may be expressed in terms of suitable logarithms of $q$ \cite[p.~107]{Volkov}. As already pointed out, a similar description was missing in the case of potentially supersingular reduction, although the connection of $\beta$ with formal logarithms was predicted by Volkov \cite[Remarque on p.~133]{Volkov}, who, however, did not suggest an explicit formula.

We explain why the use of formal logarithms bypasses a substantial issue. Recall that the filtration on the filtered $(\varphi, \op{Gal}(K/\Q_p))$-module corresponding to $E$ comes from the comparison isomorphism between the crystalline and de Rham cohomologies of $E_K$ (the former base-changed to $K$): the line giving the non-trivial piece of the filtration corresponds to the line of regular differentials in $H^1_{\operatorname{dR}}(E_K/K)$. Thus, computing the filtration based on the definitions would involve making explicit the comparison between crystalline and de Rham cohomology over the ramified extension $K$, a task which seems extremely difficult in practice. The recent preprint \cite{booher2026computingcrystallinecohomologypdivisible} takes an important first step in this direction, but only for curves having good reduction over an unramified extension of $\Q_p$, which would not be sufficient for the applications in this paper.

\subsection{Layout}
We begin in \S\ref{sect: Prelims} with background on Cartan subgroups. Working over $\Q_p$, we also show that if the mod $p$ image is contained in $C_{ns}^+(p)$, then the curve has potentially supersingular reduction, and we establish a few basic properties that hold in this situation. We continue with a description and classification of possible lifts of $C_{ns}^+(p)$ to $\op{GL}_2(\Z_p)$, as well as analysing the special cases when $E/\Qp$ has CM or semistability defect at most $2$.

In \S\ref{sect: Volkov results}, we review the results of Volkov's PhD thesis which provides a description of $V_pE$ via filtered $(\varphi,\op{Gal}(K/\Qp))$-modules, most significantly providing a bijection between elements of $V_pE$ and roots of an infinite series which depends on a parameter $\alpha \in \mathbb{P}^1(\Qp)$. The polynomials $g_k$ of \Cref{introthm: alternative division polynomials} are defined in \S\ref{sect: description of p2 torsion} as truncations of this infinite series, with the roots of $g_k$ serving as avatars for the $p^k$-torsion of $E$; the remainder of the section is dedicated to the polynomials $g_k$ earning their description as \textit{alternative ``division" polynomials} by establishing a natural Galois-equivariant bijection between the roots of $g_k$ and $E[p^k]$.

In \S\ref{sect: integral Tate module} we turn our focus to proving \Cref{introthm: no sharp groups over Q}, which follows from \Cref{introthm: more precise description padic Galreps} after some preliminary reductions. In \Cref{sect: large valuation,sect: valuation zero} we prove a version of \Cref{introthm: more precise description padic Galreps} in terms of the valuation of $\alpha$: we describe the image modulo $p^k$ in terms of $\alpha$, with the split according to a suitable inequality between $k$ and $v(\alpha^{-1})$.

\Cref{sect: parameter alpha} is then devoted to determining the deformation parameter $\alpha$ associated to $E$ culminating in an algorithm for computing $\alpha$. This is achieved via an analysis of the related invariant $\beta$ for which we have a simple closed formula for $v(\beta)$ in terms of the $j$-invariant as in \Cref{introthm: formula for beta}; this completes the proof of \Cref{introthm: more precise description padic Galreps}.

We close in \S\ref{sect: adelic index} by giving the full statement and proof of our global results, namely \Cref{introthm: no sharp groups over Q} and the refined version of \Cref{introthm: adelic bound}. 

\vspace{5pt}
\noindent \textbf{Acknowledgements.}
DL is supported by MUR grant PRIN-2022HPSNCR (funded by the European Union project Next Generation EU) and is a member of the INdAM group GNSAGA. LF is supported by the grant ANR-HoLoDiRibey of the Agence Nationale de la Recherche (PI Gregorio Baldi). MB is supported by both the MUR grant PRIN-2022HPSNCR and University Adam Mickiewicz. DL thanks Fabrizio Andreatta for an illuminating conversation about integral $p$-adic Hodge theory. LF and DL are grateful to Maja Volkov for hosting them in Mons and for interesting discussions about her work. All the authors are also grateful to her for many insightful comments on a preliminary version of the paper.

\section{Preliminaries}\label{sect: Prelims}

\subsection{Cartan subgroups and their normalisers}

In this section, we introduce the non-split Cartan subgroups of $\GL_2(\Z_p)$ and review some of their basic properties. We will write $\op{Id}$ for the identity matrix.

\begin{definition}
	Given a prime $p$ and a subgroup $G \subseteq \GL_2(\Z_p)$, for every $n \ge 1$ we define:
	\begin{itemize}
		\item $G(p^n):= G \bmod {p^n} \subseteq \GL_2\left({\Z}/{p^n\Z}\right)$;
		\item $G_n := \left\lbrace A \in G \mid A \equiv \op{Id} \bmod{\, p^n} \right\rbrace$,
	\end{itemize}
    so that $G(p^n)=G/G_n$.
\end{definition}

\begin{definition}\label{def:cartan}
	Let $p$ be an odd prime and let $\varepsilon$ be the reduction modulo $p$ of the least positive integer which represents a quadratic non-residue in $\F_p^\times$. We define the non-split Cartan subgroup of $\GL_2(\Z_p)$ as
	\begin{align*}
		C_{ns}:= \left\lbrace \begin{pmatrix} a & \varepsilon b \\ b & a \end{pmatrix} \,\middle|\, a,b \in \Z_p, \ (a,b) \not\equiv (0,0) \mod p \right\rbrace.
	\end{align*}
	The normaliser of $C_{ns}$ in $\GL_2(\Z_p)$ is $C_{ns}^+:= C_{ns} \cup \begin{pmatrix} 1 & 0 \\ 0 & -1 \end{pmatrix} C_{ns}$.
\end{definition}

We list a few basic properties of the group $C_{ns}^+(p)$. The proofs are all straightforward.
\begin{remark}\label{rmk: basic properties Cns}
    The following hold:
    \begin{enumerate}
        \item The group $C_{ns}(p)$ is cyclic of order $p^2-1$. It can be identified with $\mathbb{F}_{p^2}^\times$: by writing $\mathbb{F}_{p^2} = \mathbb{F}_p \oplus \mathbb{F}_p \sqrt{\varepsilon}$ as $\mathbb{F}_p$-vector spaces, the matrix $\left( \begin{smallmatrix}
            a & \varepsilon b \\
            b & a
        \end{smallmatrix} \right)$ encodes the $\F_p$-linear map $\mathbb{F}_{p^2} \to\mathbb{F}_{p^2}$ given by multiplication by $a+b\sqrt{\varepsilon}$.
        \item The square of any matrix in $C_{ns}^+(p) \setminus C_{ns}(p)$ is a multiple of the identity. As a consequence, every element in $C_{ns}^+(p) \setminus C_{ns}(p)$ has order dividing $2(p-1)$.
        \item The eigenvalues of a matrix $g$ in $C_{ns}^+(p)$ are either of the form $\lambda, \lambda^p$ for some $\lambda \in \mathbb{F}_{p^2}^\times$ (if $g \in C_{ns}(p)$), or of the form $\lambda, -\lambda$ for some $\lambda \in \mathbb{F}_{p^2}^\times$ (if $g \in C_{ns}^+(p) \setminus C_{ns}(p)$).
    \end{enumerate}
\end{remark}

Non-split Cartan groups have been studied extensively, especially in relation to Galois representations of elliptic curves. We recall some key group-theoretic properties they enjoy (for proofs see \cite[Lemma 2.4]{zywina2011boundsserresopenimage} and \cite[\S 3.1]{furio24}).

\begin{lemma}\label{lemma: conjugacy action of Cnsp}
    Writing $G=C_{ns}^+$ we have:
\begin{enumerate}
    \item each quotient $G_n /G_{n+1}$ can be identified with an additive subgroup of $M_{2}(\F_p)$ via the map
    \[
    G_n/G_{n+1} \to  M_{2}(\F_p), \quad\qquad \operatorname{Id}+p^n A  \mapsto  A;
    \]
    \item the group $G(p)$ acts by conjugation on $M_{2}(\F_p)$ and on each quotient $G_n/G_{n+1}$;
    \item the action of $G(p)$ (or any index $3$ subgroup thereof) on $M_{2}(\F_p)$ is semisimple. For $p>7$, the $\F_p$-vector space $M_{2}(\F_p)$ decomposes as the sum of three irreducible and non-isomorphic representations of $G(p)$, given by
    \[
    V_1=\mathbb{F}_p \cdot \operatorname{Id}, \qquad V_2=\mathbb{F}_p \cdot \begin{pmatrix}
        0 & \varepsilon \\ 1 & 0
    \end{pmatrix}, \qquad V_3=\F_p \cdot \left\langle \begin{pmatrix}
        0 & \varepsilon \\ -1 & 0
    \end{pmatrix}, \begin{pmatrix}
        1 & 0 \\ 0 & -1
    \end{pmatrix} \right\rangle.
    \]
\end{enumerate}
\end{lemma}

Combining the last two properties in the lemma, it is easy to classify the subgroups of $\GL_2(\Z/p^2\Z)$ whose projection to $\GL_2(\Z/p\Z)$ is $C_{ns}^+(p)$. We'll be especially interested in one of these groups, which we introduce in the following definition.

\begin{definition}\label{def: ciuffetto}
Let $p$ be an odd prime and let $V = V_1 \oplus V_3 \subseteq M_{2}(\F_p)$. 
We define the group $G_{ns}^\#(p^2)$ as the unique subgroup of $\GL_2(\Z/p^2\Z)$ of order $2p^3(p^2-1)$ containing $\operatorname{Id}+pV$ and whose projection modulo $p$ equals $C_{ns}^+(p)$. In particular, we have $G_{ns}^\#(p^2) \cong C_{ns}^+(p) \ltimes V$, where the semidirect product is defined by the conjugation action of $C_{ns}^+(p)$ on $M_{2}(\F_p)$.
\end{definition}

\subsection{Supersingularity and canonical subgroups}

Let $E/\BaseField$ be an elliptic curve over a $p$-adic field. In this section, we show that if the image of the mod $p$ representation is contained in $C_{ns}^+(p)$, then $E$ has potentially good supersingular reduction and does not have a canonical subgroup at $p$, unless $p$ is small relative to the absolute ramification index of $\BaseField$. We first consider general $p$-adic fields and then specialise to $\BaseField=\Q_p$. 
Some of the statements that we present can be found in \cite[\S 4]{furio24}. However, these were proven for elliptic curves over number fields; here, we generalise them to elliptic curves over $p$-adic fields. We begin with introducing some notation that we shall use frequently throughout this section, including the semistability defect of $E/\BaseField$, which we remind the reader is defined as the minimal degree of an extension $L/F$ such that $E/L$ is semistable.

\vspace*{25pt}

\begin{notation*}
\phantom{~}

\begin{tabular}{cl}
   $p$ & rational odd prime \\
   $\BaseField$ & finite extension of $\Qp$ \\
   $e_\BaseField$ & ramification degree of $F/\Qp$ \\
   $E/\BaseField$ & elliptic curve \\
   $j(E)$ & $j$-invariant of $E$ \\
   $e$ & semistability defect of $E/\BaseField$ \\
   $E[p^n]$  & the $p^n$-torsion subgroup of $E(\overline{F})$ \\
   $\rho_{E,p^n}$ & the representation of $\Gal(\overline{F}/F)$ on $E[p^n]$ \\
   $\rho_{E,p^\infty}$ & the representation of $\Gal(\overline{F}/F)$ on $T_pE$ \\
   $I_L$ & the absolute inertia group of a local field $L$ \\
   $\Q_{p^2}$ & the unique quadratic unramified extension of $\Q_p$
\end{tabular}
\end{notation*}

\begin{remark}
    Since the exact image of $\rho_{E,p^n}$ depends on the choice of a basis which we do not pick, we shall write $\op{Im} \rho_{E,p^n} \subseteq H \subseteq \op{GL}_2(\mathbb{Z}/p^n\mathbb{Z})$ to mean that $\op{Im} \rho_{E,p^n}$ is contained in a conjugate subgroup of $H$, and similarly for $\rho_{E,p^\infty}$.
\end{remark}

The following simple lemma is well-known, but we provide some details.

\begin{lemma}\label{lemma: quadratic twist acquires good reduction over a small extension}
    Let $p > 3$ be a prime and let $E/F$ be an elliptic curve with potentially good reduction. Let $\Delta$ be the discriminant of any defining equation for $E$, and let $v_F$ be a valuation of $F$, normalised so that a uniformiser of $F$ has valuation $1$. 
    \begin{enumerate}
        \item $E$ obtains good reduction over a totally tamely ramified extension of degree $e=\frac{12}{\gcd(12,v_F(\Delta))} \in \{1,2,3,4,6\}$.
        \item There is a quadratic twist of $E$ that acquires good reduction over an extension of $F$ of degree at most 4.
    \end{enumerate}
\end{lemma}

\begin{proof}
    (i) is \cite[5.6 a1)]{MR387283}. For (ii), we only need to consider the case $e=6$: then $v_F(\Delta) \equiv \pm 2 \bmod{12}$ whereas the discriminant $\Delta'$ of a ramified quadratic twist $E'$ satisfies $v_F(\Delta') \equiv v_F(\Delta)+6 \equiv \pm 4 \bmod{12}$, hence $E'$ has semistability defect $e'=3$.
\end{proof}

\begin{proposition}\label{prop: potentially supersingular}
   Let $E/\BaseField$ be an elliptic curve such that $\NRep{p} \subseteq C_{ns}^+(p)$. If $p-1 \nmid 2e_\BaseField$, then $E$ has potentially good reduction. Moreover, if $p-1 \nmid 6e_\BaseField$ and $p-1 \nmid 8e_\BaseField$, then $E$ has potentially supersingular reduction.
\end{proposition}

\begin{proof}
    The first statement is proved as in \cite[Proposition 4.1]{furio24}, noting that the argument given in \textit{loc.~cit.}~is in fact carried out in $p$-adic fields.
    Suppose now by contradiction that $E$ has potentially ordinary reduction and let $\GoodReductionExtension/\BaseField$ be a totally ramified extension of degree $e \in \{1,2,3,4,6\}$ over which $E$ acquires good reduction. Using \Cref{lemma: quadratic twist acquires good reduction over a small extension} we can assume $e \leq 4$ (note that taking a quadratic twist does not change the potential reduction type).
    Since the restriction of the mod $p$ cyclotomic character to $\operatorname{Gal}(\overline{\mathbb{Q}}_p/\GoodReductionExtension)$ has order $\frac{p-1}{\gcd(e_\BaseField e,p-1)}$, by \cite[Lemma 4.3]{furio24} we know that there is an element of $\NRep{p}$ with eigenvalues $1$ and $k$, where $k \in \F_p^\times$ has order $\frac{p-1}{\gcd(e_\BaseField e,p-1)}$. \Cref{rmk: basic properties Cns}(iii) gives $k=\pm 1$, and so $\frac{p-1}{\gcd(e_\BaseField e,p-1)} \mid 2$. This implies that $p-1 \mid 2e_\BaseField e$, which contradicts the hypothesis.
\end{proof}

For applications, we need to know that $E$ does not admit a canonical subgroup of order $p$ in order to apply certain results in the literature. This is however the extent of our use of canonical subgroups, so we choose to omit the definition and instead refer the interested reader to e.g.~\cite[Definition 3.5]{furiolombardo23}. The following result is a $p$-adic version of \cite[Theorem 4.5]{furio24}, which is in turn a generalisation of \cite[Theorem 3.11]{furiolombardo23} and proves the non-existence of canonical subgroups in the situation of interest.

\begin{theorem}\label{thm: canonicalsbg}
Let $E/\BaseField$ be an elliptic curve with potentially good reduction such that $\operatorname{Im}\rho_{E,p} \subseteq C_{ns}^+(p)$.
	\begin{enumerate}
		\item If $p>e_\BaseField e+1$ and $p \ne 2e_\BaseField e+1$, then $E$ does not have a canonical subgroup of order $p$.
		\item If $E$ has potentially good supersingular reduction and $p > \max\{e_\BaseField e-2,2\}$, then $E$ does not have a canonical subgroup of order $p$.
	\end{enumerate}
\end{theorem}
	
\begin{proof}
	The proof is the same as that of \cite[Theorem 4.5]{furio24}, noting that \cite[Theorem 4.6]{smith23} and \cite[Theorem 3.10]{furiolombardo23} are stated for number fields, but their proofs hold for $p$-adic fields.
\end{proof}

We now specialise to elliptic curves over $\Q_p$. The previous general results immediately imply:
\begin{corollary}\label{cor: supersingular and no canonical subgroup}
    Let $p > 7$ be a prime number and $E/\mathbb{Q}_p$ be an elliptic curve. If $\operatorname{Im}\rho_{E,p} \subseteq C_{ns}^+(p)$, then $E$ has potentially good supersingular reduction and does not have a canonical subgroup of order $p$.
\end{corollary}

\begin{proof}
    Combine \Cref{prop: potentially supersingular} and \Cref{thm: canonicalsbg}(ii) in the case $e_\BaseField=1$.
\end{proof}

\begin{proposition}\label{prop: e divides p plus one}
    Let $p>3$ be a prime and $E/\Q_p$ be an elliptic curve with semistability defect $e \in \{3,4,6\}$. 
    Then $E$ has potentially good reduction and $j(E) \equiv 0 \bmod{p}$ (resp.~$j(E) \equiv 1728 \bmod{p}$) if $e \in \{3,6\}$ (resp.~$e=4$).
    Moreover, the reduction is potentially supersingular if and only if $e \mid p+1$. In particular, for every prime $p>7$ and every elliptic curve $E/\Q_p$ such that $\operatorname{Im}\rho_{E,p} \subseteq C_{ns}^+(p)$, the semistability defect $e$ of $E$ divides $p+1$.
\end{proposition}

\begin{proof}
    It is a classical result that $e \geq 3$ implies that $E$ has potentially good reduction (see e.g. \cite[Theorem C.14.1(d)]{MR2514094}, which shows that if the reduction is potentially multiplicative, then semistability is achieved over an extension of degree at most $2$). The reduction is potentially supersingular if and only if $e \mid p+1$ by \cite[Remarque 2 on p.~112]{Volkov}.
    To conclude, it suffices to notice that by \Cref{cor: supersingular and no canonical subgroup} the curve $E$ has potentially good supersingular reduction whenever $p>7$ and $\operatorname{Im}\rho_{E,p} \subseteq C_{ns}^+(p)$.
\end{proof}

\begin{lemma}\label{lemma: G(p) is not in the Cartan}
Let $p>7$ be a prime number and $E/\Q_p$ be an elliptic curve such that $\NRep{p} \subseteq C_{ns}^+(p)$. Then 
    $\NRep{p} \not\subseteq C_{ns}(p)$ and $\rho_{E, p}(I_{\Q_p})$ is a cyclic group of order divisible by $\frac{p^2-1}{\gcd(e, 3)}$. If $p \equiv 1 \bmod 3$, then $\rho_{E, p}(I_{\Q_p})$ has order $p^2-1$.
\end{lemma}

\begin{proof}
By \Cref{cor: supersingular and no canonical subgroup}, $E$ has potentially good supersingular reduction. The Galois representation $E[p]$ is described explicitly in \cite[\S B.2.2]{Volkov} for $e \in \{1,2\}$ and in \cite[\S B.3.2.3]{Volkov} for $e \in \{3,4,6\}$. In the latter situation, there are several subcases depending on a parameter $v(\alpha)$. In the notation of \cite[p.~132]{Volkov}, the case $v(\alpha)=1$ cannot happen, because the image would not be contained in $C_{ns}^+(p)$. We then read from \cite{Volkov} that the representation is non-abelian -- hence the image is not contained in $C_{ns}(p)$ -- and that the restriction of $E[p] \otimes \overline{\mathbb{F}}_p$ to $I_{\Q_p}$ factors via powers of the fundamental character of level $2$ with exponents $\{1 - \frac{p^2-1}{e}, p+\frac{p^2-1}{e}\}$ (if $e \in \{1,2\}$, or $e \in \{3,4,6\}$ and $v(\alpha) \leq 0$) or $\{1 + \frac{p^2-1}{e}, p-\frac{p^2-1}{e}\}$ (otherwise). Using $e \in \{1,2,3,4,6\}$, one checks immediately that $\gcd\left( 1 \pm \frac{p^2-1}{e}, p^2-1 \right)=\gcd( 1 \pm \frac{p^2-1}{e},\mp e) \mid \gcd(e,3)$ and $\gcd\left( p \pm \frac{p^2-1}{e}, p^2-1 \right) = \gcd(p \pm \frac{p^2-1}{e}, \mp ep) \mid \gcd(e,3)$, so both characters have order divisible by $\frac{p^2-1}{\gcd(e, 3)}$. When $p \equiv 1 \bmod 3$, at least one of the exponents is also prime to $3$ (their sum is $p+1 \equiv 2 \bmod 3$), so one of the two characters has order $p^2-1$.
\end{proof}

\subsection{N-Cartan lifts}
We review some results on subgroups of $\op{GL}_2(\mathbb{Z}_p)$ that were proved in \cite{furio24}. In particular, we focus on specific subgroups that we call \textit{N-Cartan lifts}, which represent the possible images of the $p$-adic Galois representations attached to elliptic curves over $\Q_p$ in the non-split Cartan case.
\begin{definition}
	Let $p$ be an odd prime and let $G \subseteq \GL_2(\Z_p)$ be a subgroup. Following \cite[Definition 3.6]{furio24}, we say that $G$ is a \textit{non-split N-Cartan lift} if it satisfies the following properties:
	\begin{itemize}
		\item $G$ is closed;
		\item $\det(G) = \Z_p^\times$;
		\item $G(p)$ $\subseteq C_{ns}^+(p)$, but $G(p) \not\subseteq C_{ns}(p)$;
		\item there exists $g \in G(p) \cap C_{ns}(p)$ which is not a scalar matrix.
	\end{itemize}
\end{definition}

We have the following classification result on non-split N-Cartan lifts (\cite[Theorem 3.14]{furio24}). {Note that, given a subgroup $G$ of $\GL_2(\Z_p)$ and a positive integer $n$, the condition $G \supseteq \operatorname{Id} + p^n M_{2}(\Z_p)$ is equivalent to the fact that $G$ is the inverse image in $\GL_2(\Z_p)$ of its projection $G(p^n) \subseteq \GL_2(\Z/p^n\Z)$.}

\begin{theorem}\label{thm:cartantower}
	Let $G \subseteq \GL_2(\Z_p)$ be a non-split N-Cartan lift satisfying all of the following conditions:
    \begin{itemize}
        \item $|G_1/G_2| \ge p^2$;
        \item there exists an element in  $G(p) \cap C_{ns}(p)$ whose image in $\operatorname{PGL}_2(\F_p)$ has order greater than $2$;
        \item $G \supset (1+p\Z_p) \op{Id}$.
    \end{itemize}
    Then one of the following holds:
	\begin{itemize}
        \item $G \subseteq C_{ns}^+$ and $[C_{ns}^+ : G] = [C_{ns}^+(p) : G(p)]$;
		\item There exists $n \ge 1$ such that $G \supseteq \op{Id} + p^n M_{2}(\Z_p)$ and $G(p^n) \subseteq C_{ns}^+(p^n)$, with $[C_{ns}^+(p^n) : G(p^n)] = [C_{ns}^+(p) : G(p)]$;
		\item $G \supseteq \op{Id} + p^2 M_{2}(\Z_p),$ $G(p^2) \subseteq G_{ns}^\#(p^2)$, and $G(p^2) \cong G(p) \ltimes V$, with $V$ as in \Cref{def: ciuffetto}.
	\end{itemize}
\end{theorem}

It is shown in \cite{furio24} that for a non-CM elliptic curve $E/\Q$ such that $\NRep{p} \subseteq C_{ns}^+(p)$ for an odd prime $p$, the image $\operatorname{Im}\rho_{E,p^\infty}$ of the $p$-adic Galois representation is a non-split N-Cartan lift. We now extend this result to elliptic curves defined over $\Q_p$ and show that the groups $\operatorname{Im}\rho_{E,p^\infty}$ satisfy the conditions of \Cref{thm:cartantower}.

\begin{proposition}\label{prop: N-Cartan lift}
    Let $p>7$ be a prime and let $E/\Q_p$ be an elliptic curve such that $\NRep{p} \subseteq C_{ns}^+(p)$. The group $G:=\pnRep{p}{\infty}$ is a non-split N-Cartan lift.
\end{proposition}

\begin{proof}
    We have $\det(G) = \Z_p^\times$ since $\det \rho_{E,p^\infty}$ is the $p$-adic cyclotomic character.
    By Lemma \ref{lemma: G(p) is not in the Cartan} we know that $G(p) \not\subseteq C_{ns}(p)$. 
    The group $G$ is a continuous image of the profinite (hence compact) group $\operatorname{Gal}(\overline{\mathbb{Q}}_p/\mathbb{Q}_p)$, so it is compact. As the topology of $\operatorname{GL}_2(\mathbb{Z}_p)$ is Hausdorff, every compact subspace is closed, so $G$ is closed.
	Finally, we need to show that $G(p) \cap C_{ns}(p)$ contains an element which is not a scalar matrix. By \Cref{rmk: basic properties Cns}(ii), every element in $C_{ns}^+(p) \setminus C_{ns}(p)$ has order dividing $2(p-1)$, and the same holds for scalar matrices. Suppose by contradiction that $G(p) \cap C_{ns}(p)$ consists of multiples of the identity. In particular, every element of $G(p)$ has order dividing $2(p-1)$.
    By \Cref{lemma: G(p) is not in the Cartan}, the group $G(p)$ contains an element of order $\frac{p^2-1}{3}$. However, this is impossible because $\frac{p^2-1}{3} \nmid 2(p-1)$ for $p>7$.
\end{proof}

Our next theorem classifies the possible Galois images subject to the group-theoretic constraints described above. We will show later that case (iii), although compatible with these abstract constraints, cannot occur for elliptic curves over $\Q_p$.

\begin{theorem}\label{thm:ellipticcartantower}
Let $p>7$ be a prime and let $E/\Q_p$ be an elliptic curve such that $\NRep{p} \subseteq C_{ns}^+(p)$. Set $G:=\operatorname{Im}\rho_{E,p^\infty}$ and let $e$ be the semistability defect of $E/\Q_p$. One of the following holds:
	\begin{enumerate}
		\item $G \subseteq C_{ns}^+$ and $[C_{ns}^+ : G] \in \{1,3\}$, with $G=C_{ns}^+$ if $p \equiv 1 \bmod 3$ or $e \in \{1,2,4\}$;
		\item there exists $n \ge 1$ such that $G \supseteq \op{Id} + p^n M_{2}(\Z_p)$ and $G(p^n) \subseteq C_{ns}^+(p^n)$, with 
        \[
        [C_{ns}^+(p^n) : G(p^n)] =[C_{ns}^+(p) : G(p)] \in \{1,3\}
        \]
        and $G(p^n) = C_{ns}^+(p^n)$ if $p \equiv 1 \bmod 3$ or $e \in \{1,2,4\}$;
		\item $G \supseteq \op{Id} + p^2 M_{2}(\Z_p)$ and $G(p^2) \subseteq G_{ns}^\#(p^2)$, with $[G_{ns}^\#(p^2) : G(p^2)] \in \{1,3\}$ and $G(p^2) = G_{ns}^\#(p^2)$ if $p \equiv 1 \bmod 3$ or $e \in \{1,2,4\}$.
	\end{enumerate}
\end{theorem}

\begin{remark}\label{rmk: if G grows quadratically it is a nonsplit Cartan}
With the notation of \Cref{thm:ellipticcartantower}, note that
\[
[G(p^{n+1}):G(p^n)]=[\Qp(E[p^{n+1}]):\Qp(E[p^n])] = \begin{cases}
    p^2 \,\, \text{ for all } n \geq 1 & \quad \text{in case (i)}, \\
    p^4 \,\, \text{ for } n\gg 0 & \quad \text{in cases (ii) and (iii).}
\end{cases}
\]

\noindent As a consequence, if under the hypotheses of the theorem we assume furthermore that $[G(p^{n+1}) : G(p^n)] = p^2$ for all sufficiently large $n$, or even just $\#G(p^{n}) \leq Cp^{2n}$ for some $C \in \mathbb{R}$, then $G$ is a subgroup of $C_{ns}^+$ of index $1$ or $3$ (and $G=C_{ns}^+$ if $p \equiv 1 \bmod 3$ or $e \in \{1,2,4\}$). Similarly, if the $p$-adic valuation of $\#G(p^2)$ is $3$, we must be in case (iii) and $G(p^2)$ is a subgroup of $G_{ns}^\#(p^2)$ of index dividing $3$.
\end{remark}

\begin{proof}[Proof of \Cref{thm:ellipticcartantower}]
    We show that $G$ satisfies the hypotheses of \Cref{thm:cartantower}.
    First, $G$ is a non-split N-Cartan lift by \Cref{prop: N-Cartan lift}.
    Next, we show that the image of the inertia subgroup $I_{\Qp}$ via $\rho_{E,p^2}$ has order divisible by $\frac{p^2(p^2-1)}{3}$, and by $p^2(p^2-1)$ when $p \equiv 1 \bmod 3$.
    By \Cref{thm: canonicalsbg} we know that $E$ does not have a canonical subgroup of order $p$. In particular, by \cite[Theorem 4.6]{smith23} -- which is stated over number fields, but proved over $p$-adic fields -- we know that the group $\rho_{E,p^2}(I_K)$ has order divisible by $\frac{p^2(p^2-1)}{e}$. On the other hand, by \Cref{lemma: G(p) is not in the Cartan} we know that $\rho_{E,p}(I_{\Qp})$, and hence a fortiori 
    $\rho_{E,p^2}(I_{\Qp})$, contains an element of order $\frac{p^2-1}{3}$ (or order $p^2-1$, in case $p \equiv 1 \bmod 3$ or $e \in \{1,2,4\}$), and so $\rho_{E,p^2}(I_{\Qp})$ has order divisible by $\frac{p^2(p^2-1)}{3}$ (or $p^2(p^2-1)$ respectively).
    
    Since $p \nmid \#\NRep{p}$, the image of $I_{\Qp}$ via $\rho_{E,p}$ is cyclic, hence it contains an element $g$ of order $\frac{p^2-1}{3}$. Elements in $C_{ns}^+(p) \setminus C_{ns}(p)$ have order $2$ in $\operatorname{PGL}_2(\F_p)$ (they become scalar matrices when squared, see \Cref{rmk: basic properties Cns}(ii)), and so every element of order greater than $2$ in $\operatorname{PGL}_2(\F_p)$ must be contained in $C_{ns}(p)$. However, $g$ has order at least $\frac{p+1}{(3,p+1)}$ in $\operatorname{PGL}_2(\F_p)$, which is greater than $2$ for $p>7$. Moreover, since $p^2 \mid \#\rho_{E,p^2}(I_{\Qp})$ and $p \nmid \#\NRep{p}$, we know that $p^2$ divides the cardinality of $\ker(G(p^2) \to G(p))$. This implies $|G_1/G_2| \ge p^2$.
    
    To check the last hypothesis of \Cref{thm:cartantower}, it suffices to notice that \cite[Corollary 3.7]{lombardotronto22} shows that $\Z_p^\times \cdot \op{Id} \subseteq G$, because $G(p)$ contains $C_{ns}(p)^3$ (or $C_{ns}(p)$, when $p \equiv 1 \bmod 3$ or $e \in \{1,2,4\}$) and an element of $C_{ns}^+(p) \setminus C_{ns}(p)$ (since $G$ is an N-Cartan lift).
\end{proof}

\subsection{Description of the Galois image in special cases}
We assume again that $E/\Q_p$ is an elliptic curve such that $\operatorname{Im} \rho_{E,p}$ is contained in $C_{ns}^+(p)$. We describe the image of the full $p$-adic representation attached to $E$ in two simple special cases: when $E$ has potential complex multiplication and when it acquires good reduction over an at most quadratic extension of $\Q_p$. First, however, we show that when $E[p]$ is an irreducible Galois module, the integral $p$-adic Tate module of $E$ is determined by its rational counterpart.

In the next statement, by a \textit{Galois-stable lattice} in $V_pE$ we mean a (necessarily free) $\Z_p$-submodule $T$ of $V_pE$ such that $\rho_{E, p^\infty}(\sigma)(T) \subseteq T$ for all $\sigma \in \op{Gal}(\Qpbar/\Qp)$ and $T \otimes_{\Z_p} \Q_p= V_pE$.

\begin{lemma}\label{lemma: any lattice will do}
Let $E/\Q_p$ be an elliptic curve and suppose that $E[p]$ is an irreducible $\op{Gal}(\Qpbar/\Qp)$-module. Then:
\begin{enumerate}
    \item every Galois-stable lattice $T$ in $V_pE$ is homothetic to $T_pE$, and hence in particular isomorphic to $T_pE$ as a $\mathbb{Z}_p[\op{Gal}(\Qpbar/\Qp)]$-module;
    \item given any $T$ as above, $T/p^kT$ is isomorphic to $E[p^k]$ as a $\op{Gal}(\Qpbar/\Qp)$-module for all $k>0$;
    \item if $E'/\mathbb{Q}_p$ is another elliptic curve with $V_pE \cong V_pE'$ as Galois modules, then $T_pE \cong T_pE'$ as Galois modules.
\end{enumerate}
\end{lemma}

\begin{proof}
Let $T_1=T_pE \subseteq V_pE$ be the Tate module and let $T_2$ be any other Galois-stable lattice. All $\Z_p$-lattices in $V_pE \cong \Q_p^2$ are commensurable, so there exists $m \in \mathbb{Z}$ such that $p^m T_2 \subseteq T_1$. Choose $m \in \mathbb{Z}$ minimal with this property, so that $p^m T_2 \subseteq T_1$ but $p^{m-1}T_2 \not \subseteq T_1$. Consider the submodule $p^{m}T_2$ of $T_1$ and its image $p^mT_2/pT_1$ inside $T_1/pT_1 \cong E[p]$. Since $T_2$ is Galois-stable, so is $p^mT_2/pT_1$, but $E[p]$ is irreducible, so $p^mT_2/pT_1$ is isomorphic to either $\{0\}$ or $E[p] = T_1/pT_1$. In the former case, we have $p^mT_2 \subseteq pT_1$, contradicting the choice of $m$. In the latter, we obtain $p^mT_2/pT_1 = T_1/pT_1$, hence $p^{m}T_2 + pT_1 = T_1$. Nakayama's lemma then implies $p^mT_2 =T_1$, as desired.

Statements (i) and (ii) follow immediately. As for (iii), fix a Galois-equivariant isomorphism $\varphi : V_pE' \to V_pE$ and define $T := \varphi(T_pE')$. Then $T$ is a Galois-stable lattice in $V_pE$, hence it is isomorphic to $T_pE$ by (i).
\end{proof}

\begin{lemma}\label{lemma: description in potential CM case}
    Let $p>3$ be a prime and $E/\Q_p$ be an elliptic curve with potential complex multiplication. Let $\mathcal{O}:=\op{End}_{\Qpbar}(E)$ be the ring of endomorphisms of $E_{\overline{\Q}_p}$.
    \begin{enumerate}
        \item We have $\# \op{Gal}(\Q_p(E[p^k])/\Q_p) \mid 2p^{2(k-1)} \cdot \#(\mathcal{O} \otimes \Z/p\Z)^\times$. In particular, if $j(E) \in \{0, 1728\}$, then the $p$-adic valuation of $\# \op{Gal}(\Q_p(E[p^k])/\Q_p)$ is at most $2(k-1)$.
        \item If $p>7$ and
        $\operatorname{Im} \rho_{E, p} \subseteq C_{ns}^+(p)$, then $\operatorname{Im} \rho_{E, p^\infty}$ is a subgroup of $C_{ns}^+$ of index dividing $3$.
    \end{enumerate}
\end{lemma}

\begin{proof}
    Let $F$ be a quadratic extension of $\Qp$ containing the CM field $\mathcal{O} \otimes \Q$. Note that $F$ is a field of definition for the endomorphisms of $E_{\Qpbar}$. Since the action of $\op{Gal}(\overline{F}/F)$ must be compatible with the action of $\mathcal{O}$, we have that $\op{Gal}(F(E[p^k])/F) \subseteq (\mathcal{O} \otimes \mathbb{Z}/p^k\mathbb{Z})^{\times}$ \cite[p.~502, Corollary 2]{MR236190}, hence $\# \op{Gal}(F(E[p^k])/F) \mid \# (\mathcal{O} \otimes \mathbb{Z}/p^k\mathbb{Z})^{\times} = \# (\mathcal{O} \otimes \mathbb{Z}/p\mathbb{Z})^{\times} \cdot p^{2(k-1)}$ and $\# \op{Gal}(\Q_p(E[p^k])/\Q_p) \mid 2p^{2(k-1)} \cdot \#(\mathcal{O} \otimes \mathbb{Z}/p\mathbb{Z})^{\times}$. To complete the proof of (i), note that in the cases $j \in \{0, 1728\}$ and $p>3$ we have $p \nmid \#(\mathcal{O} \otimes \mathbb{Z}/p\mathbb{Z})^\times$. Now (ii) follows from \Cref{rmk: if G grows quadratically it is a nonsplit Cartan}, which applies since $p>7$.
\end{proof}

\begin{lemma}\label{lemma: p-adic image in the twisted crystalline case}
Assume $p>7$ and let $E/\mathbb{Q}_p$ be an elliptic curve such that $\operatorname{Im} \rho_{E,p} \subseteq C_{ns}^+(p)$. If $E$ has good reduction over $\Q_p$ or acquires it over a quadratic extension, then $\operatorname{Im} \rho_{E,p^\infty}$ is equal to $C_{ns}^+$.
\end{lemma}

\begin{proof}
    First, we suppose $E$ has good reduction. Let $\overline{E}/\mathbb{F}_p$ be the reduction of $E$ and let $\overline{j}$ be its $j$-invariant. If $\overline{j} \neq 0, 1728$, \cite[Lemma 3.3(i)]{HabeggerSmallHeight} shows that for all $k \geq 1$ we have $[\mathbb{Q}_{p^2}(E[p^k]) : \mathbb{Q}_{p^2}] = (p^2-1)p^{2(k-1)}$. Since $\operatorname{Im} \rho_{E, p^\infty}$ is a non-split N-Cartan lift by \Cref{prop: N-Cartan lift}, it follows from the classification in \Cref{thm:ellipticcartantower} that $G=C_{ns}^+$ (see also \Cref{rmk: if G grows quadratically it is a nonsplit Cartan}). If instead $\overline{j} =0$ (resp.~1728), let $E'/\mathbb{Q}_p$ be an elliptic curve with good reduction and $j$-invariant $0$, e.g.~$y^2=x^3+1$ (resp.~with $j$-invariant $1728$, such as $y^2=x^3+x$). By \Cref{cor: supersingular and no canonical subgroup} we know that $E$ has supersingular reduction, so $\overline{j}$ is a supersingular $j$-invariant. We read from \cite[Proposition 2.7]{VolkovArticle} that there is a unique isomorphism class of rational Tate modules attached to supersingular elliptic curves with good reduction over $\mathbb{Q}_p$: in the notation of \cite[\S 2.2.1]{VolkovArticle}, one has $D^*_{pst}(V_pE) \cong \mathbf{D}_c(\mathbf{1}; \mathbf{0})$. It follows $V_pE \cong V_pE'$ as Galois modules, and moreover $T_pE \cong T_pE'$ by \cite[\S B.2.2]{Volkov}. The result then follows from the fact that in the potential CM case the Galois action on the $p$-adic Tate module has the stated property (\Cref{lemma: description in potential CM case}; note that the index cannot be 3 by \Cref{thm:ellipticcartantower}(i) since $e \in \{1,2\}$).

    Finally, if $E$ has bad reduction, then a quadratic twist $E'$ of $E$ has good reduction over $\Q_p$. The claim holds for $E'$ by what we already proved. Note that the degrees $[\mathbb{Q}_p(E[p^{k+1}]) : \mathbb{Q}_p(E[p^{k}])]$ and $[\mathbb{Q}_p(E'[p^{k+1}]) : \mathbb{Q}_p(E'[p^{k}])]$ are equal for all $k \geq 1$, because they can differ at most by a factor of $2$, but they are all powers of $p$. As for $E'$ we have $[\mathbb{Q}_p(E'[p^{k+1}]) : \mathbb{Q}_p(E'[p^{k}])]=p^2$ for all $k$, the same holds for $E$ and the claim follows from \Cref{rmk: if G grows quadratically it is a nonsplit Cartan}.
\end{proof}

\section{Review of Volkov's results}\label{sect: Volkov results}

We review the results in Appendix B of Volkov's PhD thesis \cite{Volkov}, which gives a description of the rational Tate modules of elliptic curves with semistability defect $e \in \{3, 4, 6\}$ and potentially supersingular reduction. We will largely omit proofs; all the details can be found in \cite{Volkov}. The results we need from \cite{Volkov} all require $p > 3$, so we make this assumption throughout this section.

\begin{notation*}
    For the remainder of the paper, we use $v$ to denote the $p$-adic valuation on $\C_p$, normalised such that $v(p)=1$.
\end{notation*}

\subsection{The ring $R$}
We denote by $R$ the perfection of $\mathcal{O}_{\mathbb{C}_p}/(p)$:
\[
R = \varprojlim \frac{\mathcal{O}_{\mathbb{C}_p}}{(p)},
\]
where $\mathcal{O}_{\mathbb{C}_p}$ is the valuation ring of $\mathbb{C}_p$ and the transition maps are given by $x \mapsto x^p$. 
By definition, an element $x \in R$ is a sequence $(x_n)_{n \in \mathbb{N}}$ with each $x_n \in \mathcal{O}_{\mathbb{C}_p}/(p)$ and $x_{n+1}^p=x_n$. Fix, for every $n \in \mathbb{N}$, an arbitrary lift $\hat{x}_n \in \mathcal{O}_{\mathbb{C}_p}$. For every $n \geq 0$, the sequence $\hat{x}_{n+m}^{p^m}$ converges in $\mathcal{O}_{\mathbb{C}_p}$ to an element $x^{(n)}$ independent of the choice of lifts $\hat{x}_{n}$:
\[
x^{(n)} := \lim_{m \to \infty} \hat{x}_{n+m}^{p^m}.
\]
The map $x \mapsto (x^{(0)}, x^{(1)}, \ldots)$ is a bijection from $R$ to sequences $(x^{(0)}, x^{(1)}, \ldots)$ in $\mathcal{O}_{\mathbb{C}_p}$ satisfying  $(x^{(i+1)})^p = x^{(i)}$ for all $i \geq 0$; see \cite[p.~104]{Volkov} for more details. We use this interpretation to define a valuation $v_R$ on $R$ given by
\[
v_R(x)=v_R((x^{(0)}, x^{(1)}, \ldots))=v(x^{(0)}).
\]
The elements of strictly positive valuation form the maximal ideal $\mathfrak{M}_R$ of $R$. 

\subsection{The ring of Witt bivectors $BW(R)$}
We let $BW(R)$ denote the ring of Witt bivectors over $R$ (see \cite[p.~121]{Volkov} or \cite[\S 6.3]{Fontaine82}). Elements of $BW(R)$ are infinite sequences 
        \[
        \underline{a}=(a_n)_{n \in \mathbb{Z}}, \,\, a_n \in R \text{ such that }\exists \, n_0 \in \mathbb{Z}, \; \exists \, \varepsilon>0 \text{ with } v_R(a_n) \geq \varepsilon \text{ for }n \leq n_0.
        \]
        The Witt bivectors are endowed with two operators, $\varphi$ and $\text{V}$ (Frobenius and Verschiebung), given by
        \[
        \varphi( (a_n)_{n \in \mathbb{Z}}) = (a_n^p)_{n \in \mathbb{Z}}
\quad \text{and}\quad
        \text{V}((a_n)_{n \in \mathbb{Z}}) = (a_{n-1})_{n \in \mathbb{Z}}.
        \]
        These operators satisfy the relation $\varphi \text{V} = \text{V} \varphi = [p]$, multiplication by $p$ in $BW(R)$. The ring $BW(R)$ admits a natural embedding in $B_{\operatorname{dR}}$, from which it inherits a filtration \cite[top of p.~122]{Volkov}.

We give some more details about addition in $BW(R)$. Let $\underline{a} =(a_n)_{n \in \mathbb{Z}}$ and $\underline{b}=(b_n)_{n \in \mathbb{Z}}$ be elements of $BW(R)$. We will be especially interested in the $0$-th component of the sum $\underline{d} := \underline{a} + \underline{b}$. By definition, this is given by
\[
d_0 = \lim_{m \to \infty} S_m(a_{-m},\ldots,a_0;b_{-m},\ldots,b_0),
\]
where $S_m \in \mathbb{Z}[X_0,\ldots;Y_0,\ldots]$ are the addition Witt polynomials. Recall that $S_0 = X_0 + Y_0$.
As in the proof of \cite[Lemme 2 on p.~123]{Volkov}, we define the polynomials with integer coefficients
\begin{equation*}
\begin{aligned}
Q_r^{(m+1)} & = Q_r^{(m+1)}(X_0,\ldots,X_r;Y_0,\ldots,Y_r) \\
& := \frac{1}{p^{m+1-r}}(S_r(X_0^p,\ldots,X_r^p;Y_0^p,\ldots,Y_r^p)^{p^{m-r}} - (S_r(X_0,\ldots,X_r;Y_0,\ldots,Y_r)^p)^{p^{m-r}})
\end{aligned}
\end{equation*}
and have the recurrence relation $S_{m+1} = X_{m+1} + Y_{m+1} + \sum_{0 \leq r \leq m} Q_r^{(m+1)}.$ For our purposes, we only need the following simple property of the addition polynomials, which will be applied later. 

\begin{lemma}\label{lemma: addition polynomials in one variable}
For every $m \geq 0$ we have
    \[
S_m((-1)^mX_m^{p^m} , \ldots, X_m^{p^2}, -X_m^p, X_m; (-1)^{m+1}Y_m^{p^m} , \ldots, -Y_m^{p^2}, Y_m^p, -Y_m) = (X_m-Y_m)(1 + g_m(X_m, Y_m))
\]
for some polynomial $g_m$ in the ideal $(X_m, Y_m)^p$ of $\mathbb{Z}[X_m, Y_m]$.
\end{lemma}

\begin{proof}
    Straightforward induction using the recurrence relation.
\end{proof}

\subsection{The filtered $(\varphi, G_{K/\mathbb{Q}_p})$-modules $D_\alpha$}\label{sect: filtred phi modules}

We consider elliptic curves $E/\mathbb{Q}_p$ with semistability defect $e \geq 3$, which necessarily have potentially good reduction.  Fix $e \in \{3,4,6\}$ and assume $p \equiv -1 \bmod{e}$. We write $K=\Qp(\pi_e,\zeta_e)$ for the splitting field of $x^e+p$, where $\pi_e=\sqrt[e]{-p}$ is a fixed root of $x^e+p$ and $\zeta_e$ is a fixed primitive $e$-th root of unity. It is a standard fact that any $E/\Q_p$ with semistability defect $e$ acquires semistable (hence good) reduction over $\Qp(\pi_e)$, and therefore over its Galois closure $K$.

We denote the Galois group of $K$ over $\Qp$ by $G_{K/\Qp}$; it is generated by a Frobenius $\omega$ and inertia element $\tau_e$ with actions on $K$ defined by
\[
\omega(\zeta_e)=\zeta_e^{-1}, \quad \omega(\pi_e)=\pi_e \qquad\quad \text{and} \qquad\quad
\tau_e(\zeta_e)=\zeta_e, \quad \tau_e(\pi_e)=\zeta_e \pi_e.
\]

By results of Volkov that we now recall, every elliptic curve $E/\mathbb{Q}_p$ with semistability defect $e$ corresponds to a filtered $(\varphi, G_{K/\mathbb{Q}_p})$-module, and the filtered $(\varphi, G_{K/\mathbb{Q}_p})$-modules that can arise in this way are parametrised by $\alpha \in \mathbb{P}^1(\mathbb{Q}_p)$, as in the following definition. Throughout the paper, we use the convention that $\alpha^{-1}=0$ when $\alpha=\infty$.

\begin{definition}[{\cite[p.~125]{Volkov}}]\label{def: D alpha}
Let $e \in \{3, 4, 6\}$ and suppose $e \mid p+1$. For $\alpha \in \mathbb{P}^1(\mathbb{Q}_p)$ we let $D_\alpha$ be the filtered $(\varphi, G_{K/\mathbb{Q}_p})$-module given by a 2-dimensional $\mathbb{Q}_{p^2}$-vector space with basis $\Dbasis_1, \Dbasis_2$ endowed with the following additional structures:
    
  \begin{enumerate}
        \item a Frobenius $\varphi$, defined by $\varphi(\Dbasis_1) = \Dbasis_2$, $\varphi(\Dbasis_2) = -p\Dbasis_1$;
        \item a filtration on $(D_\alpha)_K := D_\alpha \otimes_{\mathbb{Q}_{p^2}} K$ defined by
        \[
        \operatorname{Fil}^0(D_\alpha)_K = (D_\alpha)_K, \qquad \,\, \operatorname{Fil}^1 (D_\alpha)_K = (\Dbasis_1 \otimes \pi_e^{-1} + \alpha^{-1} \cdot \Dbasis_2 \otimes \pi_e)K, \qquad \,\,\ \operatorname{Fil}^2(D_\alpha)_K = \{0\},
        \]
        where $\operatorname{Fil}^1(D_0)_K = (\Dbasis_2 \otimes \pi_e)K$ by convention;
        \item a semi-linear action of $G_{K/\Q_p}$,
        given on the basis by
        \[
        \omega(\Dbasis_1) = \Dbasis_1, \quad \omega(\Dbasis_2)=\Dbasis_2 \quad\qquad \text{and} \quad\qquad \tau_e(\Dbasis_1) = \zeta_e \cdot \Dbasis_1, \quad \tau_e(\Dbasis_2) = \zeta_e^{-1} \cdot \Dbasis_2.
        \]
    \end{enumerate}
\end{definition}

Since $E$ acquires good reduction over $K$, we can apply the contravariant functor $D^*_{\op{cris},K/\Q_p}$ (see \cite[p.~22]{Volkov} and \cite[5.3.7]{MR1293972}) to the $p$-adic Galois representation $V_pE$. Volkov shows \cite[pp.~111-112]{Volkov} that the resulting filtered $(\varphi, G_{K/\Q_p})$-module is of the form $D_\alpha$ for some suitable $\alpha \in \mathbb{P}^1(\Qp)$. Moreover, if $D^*_{\operatorname{cris}, K/\mathbb{Q}_p}(V_pE) \cong D_\alpha$, by \cite[p.~125]{Volkov} we can recover the $\mathbb{Q}_p[\operatorname{Gal}(\overline{\mathbb{Q}}_p/\mathbb{Q}_p)]$-module $V_pE$ as
    \begin{equation}\label{eq: definition of V alpha}
    V_pE \cong V_\alpha := \operatorname{Hom}_{(\varphi, G_K)\operatorname{-mf}}(D_\alpha, BW(R)),
    \end{equation}
where morphisms are in the category of $(\varphi, G_{K})$-filtered modules (that is, $\mathbb{Q}_{p^2}$-linear maps that respect Frobenius and the filtration, and are equivariant with respect to the action of $G_K := \op{Gal}(\overline{K}/K)$). Note that $V_\alpha$ is in an obvious way a $G_K$-module via the action
\[
(g \cdot u)(d) = g \cdot (u(d)) \quad \forall g \in G_K, u \in V_\alpha, d \in D_\alpha.
\]
This $G_K$-action can be extended to an action of the full Galois group $\op{Gal}(\overline{\Q}_p/\Q_p)$ as in \cite[p.~122]{Volkov}, namely, by letting $g \in \operatorname{Gal}(\overline{\Q}_p/\Q_p)$ act on $u \in V_\alpha$ by 
\begin{equation}\label{eq: twisted action}
(g \cdot u)(d) = g \cdot ( u ( [g^{-1}] \cdot d ) ),
\end{equation}
where $[g^{-1}]$ denotes the class of $g^{-1}$ in $G_{K/\Q_p}$. 

We summarise the above discussion in the following theorem.

\begin{theorem}[{\cite[Theorem 2.1 on p.~25 and p.~125]{Volkov}}]\label{thm: Volkov parametrisation}
Let $e \in \{3,4,6\}$ and let $p> 3$ be a prime with $e \mid p+1$.
\begin{enumerate}
    \item Let $E/\Q_p$ be an elliptic curve with semistability defect $e$. Then there exists $\alpha \in \mathbb{P}^1(\Q_p)$ such that $V_pE$ is isomorphic to $V_\alpha$ as a $\Qp[\op{Gal}(\Qpbar/\Qp)]$-module, where $V_\alpha$ is defined in \Cref{eq: definition of V alpha}.
    \item Conversely, for every $\alpha \in \mathbb{P}^1(\mathbb{Q}_p)$ there exists an elliptic curve $E/\Q_p$ with semistability defect $e$, potentially supersingular reduction, and $V_pE \cong V_\alpha$ as $\Qp[\op{Gal}(\Qpbar/\Qp)]$-modules.
\end{enumerate}
\end{theorem}

\begin{remark}
    Note that, under the assumptions of the theorem, $E$ has potentially good supersingular reduction by \Cref{prop: e divides p plus one}.
\end{remark}

\begin{remark}\label{rmk: alpha role}\phantom{~}
\begin{enumerate}
    \item Taking a ramified quadratic twist of $E$ replaces $\alpha$ by $\pm p^2\alpha^{-1}$, with a different choice of $e$ when $e \in \{3,6\}$  (see \cite[top of p.~25]{Volkov}).
    \item We have $v(\alpha)=1$ precisely when the valuation of the Hasse invariant of $E$ is strictly less than $\frac{p}{p+1}$ (\cite[p.~133]{Volkov}). This condition is in turn equivalent to the fact that $E$ has a canonical subgroup of order $p$ by \cite[Theorem 3.1]{MR447119}.
    \item The cases $\alpha=0,\infty$ correspond to $j$-invariants in $\{0,1728\}$ \cite[p.~26]{Volkov}.
\end{enumerate}
\end{remark}

We now make $V_\alpha$ more concrete. 

Note first that a $\Q_{p^2}$-linear map $u : D_\alpha \to BW(R)$ is uniquely determined by the image of $\Dbasis_1$ and $\Dbasis_2$. If it is furthermore $\varphi$-equivariant, then $u(\Dbasis_2) = u(\varphi(\Dbasis_1))=\varphi(u(\Dbasis_1))$, so $u$ is determined by $u(\Dbasis_1)$ alone.

\begin{definition}\label{def: map u subscript a}
Given $\underline{a} \in BW(R)$, we denote by $u_{\underline{a}}$ the unique $\Q_{p^2}$-linear and $\varphi$-equivariant map $u: D_\alpha \to BW(R)$ such that $u(\Dbasis_1)=\underline{a}$.
\end{definition}

\begin{lemma}\label{lem: V_alpha}
Let $u \in V_\alpha$.
\begin{enumerate}
    \item $u$ is of the form $u_{\underline{a}}$ for some $\underline{a} \in BW(R)$;
    \item If $u=u_{\underline{a}}$ then $\varphi^2\underline{a}=-p\underline{a}$;
    \item Let $\underline{a} \in BW(R)$. Then $\varphi^2\underline{a}=-p\underline{a}$ if and only if $\underline{a}=((-1)^na^{p^{-n}})_{n \in \mathbb{Z}}$ for some $a=(a^{(n)})_{n \geq 0} \in \mathfrak{M}_R$;
    \item Let $a=(a^{(n)})_{n \geq 0} \in \mathfrak{M}_R$ and set $\underline{a}=((-1)^na^{p^{-n}})_{n \in \mathbb{Z}} \in BW(R)$. The map $u_{\underline{a}} : D_\alpha \to BW(R)$ is a map of filtered modules if and only if $a=(a^{(n)})_{n \geq 0}$ solves the equation
    \begin{equation}\label{eq: condition star alpha}
        \sum_{n \in \mathbb{Z}} (-1)^n p^n\left( \pi_e^{-1} {a^{(0)}}^{p^{-2n}} + \alpha^{-1}\pi_e {a^{(0)}}^{p^{-2n+1}} \right) = 0 \in \mathbb{C}_p,
    \end{equation}
    where for $n \geq 0$ we have set ${a^{(0)}}^{p^{-n}}:=a^{(n)} \in \mathbb{C}_p$;
    \item $V_\alpha = \{ u_{\underline{a}} :  \underline{a}=((-1)^na^{p^{-n}})_{n \in \mathbb{Z}} \text{ for some } a{=}(a^{(n)})_{n \geq 0} \in \mathfrak{M}_R \text{ that solves } \Cref{eq: condition star alpha} \}$.
\end{enumerate}
\end{lemma}

\begin{proof}
By definition, elements of $V_\alpha$ are $\Q_{p^2}$-linear and $\varphi$-equivariant, so we have already proved (i). For (ii),  note $-pu(\Dbasis_1)=u(-p\Dbasis_1)=u(\varphi(\Dbasis_2))=\varphi(u(\Dbasis_2))=\varphi^2(u(\Dbasis_1))$.
(iii)-(v): See \cite[p.~125]{Volkov}, noting the rescaling in (iv).
\end{proof}

The upshot of this discussion is that, under the assumptions of \Cref{thm: Volkov parametrisation}, the elements of the rational Tate module $V_pE$ are in Galois-equivariant bijection with the solutions of \Cref{eq: condition star alpha}. In the next few sections, we exploit this description to compute explicitly the image of the mod $p^2$ Galois representations attached to potentially supersingular elliptic curves over $\Q_p$, and more generally to describe the elements of $E[p^k]$ in terms of roots of explicit polynomials.

\section{Description of $E[p^k]$ in terms of the roots of a polynomial}\label{sect: description of p2 torsion}

We now apply Volkov's results to give a simpler description of the $p^k$-torsion of $E$ when $E$ has semistability defect $e \in \{3,4,6\}$ and potentially supersingular reduction. In particular, this implies that $e \mid p+1$ (\Cref{prop: e divides p plus one}). We will continue to assume $p>3$ throughout in order to apply \Cref{thm: Volkov parametrisation}; in particular we have that $V_pE \cong V_{\alpha}$ for some $\alpha \in \mathbb{P}^1(\Qp)$, where $V_\alpha$ is defined in \Cref{eq: definition of V alpha}. We assume $v(\alpha^{-1}) \geq 0$: by convention, this includes $\alpha=\infty$. Recall that $\pi_e$ denotes a fixed element of $\overline{\Q}_p$ such that $\pi_e^e=-p$.

Fix $k>0$. We manipulate the series of \Cref{eq: condition star alpha} whose roots are of interest in order to derive the key polynomial we use. We start by writing
\begin{eqnarray*}
\sum_{n \in \Z} (-1)^np^n \left( \pi_e^{-1}x^{p^{-2n}} + \alpha^{-1}\pi_ex^{p^{-2n+1}} \right) &=& \frac{1}{\pi_e} \sum_{n \in \Z} (-1)^np^n \left( x^{p^{-2n}} + \alpha^{-1}\pi_e^2x^{p^{-2n+1}} \right).
\end{eqnarray*}

Consider the terms of the sum for $-k \leq n \leq 0$:
\[
\begin{aligned}
\sum_{n=-k}^{0} (-1)^np^n \bigm( x^{p^{-2n}} +  & \alpha^{-1}\pi_e^2x^{p^{-2n+1}} \bigm) 
\; = \frac{(-1)^k}{p^k} \sum_{n=0}^{k} (-1)^{n}p^{n} \left( x^{p^{2k-2n}} + \alpha^{-1}\pi_e^2x^{p^{2k-2n+1}} \right) \\
&= \frac{(-1)^k}{p^k} \left( \alpha^{-1}\pi_e^2x^{p^{2k+1}} + x^{p^{2k}} + \sum_{n=1}^{k} (-1)^{n}p^{n} \left( x^{p^{2k-2n}} + \alpha^{-1}\pi_e^2x^{p^{2k-2n+1}} \right) \right).
\end{aligned}
\]

\begin{definition}
Recall the field $K=\Q_p(\pi_e, \zeta_e)$, where $\pi_e^e =-p$. For $k \geq 1$, define:
\begin{equation}\label{eq: polynomial g}
    g_k(x) := x^{p^{2k}}+ \sum_{n=1}^{k} (-1)^np^n (x^{p^{2k-2n}} + \alpha^{-1}\pi_e^2x^{p^{2k+1-2n}})
\end{equation}
and
\[
\mathcal{R}_k := \{ \text{roots of } g_k \text{ in } \overline{K}\}.
\]
\end{definition}

It is clear from the derivation that our original series in \Cref{eq: condition star alpha} is equal to
\begin{equation}\label{eqn: gk to series}
\frac{1}{\pi_e} \left( \frac{(-1)^k}{p^k} (g_k(x)+\alpha^{-1}\pi_e^2x^{p^{2k+1}}) + \sum_{\substack{n \in \Z \\ n \neq 0,\ldots,-k}} (-1)^np^n \left( x^{p^{-2n}} + \alpha^{-1}\pi_e^2x^{p^{-2n+1}} \right) \right).
\end{equation}

For the benefit of the reader, we state here the main result of the section in the most self-contained fashion.

\begin{theorem}\label{thm: p2 torsion points correspond to roots of g}
Let $E/\Qp$ be an elliptic curve and let $k$ be a positive integer. Suppose that:
\begin{itemize}
    \item $E$ has potentially supersingular reduction;
    \item the semistability defect of $E$ is at least $3$;
    \item the deformation parameter $\alpha$ (see \Cref{thm: Volkov parametrisation}) satisfies $v(\alpha^{-1}) \geq 0$;
    \item $p > \max\{\sqrt{k}, k-v(\alpha^{-1}) -1, 3\}$.
\end{itemize}
Then:
\begin{enumerate}
    \item there is a $\op{Gal}(\overline{K}/K)$-equivariant bijection $\Phi_k$ between $E[p^k]$ and $\mathcal{R}_k$;
    \item the splitting field of $g_k(x)$ over $K$ is $K(E[p^k])$.
\end{enumerate}
\end{theorem}

\begin{remark}\label{rmk: canonical subgp}
    We briefly remark on the assumptions of the theorem. If $\op{Im} \rho_{E,p} \subseteq C_{ns}^+(p)$ and $p>7$, then the potentially supersingular hypothesis is automatic from \Cref{cor: supersingular and no canonical subgroup}.
    Moreover, one can always assume that $v(\alpha^{-1})\geq -1$ up to quadratic twisting $E$ (see \Cref{rmk: alpha role}(i)); the missing case of $v(\alpha^{-1})=-1$ is equivalent to the existence of a canonical subgroup, a scenario which will also be excluded in our application by \Cref{cor: supersingular and no canonical subgroup}.
\end{remark}

\begin{remark}\label{rmk: twisted Galois action of Phi}
    Explicitly, the first conclusion of the theorem means that there is a set-theoretic bijection $\Phi_k : E[p^k] \to \mathcal{R}_k$ such that for all $P \in E[p^k]$ and all $\sigma \in \op{Gal}(\overline{K}/K)$ we have
    \[
    \sigma(\Phi_k(P)) = \Phi_k(\sigma(P)),
    \]
    where the Galois action on the left is the natural action on the roots of the polynomial $g_k(x)$ and that on the right is the natural action on $E[p^k]$. We also note that the same map $\Phi_k$ extends to an isomorphism of $\Gal(\overline{\Q}_p/\Q_p)$-sets, once $\mathcal{R}_k$ is endowed with a suitable (twisted) Galois action. We describe it in \Cref{rmk: twisted Galois action 2}.
\end{remark}

This section is devoted to proving \Cref{thm: p2 torsion points correspond to roots of g}, which we eventually do in \S\ref{subsect: proof of iso}, after introducing the necessary objects and language required along with various properties of both $\mathcal{R}_k$ and $\Phi_k$. 
Before embarking on this expedition, we note the following corollary.

\begin{corollary}\label{cor: splitting field in limit case}
    Fix $e \in \{3,4,6\}$ and let $p>3$ be a prime such that $e \mid p+1$. Then there exists a CM elliptic curve $E/\Qp$ with $j(E) \in \{0,1728\}$, potentially supersingular reduction, semistability defect $e$, and $\alpha$-parameter (cf.~\Cref{thm: Volkov parametrisation}) $\alpha=\infty$. Moreover, if $p> \sqrt{k}$, then $K(E[p^k])$ is equal to the splitting field of
    \[
    g_k^{\infty}(x)= \sum_{n=0}^k (-1)^np^n x^{p^{2k-2n}} \in K[x].
    \]
\end{corollary}

\begin{proof}
    The existence of $E/\Q_p$ with semistability defect $e$ and $\alpha=\infty$ is guaranteed by \Cref{thm: Volkov parametrisation}(ii). Its $j$-invariant is $0$ or $1728$ by \Cref{rmk: alpha role}(iii) so $E$ has CM. The result now follows immediately from \Cref{thm: p2 torsion points correspond to roots of g}, recalling that $\alpha^{-1}=0$ has valuation $\infty$.
\end{proof}

\subsection{Lattices in $V_\alpha$ and a description of $E[p^k]$} \label{sect: lattices in Valpha}
Throughout this section, we let $e, p$ and $E/\Q_p$ be as in the statement of \Cref{thm: Volkov parametrisation} and let $\alpha \in \mathbb{P}^1(\Q_p)$ be such that $V_pE \cong V_\alpha$ as $\Q_p[\Gal(\overline{\Q}_p/\Q_p)]$-modules.
We now introduce certain Galois-stable lattices in $V_\alpha$ and show how they relate to a description of $E[p^k]$ as a $\op{Gal}(\overline{\Q}_p/K)$-module.
Recall from \Cref{lem: V_alpha}(i) the map $u_{\underline{a}}$ associated with $\underline{a} \in BW(R)$. 
\begin{definition}\label{def: Tc}
For $c>0$ we define the Galois stable sublattice of $V_\alpha$ given by
\[
T_c = \{ u_{\underline{a}} \in V_\alpha : \underline{a} = ((-1)^na^{p^{-n}})_{n \in \mathbb{Z}} \text{ with } v_R(a) \geq c \}.
\]
\end{definition}

We now define a family of Galois lattices $A_c$ in $\mathfrak{M}_R$ that admit Galois-equivariant bijections with the lattices $T_c$ (see \Cref{lemma: Tc or Ac ça m'est égal}).

\begin{definition}\label{def: A alpha and A c}
    We set
    \[
    A_\alpha := \{ a \in \mathfrak{M}_R : a=(a^{(n)})_{n \in \mathbb{N}} \,\text{ satisfies } \text{\Cref{eq: condition star alpha}} \} \qquad \text{and }  \qquad A_c := \{a \in A_\alpha : v_R(a) \geq c\}.
    \]
    For $a \in \mathfrak{M}_R$, we define $\mathbf{a}:=((-1)^na^{p^{-n}})_{n \in \mathbb{Z}} \in BW(R)$.
The natural bijection 
\[
\begin{array}{cccc}
\Psi: & V_\alpha & \to & A_\alpha \\
& u_{\mathbf{a}} & \mapsto & a
\end{array}
\]
provided by \Cref{lem: V_alpha}(v) enables us to equip $A_\alpha$ with the structure of a $\mathbb{Q}_p$-vector space as follows:
\begin{itemize}
    \item sum $a \oplus b := \Psi(\Psi^{-1}(a) + \Psi^{-1}(b))$ for all $a, b \in A_\alpha$;
    \item scalar product  $\lambda \cdot a := \Psi(\lambda \cdot \Psi^{-1}(a))$ for all $a \in A_\alpha$ and $\lambda \in \mathbb{Q}_p$.
\end{itemize}
\end{definition}

Operations in the vector space $A_\alpha$ can be made more explicit (see also \cite[p.~126]{Volkov}): we describe in particular sums and differences.
\begin{enumerate}
    \item Given $a, b \in A_{\alpha}$, we consider the corresponding maps $u_{\mathbf{a}}, u_{\mathbf{b}}$ and their natural sum $u_{\mathbf{a}} + u_{\mathbf{b}} = u_{\mathbf{a} +_{BW(R)} \mathbf{b}}$. The sum $\mathbf{a} +_{BW(R)} \mathbf{b}$ is given by $\mathbf{s}$, where $s = (\mathbf{a} +_{BW(R)} \mathbf{b})_0 \in \mathfrak{M}_R$. The definition of the addition law on $BW(R)$ finally yields
    \[
    a \oplus b = (\mathbf{a} + \mathbf{b})_0 = \lim_{m \to \infty} S_m( (-1)^m a^{p^m}, \ldots, a; (-1)^m b^{p^m}, \ldots, b).
    \]
\item Multiplication by $-1$ in the $\mathbb{Q}_p$-vector space structure introduced above is given simply by $-1 \cdot a = -1 \cdot (a^{(n)})_{n \in \mathbb{N}} = (-a^{(n)})_{n \in \mathbb{N}}$. In particular, denoting by $\ominus$ the difference in $A_{\alpha}$, we have
\begin{equation}\label{eq: subtraction Ac}
\begin{aligned}
a \ominus b & = (\mathbf{a} - \mathbf{b})_0 = (\mathbf{a} \oplus (-\mathbf{b}))_0 \\
& = \lim_{m \to \infty} S_m( (-1)^m a^{p^m}, (-1)^{m-1} a^{p^{m-1}}, \ldots, a; (-1)^{m+1} b^{p^m}, (-1)^{m} b^{p^{m-1}}, \ldots, -b).
\end{aligned}
\end{equation}
\end{enumerate}
We note furthermore that the map $\Psi$ described above is $\operatorname{Gal}(\overline{\mathbb{Q}}_p/K)$-equivariant  for the natural action of $\operatorname{Gal}(\overline{\mathbb{Q}}_p/K)$ on $A_\alpha \subseteq \mathfrak{M}_R$ \cite[p.~126]{Volkov}. In particular, we have:

\begin{lemma}\label{lemma: Tc or Ac ça m'est égal}
    Let $c' > c >0$. The map $\Psi$ induces $\operatorname{Gal}(\overline{\mathbb{Q}}_p/K)$-equivariant isomorphisms $T_{c'} \to A_{c'}$ and $T_c \to A_c$, and therefore also a Galois-equivariant isomorphism $T_c/T_{c'} \cong A_c/A_{c'}$.
\end{lemma}

Since we will often have to consider quotients of the form $A_c/A_{c'}$, it will be useful to use the notation $a \equiv b \bmod{A_{c'}}$ to mean $a \ominus b \in A_{c'}$. Note that the difference is taken with respect to the vector space structure introduced above.

\begin{lemma}\label{lemma: valuations of elements in Ac}
    Let $v(\alpha^{-1}) \geq 0$ and let $c>0$. For every $a \in A_c$ there exists $n \in \mathbb{Z}$ such that $v_R(a) = \frac{p^{2n}}{p^2-1}$.
\end{lemma}

\begin{proof}
   See \cite[top of p.~128]{Volkov}.
\end{proof}

\begin{proposition}\label{prop: Tc is TpE}
Let $v(\alpha^{-1}) \geq 0$ and let $c>0$. Then $T_c$ is homothetic to the Tate module $T_pE$. In particular $E[p^k] \cong T_c/p^kT_c$ for all integers $k>0$.
\end{proposition}

\begin{proof}
    Since $v(\alpha^{-1})\geq 0$, the representation $E[p]$ is absolutely irreducible by \cite[bottom of p.~129]{Volkov}. The proposition now follows from \Cref{lemma: any lattice will do}(i)-(ii).
\end{proof}

\begin{lemma}\label{lem: Tc vs Tp}
    Let $c>0$ and $v(\alpha^{-1})\geq 0$. We have $p^{k}T_{c}=T_{p^{2k}c}$, hence $E[p^k] \cong  T_c/T_{p^{2k}c} \cong A_c/A_{p^{2k}c}$ as $\operatorname{Gal}(\overline{\mathbb{Q}}_p/K)$-modules. Moreover, $A_{p^{2k}c}=p^kA_c$.
\end{lemma}

\begin{proof}
        
Let $u \in T_c$. By \Cref{def: Tc,def: A alpha and A c} and \Cref{lem: V_alpha}(v), there exists $a \in A_c$ such that $u=u_{\mathbf{a}}$, where 
\[
\mathbf{a}=(a_n)_{n \in \mathbb{Z}} \in BW(R), \quad a_n = (-1)^n a^{p^{-n}}
\]
and we use again the convention $a^{p^{-n}}=a^{(n)}$ for $n \geq 0$. Recalling that $[p] = \varphi \text{V}$ on $BW(R)$, we have
\[
[p] \mathbf{a} = [p](a_n)_{n \in \mathbb{Z}} = \varphi \text{V} (a_n)_{n \in \mathbb{Z}}= \varphi (a_{n-1})_{n \in \mathbb{Z}} = (a_{n-1}^p)_{n \in \mathbb{Z}} = ( -a_{n}^{p^2} )_{n \in \mathbb{Z}},
\]
where the last equality follows from the definition of the $a_n$.
Clearly, $[p] \mathbf{a}$ is of the form $\mathbf{b} = \left( (-1)^n b^{p^{-n}} \right)$, with $b=-a^{p^2}$, and in particular $v_R(b)=p^2v_R(a)$. Applying the bijection $\Psi^{-1}$ we find $[p]u_{\mathbf{a}}=u_{\mathbf{b}}$, and $u_{\mathbf{a}}$ is in $T_c$ if and only if $u_{\mathbf{b}}$ is in $T_{p^2c}$. On the other hand, given that multiplication by $p$ is bijective on $V_\alpha$, every $u_\mathbf{b} \in T_{p^2c}$ is of the form $u_\mathbf{b} = p u_\mathbf{a}$ for some $u_\mathbf{a} \in V_\alpha$, and the same valuation calculation shows $u_\mathbf{a} \in T_c$.
It follows that $pT_c = T_{p^2c}$. 
This clearly implies $p^kT_c=T_{p^{2k}c}$ by induction. The isomorphisms involving $A_c$ and the final statement follow from \Cref{lemma: Tc or Ac ça m'est égal} and the linearity of $\Psi$, and the isomorphism with $E[p^k]$ follows from \Cref{prop: Tc is TpE}.
\end{proof}

\subsection{Properties of $g_k(x)$}
We establish some basic properties of the roots $\mathcal{R}_k$ of $g_k(x)$ (see \Cref{eq: polynomial g}), namely their valuations and the valuations of their differences. To do this, we define a suitable shift of $g_k$ and then compute its Newton polygon (which will be predominantly done in \Cref{lemma: shifted Newton polygon}).

\begin{proposition}\label{prop: shifted Newton polygon}
Assume $v(\alpha^{-1}) \geq 0$. For all roots $\gamma \in \mathcal{R}_k,$ there is a partition $\mathcal{R}_k\setminus \{\gamma\} = \bigsqcup_{i=0}^{k-1} \mathcal{R}_{k,\gamma}^i$ which satisfies:
\begin{itemize}
    \item $\# \mathcal{R}_{k,\gamma}^i = p^{2i}(p^2-1)$;
    \item $v(\gamma'-\gamma)=\frac{1}{p^{2i}(p^2-1)}$ for all $\gamma' \in \mathcal{R}_{k,\gamma}^i$.
\end{itemize}
In particular, all roots of $\mathcal{R}_k$ are distinct and $\#\mathcal{R}_k=p^{2k}$.
\end{proposition}

\begin{proof}
Let $h_k(t)=g_k(t+\gamma)$. The elements $\gamma'-\gamma$ for $\gamma' \in \mathcal{R}_k$ are the roots of $h_k(t)$. The result will follow from a determination of the Newton polygon of $h_k$; namely, that it is given by the convex hull of the points
\[
\{(0,\infty) \} \, \cup \, \{(p^{2i},k-i) : 0 \leq i \leq k\}.
\]

Indeed, this Newton polygon consists of segments of length $1$ and $p^{2i+2}-p^{2i}$ for $0 \leq i \leq k-1$, with respective slopes $\infty$ and $\frac{1}{(p^2-1)p^{2i}}$ for $0 \leq i \leq k-1$, which implies the statement. Observe that the coefficient $h_k^{(i)}$ of $t^i$ in $h_k$ is given by
\begin{equation}\label{eq: coefficients shifted polynomial}
h_k^{(i)} = \binom{p^{2k}}{i}\gamma^{p^{2k}-i} + \sum_{n=1}^k (-1)^np^n \left( \alpha^{-1}\pi_e^2 \binom{p^{2k+1-2n}}{i}\gamma^{p^{2k+1-2n}-i} + \binom{p^{2k-2n}}{i}\gamma^{p^{2k-2n}-i} \right),
\end{equation}
where $\binom{a}{b}=0$ if $b>a$. The Newton polygon of $h_k$ now follows from \Cref{lemma: shifted Newton polygon}.
\end{proof}

\begin{remark}\label{rmk: sum of roots in a group}
    Note that for all $0 \leq i \leq k-1$ we have $\sum_{\gamma' \in \mathcal{R}_{k, \gamma}^i} v(\gamma-\gamma') = 1$.
\end{remark}

Before we complete the computation of the Newton polygon, we first state a couple of straightforward combinatorial facts on the valuation of binomial coefficients.

\begin{lemma}\label{lemma: binomial valuation}
Let $p$ be a prime and let $j$ be a positive integer.  Then $v\left( \binom{p^j}{a} \right) = j-v(a)$ for all $1 \leq a \leq p^j$. In particular:
\begin{enumerate}
    \item $v\left( \binom{p^j}{a} \right) \ge 1$ whenever $a \ne 0, p^j$;
    \item If $1 \leq a< p^m$ for some $m \leq j$, then $v\left( \binom{p^j}{a} \right) \geq j-m+1$.
\end{enumerate}
\end{lemma}

\begin{proof}
We use Legendre’s formula $v(n!)=\sum_{i\ge1}\left\lfloor \frac{n}{p^i}\right\rfloor$ to compute $v\binom{p^j}{a}
= \sum_{i \geq 1} \left( \left\lfloor \frac{p^j}{p^i}\right\rfloor
- \left\lfloor \frac{a}{p^i}\right\rfloor
- \left\lfloor \frac{p^j-a}{p^i}\right\rfloor \right)$. For each $1 \leq i \leq j$, write $a=qp^i +r$ with $0 \leq r < p^i$, so $p^j-a=(p^{j-i}-q-1)p^i+(p^i-r)$. One then checks that $\left\lfloor \frac{p^j}{p^i}\right\rfloor
- \left\lfloor \frac{a}{p^i}\right\rfloor
- \left\lfloor \frac{p^j-a}{p^i}\right\rfloor \in \{0,1\}$ and equals $0$ if and only if $r=0$. It is now clear that the terms that contribute $1$ are precisely the $i$ such that $v(a) +1 \leq i \leq j$, giving $v\binom{p^j}{a}
= j-v(a)$, as claimed.
\end{proof}

\begin{lemma}\label{lemma: shifted Newton polygon}
Assume $v(\alpha^{-1}) \geq 0$ and $k\geq 1$. Fix $\gamma \in \mathcal{R}_k$ and let $h_k(t) = g_k(t+\gamma) = \sum_{i=0}^{p^{2k}} h_k^{(i)} t^i$. Then
\begin{enumerate}
    \item $h_k^{(p^{2k})}=1$ and $h_k^{(0)}=0$;
    \item $v(h_k^{(p^{2j})})=k-j$ for $0 \leq j \leq k-1$;
    \item if $p^{2m} < i < p^{2m+2}$ for some $0 \leq m \leq k-1$, then $v(h_k^{(i)}) \geq v(h_k^{(p^{2m})})=k-m$.
\end{enumerate}
\end{lemma}

\begin{proof}
(i): $h_k^{(p^{2k})}$ is the leading coefficient of $h_k$ and hence equal to that of $g_k$. Note $h_k^{(0)}=g_k(\gamma)=0$ by definition.

(ii): Let $0 \leq j \leq k-1$ and note that the binomomial coefficients in \eqref{eq: coefficients shifted polynomial} (taking $i=p^{2j}$) are 0 for $n>k-j$, so we get the following expression for the $p^{2j}$-th coefficient of $h_k$:
\[
h_k^{(p^{2j})} = \underbrace{\binom{p^{2k}}{p^{2j}}\gamma^{p^{2k}-p^{2j}}}_{>2k-2j} + \sum_{n=1}^{k-j} (-1)^n \left( \underbrace{\alpha^{-1}\pi_e^2p^n \binom{p^{2k+1-2n}}{p^{2j}}\gamma^{p^{2k+1-2n}-p^{2j}}}_{>2k+1-n-2j} + \underbrace{p^n\binom{p^{2k-2n}}{p^{2j}}\gamma^{p^{2k-2n}-p^{2j}}}_{\geq 2k-n-2j} \right)
\]
where in the underbraces we give bounds on the valuations of the terms, using \Cref{lemma: binomial valuation} and noting that $v(\alpha^{-1}) \geq 0$ and $v(\gamma),v(\pi_e)>0$. It is clear that there is a unique term of minimal valuation: the rightmost term of the sum when $n=k-j$, in which case the exponent of $\gamma$ is $0$ and the valuation is $k-j$ as claimed.

(iii) Let $p^{2m}<i<p^{2m+2}$. Using \Cref{lemma: binomial valuation}(ii) and considering only the powers of $p$ and binomial coefficients, we get the following lower bounds on valuations:
\[
h_k^{(i)} = \underbrace{\binom{p^{2k}}{i}\gamma^{p^{2k}-i}}_{{\geq 2k-2m-1}} + \sum_{n=1}^{k-1} (-1)^n \left( \underbrace{\alpha^{-1}\pi_e^2p^n \binom{p^{2k+1-2n}}{i}\gamma^{p^{2k+1-2n}-i}}_{{\geq 2k-2m-n}} + \underbrace{p^n\binom{p^{2k-2n}}{i}\gamma^{p^{2k-2n}-i}}_{{\geq 2k-2m-n-1}} \right).
\]

We now examine these terms more closely, showing that each of them has valuation at least $k-m$. Note $2k-2m-1=(k-m) + ((k-1) - m) \geq k-m$ since $m \leq k-1$, so $v(\binom{p^{2k}}{i}) \geq k-m$ always holds.
We now consider the valuation of $p^n\binom{p^{2k+1-2n}}{i}$ which is at least $2k-2m-n$; this is at least $k-m$ whenever $n \leq k-m$. On the other hand, for this summand to occur, we must have $p^{2m} < i \leq p^{2k+1-2n}$ which implies that $n \leq k-m$, so this summand always has valuation at least $k-m$ as claimed.
The final term of $p^n\binom{p^{2k-2n}}{i}$ with valuation at least $2k-2m-n-1$ is analogous: this is at least $k-m$ when $n \leq k-m-1$ and the appearance of the summand implies $n \leq k-m-1$ as before.
\end{proof}

\begin{remark}\label{rmk: bijection is sensible}
As we will prove formally in what follows, there is a natural bijection between the roots of $g_k(x)$ and the points in $E[p^k]$.
Under this bijection, the root $0$ corresponds to the identity element of $E[p^k]$ and the $p^{2i}(p^2-1)$ roots of valuation $\frac{1}{p^{2i}(p^2-1)}$ correspond to the $p^{2i}(p^2-1)$ points of order exactly $p^{i+1}$ for $0 \leq i \leq k-1$. The intuitive interpretation of the fact that the partition structure of $\mathcal{R}_k \setminus \{\gamma\}$ is independent of the choice of $\gamma$ is that for every choice of $P \in E[p^k]$, the number of elements $Q \in E[p^k]$ such that $Q-P$ has order $p^i$ depends only on $i$ and not on $P$.
\end{remark}

\subsection{Construction of the map $\Phi_k$}\label{sect: construction of Phik}
In this section, we construct the map $\Phi_k$ whose existence is asserted by \Cref{thm: p2 torsion points correspond to roots of g}, assuming the notation and hypotheses of that theorem. 

\begin{lemma}\label{lemma: roots distance}
    Assume $v(\alpha^{-1}) \geq 0$. Let $n \in \{0, \dots , k\}$ and $a \in \C_p$. Let $\gamma \in \mathcal{R}_k$ be a root of $g_k$ that maximises $v(a-\gamma)$ amongst the roots of $g_k$.
    If $v(g_k(a)) > k-n+\frac{1}{p^2-1}$, then $v(a-\gamma)> \frac{1}{p^{2n}(p^2-1)}$. Moreover, if $v(g_k(a)) \le k-n+1 + \frac{1}{p^2-1}$, then $v(a-\gamma) = \frac{1}{p^{2n}}(v(g_k(a)) - k+n)$. 
\end{lemma}

\begin{proof}
    Write $g_k(x) = \prod_{i=1}^{p^{2k}} (x-\gamma_i)$ where $\gamma=\gamma_1$. If $a=\gamma$, there is nothing to prove; hence, from now on we assume that $a \neq \gamma$.
    Let $r$ be the unique integer such that $\frac{1}{p^{2r}(p^2-1)} < v(a-\gamma) \le \frac{1}{p^{2r-2}(p^2-1)}$ and fix an ordering on the roots $\gamma_i$ such that the sequence $(v(\gamma-\gamma_i))_{1 \leq i \leq p^{2k}}$ is (weakly) decreasing. We claim that:
    \[
    v(a-\gamma_i)=\begin{cases}
        v(a-\gamma) & \text{if } 1 \leq i \leq p^{2r}, \\
        v(\gamma-\gamma_i) &\text{else}.
    \end{cases}
    \]
    Indeed, if $i>p^{2r}$, then by \Cref{prop: shifted Newton polygon}, $v(\gamma-\gamma_i)<v(a-\gamma)$ so $v(a-\gamma_i)= v((a-\gamma) + (\gamma-\gamma_i)) = v(\gamma_i-\gamma)$. On the other hand, if $i \leq p^{2r},$ then $v(\gamma_i - \gamma) \ge v(a-\gamma)$ so $v(a-\gamma_i) = v((a-\gamma) + (\gamma-\gamma_i)) \ge v(a-\gamma)$, which must be an equality by maximality and hence proves the claim.
    If $r \in \{0, \dots, k\}$, using \Cref{rmk: sum of roots in a group} we obtain
    \begin{align}\label{eq: valuation of g(x)}
        v(g_k(a)) &= v\left( \prod_{i=1}^{p^{2k}} (a-\gamma_i) \right) = v\left( \prod\limits_{i=1}^{p^{2r}} (a-\gamma_i) \right) + v\left( \prod\limits_{i=p^{2r}+1}^{p^{2k}} (a-\gamma_i) \right) = p^{2r} v(a-\gamma) + k-r.
    \end{align}
    For $r<0$ or $r>k$, by a similar calculation we get
    \[
    v(g_k(a)) = v(a-\gamma)+k \quad \text{(if $r<0$)} \quad \quad \text{or} \quad \quad v(g_k(a)) = p^{2k}v(a-\gamma) \quad \text{(if $r>k$)}.
    \]
    
    Now note that, for $r<0$, by definition of $r$ we have $v(a-\gamma) > \frac{p^2}{p^2-1}$ and the first part of the lemma follows. For the second part, the inequality $v(g_k(a)) = v(a-\gamma) + k > 1 + \frac{1}{p^2-1} + k$ implies that the additional assumption never holds. The case $r>k$ cannot arise: by definition of $r$ we would have $v(a-\gamma) \leq \frac{1}{p^{2k}(p^2-1)}$, but the inequality $v(g_k(a)) = p^{2k}v(a-\gamma) \leq p^{2k} \frac{1}{p^{2k}(p^2-1)} = \frac{1}{p^2-1}$ then contradicts the assumption.
    
    It follows that $0 \leq r \leq k$. \Cref{eq: valuation of g(x)}, combined with the assumption $v(g_k(a)) > k-n+\frac{1}{p^2-1}$, then implies
    \[
    r-n+\frac{1}{p^2-1} < v(g_k(a))+r-k = p^{2r}v(a-\gamma) \leq \frac{p^{2r}}{p^{2r-2}(p^2-1)} = 1+ \frac{1}{p^2-1},
    \]    
    \noindent and hence $n \ge r$. In particular, we have $$v(a-\gamma) > \frac{1}{p^{2r}(p^2-1)} \ge \frac{1}{p^{2n}(p^2-1)},$$
    which concludes the proof of the first part of the statement.
    For the second half, it suffices to show that under the additional assumption we have $r=n$, since the conclusion then follows from \Cref{eq: valuation of g(x)}. To do so, notice that if we had $r \le n-1$ we would obtain
    \[
    \frac{1}{p^2-1} + k-n+1 <p^{2r}v(a-\gamma)+k-r = v(g_k(a)) \leq k-n+1 + \frac{1}{p^2-1},
    \]
    giving a contradiction.
\end{proof}

\begin{lemma}\label{lemma: unique root at level k}
Assume $v(\alpha^{-1}) \geq 0$. Let $c_k=\frac{1}{p^{2k-2}(p^2-1)}$ and let $a \in A_{c_k}$. Suppose $p > \max\{\sqrt{k}, k-v(\alpha^{-1}) -1\}$. 
Then there exists a unique root $\gamma \in \mathcal{R}_k$ such that $v(a^{(0)} - \gamma) > \frac{1}{p^2-1}$.
\end{lemma}

\begin{proof}
Note if $\gamma$ exists, then it must be unique since if $\gamma'$ is any other such root, then $v(\gamma-\gamma') > \frac{1}{p^2-1}$ which contradicts \Cref{prop: shifted Newton polygon}.

For ease of notation set $\tilde{a}:=a^{(0)}$ and $\varepsilon := \frac{1}{p^2-1}$. Recall that by definition $\tilde{a}$ is a root of the series in \Cref{eqn: gk to series}. Multiplying \Cref{eqn: gk to series} by $\pi_ep^k$ and rearranging gives
\[
g_k(\tilde{a}) = 
        \left( \sum_{\substack{n \in \mathbb{Z} \\ n \neq -k,...,0}} (-1)^{n+k+1} p^{n+k}\left(  \tilde{a}^{p^{-2n}} + \alpha^{-1} \pi_e^2 {\tilde{a}}^{p^{-2n+1}} \right) \right) - \alpha^{-1}\pi_e^2 \tilde{a}^{p^{2k+1}},
\]
where we use the convention $\tilde{a}^{p^{-m}} := a^{(m)}$ for $m \geq 0$.

We claim that all the terms on the right-hand side have valuation strictly greater than $k+\varepsilon$. This is clear for the summands with $n>0$, and moreover 

\begin{equation*}
    v(\alpha^{-1}\pi_e^2\tilde{a}^{p^{2k+1}}) = v(\alpha^{-1}) + p^{2k+1}v(\tilde{a}) + \frac{2}{e} \geq v(\alpha^{-1}) +  p^{2k+1}c_k + \varepsilon > v(\alpha^{-1}) + p + \varepsilon
\end{equation*}

which is at least $k+\varepsilon$ by assumption.
If $n:=-m \leq -k-1$, then

\begin{eqnarray*}
    v \left(p^{n+k}\left(\tilde{a}^{p^{-2n}}+\alpha^{-1}\pi_e^2\tilde{a}^{p^{-2n+1}} \right) \right) &=& v \left(p^{k-m}\left(\tilde{a}^{p^{2m}}+\alpha^{-1}\pi_e^2\tilde{a}^{p^{2m+1}} \right) \right) \\
    &\geq& (k-m) + \min\{p^{2m}v(\tilde{a}),p^{2m+1}v(\tilde{a})\} \\
    &\geq& (k-m) + p^{2m}c_k \\
    &=& (k-m) + p^{2m-2k}(1 + \varepsilon) \\
    &\geq& -1 + p^2(1 + \varepsilon) = p^2 + \varepsilon > k + \varepsilon.
\end{eqnarray*}
Hence we have $v(g_k(\tilde{a})) > k+\varepsilon$. The conclusion follows from \Cref{lemma: roots distance} applied with $n=0$. 
\end{proof}

\Cref{lemma: unique root at level k} justifies the following definition:
\begin{definition}\label{def: map Phi}
Assume the hypotheses of \Cref{thm: p2 torsion points correspond to roots of g} and let $c_k =\frac{1}{p^{2k-2}(p^2-1)}$. We define
\[
\Phi_k : A_{c_k} \to \mathcal{R}_k
\]
to be the map that takes $a \in A_{c_k}$ to the unique root $\gamma$ of $g_k(x)$ that satisfies $v(\gamma - a^{(0)}) > \frac{1}{p^2-1}$.
\end{definition}

\subsection{Approximate addition formulae in $A_c$}

We now prove two technical lemmas which will allow us to replace the complicated subtraction formulae of the $\mathbb{Q}_p$-vector space structure on $A_c$ with much simpler operations in $R$ and in $\mathcal{O}_{\mathbb{C}_p}$.  Note that given elements of $A_c \subset R$ we can subtract them both using the vector space structure of \S \ref{sect: lattices in Valpha} (denoted $a \ominus b$) and using the ring structure of $R$ (denoted $a-b$).

\begin{lemma}\label{lemma: valuations for the two differences are the same}
    Let $c>0$ and $a, b \in A_{c}$. Then $v_R(a \ominus b) = v_R(a-b).$
\end{lemma}

\begin{proof}
    By \Cref{eq: subtraction Ac} and \Cref{lemma: addition polynomials in one variable} we have
    \[
    \begin{aligned}
    a \ominus b & = \lim_{m \to \infty} S_m((-1)^ma^{p^m} , \ldots, a^{p^2}, -a^p, a; (-1)^{m+1}b^{p^m} , \ldots, -b^{p^2}, b^p, -b) \\
    & = \lim_{m \to \infty} (a-b)(1+g_m(a,b)).
    \end{aligned}
    \]
    In particular, since $v_R(a)>0$ and $v_R(b)>0$, and $g_m$ has no constant term, we have $v_R(g_m(a,b))>0$ and therefore $v_R(1+g_m(a,b)) = 0$. Passing to the limit in the previous equation and using the fact that the valuation is continuous, we obtain as desired $v_R(a \ominus b) = v_R(a-b)$.
\end{proof}

\begin{lemma}\label{lemma: congruence modulo Apkc}
    Let $c_k=\frac{1}{p^{2k-2}(p^2-1)}$ and $a, b \in A_{c_k}$. The following are equivalent:
    \begin{enumerate}
        \item $a \equiv b \bmod{A_{p^{2k}c_k}}$;
        \item $v_R(a - b) \geq p^{2k} c_k$;
        \item $v(a^{(0)}-b^{(0)}) > \frac{1}{p^2-1}$.
    \end{enumerate}
\end{lemma}

\begin{proof}
(i) $\Leftrightarrow$ (ii): This is immediate from \Cref{lemma: valuations for the two differences are the same} since $a \equiv b \bmod{A_{p^{2k}c_k}}$ is by definition equivalent to $a \ominus b \in A_{ p^{2k}c_k}$, that is, $v_R(a \ominus b) \geq p^{2k}c_k$.

Before proving the remaining equivalences, we first make a remark. By definition we have
\[
v_R(a - b) = \lim_{n \to \infty} v\left( a^{(n)} - b^{(n)} \right)^{p^n},
\]
and since $p$ is odd we can write
\begin{equation*}
\begin{aligned}
\left( a^{(n)} - b^{(n)} \right)^{p^n} & = \left( a^{(n)} \right)^{p^n} - \left( b^{(n)} \right)^{p^n} + \sum_{1 \leq i \leq p^n-1} \binom{p^n}{i}  \left( a^{(n)} \right)^i \left( -b^{(n)} \right)^{p^n-i} \\
& = a^{(0)}-b^{(0)} + \left(\text{terms of valuation }\geq 1+c_k\right),
\end{aligned}
\end{equation*}
where we have used \Cref{lemma: binomial valuation}(i), and $v(a^{(n)}) = \frac{1}{p^n}v(a^{(0)}) \geq \frac{c_k}{p^n}$ and similarly $v(b^{(n)}) \geq \frac{c_k}{p^n}$.
In particular, passing to the limit in $n$ we obtain 
\begin{equation}\label{eq: equivalence pos valuation}
    v_R(a-b) > \frac{1}{p^2-1} \Longleftrightarrow v(a^{(0)}-b^{(0)}) > \frac{1}{p^2-1}.
\end{equation}

(ii) $\Rightarrow$ (iii): This follows directly from \Cref{eq: equivalence pos valuation} since $v_R(a-b) \geq p^{2k}c_k = \frac{p^2}{p^2-1} > \frac{1}{p^2-1}$.  

(iii) $\Rightarrow$ (ii): Note that if $v(a^{(0)}-b^{(0)}) > \frac{1}{p^2-1}$ then $v_R(a \ominus b) =v_R(a-b) > \frac{1}{p^2-1}$ by \Cref{lemma: valuations for the two differences are the same} and \Cref{eq: equivalence pos valuation}. Now the description of possible valuations for $a \ominus b$ given by \Cref{lemma: valuations of elements in Ac} shows that $v_R(a-b)=v_R(a \ominus b) \geq \frac{p^2}{p^2-1}=p^{2k}c_k$ as desired.
\end{proof}

\subsection{Proof of \Cref{thm: p2 torsion points correspond to roots of g}}\label{subsect: proof of iso}
We now show that the map $\Phi_k$ of \Cref{def: map Phi} factors via $A_{c_k}/p^kA_{c_k}$ and induces a Galois-equivariant bijection $\Phi_k : A_{c_k}/ p^kA_{c_k} \to \mathcal{R}_k$.

\begin{theorem}\label{thm: Phi gives pk map}
$\Phi_k$ descends to an isomorphism $A_{c_k}/p^kA_{c_k} \rightarrow \mathcal R_k$ of $\op{Gal}(\overline{K}/K)$-sets.
\end{theorem}

\begin{proof}
Let $\tilde\Phi_k: A_{c_k}/p^kA_{c_k} \to \mathcal{R}_k$ be the induced map. We must show that it is well-defined, bijective and Galois equivariant. Recall from \Cref{lem: Tc vs Tp} that $p^kA_{c_k}=A_{p^{2k}c_k}$.

Well-defined: Let $a_1 \equiv a_2 \bmod{A_{p^{2k}c_k}}$ and let $\gamma_1=\Phi_k(a_1), \gamma_2=\Phi_k(a_2)$. Then $v(a_1^{(0)}-a_2^{(0)}) > \frac{1}{p^2-1}$ by \Cref{lemma: congruence modulo Apkc} and $v(\gamma_1-a_1^{(0)}), v(\gamma_2-a_2^{(0)}) > \frac{1}{p^2-1}$ by definition of $\Phi_k$. Hence

\[
v(\gamma_1-\gamma_2) = v((\gamma_1-a_1^{(0)}) + (a_1^{(0)}-a_2^{(0)}) -  (\gamma_2-a_2^{(0)})) > \frac{1}{p^2-1}.
\]

\noindent Therefore $\gamma_1=\gamma_2$ by \Cref{prop: shifted Newton polygon} and $\tilde\Phi_k$ is well-defined.

Bijective: Note that both sides have order $p^{2k}$ (see \Cref{lem: Tc vs Tp} and \Cref{prop: shifted Newton polygon}), so it suffices to prove injectivity. Let $\gamma = \tilde\Phi_k(a_1)=\tilde\Phi_k(a_2)$ so $v(\gamma-a_1^{(0)}), v(\gamma-a_2^{(0)}) > \frac{1}{p^2-1}$. Hence $v(a_1^{(0)} - a_2^{(0)}) > \frac{1}{p^2-1}$ and therefore $a_1 \equiv a_2 \bmod{A_{p^{2k}c_k}}$ by \Cref{lemma: congruence modulo Apkc}.

Galois equivariant: Let $\sigma \in \op{Gal}(\overline{K}/K)$. Since $\sigma$ permutes the roots of $g_k$ (which is defined over $K$) and is an isometry for valuations, we have that $v(\sigma(\Phi_k(a))-\sigma(a^{(0)}))=v(\Phi_{k}(a)-a^{(0)}) > \frac{1}{p^2-1}$ and hence $\sigma(\Phi_k(a))$ satisfies the defining condition to be $\Phi_k(\sigma(a))$, so $\Phi_k$ is Galois equivariant as required.
\end{proof}

\begin{proof}[Proof of \Cref{thm: p2 torsion points correspond to roots of g}]
    This is an immediate corollary of \Cref{thm: Phi gives pk map}, noting that $A_{c_k}/p^kA_{c_k} \cong E[p^k]$ by \Cref{lem: Tc vs Tp}; since $g_k$ is defined over $K$, the equality of fields also follows.
\end{proof}

Before closing this section, we detail the upgrade of $\Phi_k: E[p^k] \cong A_{c_k}/p^kA_{c_k} \rightarrow \mathcal{R}_k$ (alluded to in \Cref{rmk: twisted Galois action of Phi}) from a $\Gal(\Qpbar/K)$-set isomorphism to a $\Gal(\Qpbar/\Qp)$-set isomorphism, giving both an algebraic definition and a ``nearest root" definition. For $\sigma \in \Gal(\Qpbar/\Qp)$, we shall write $ \sigma \bmod{\Gal(\Qpbar/K)}$ for its image in $G_{K/\Qp} \cong \Gal(\Qpbar/\Qp)/\Gal(\Qpbar/K)$.

\begin{proposition}\label{rmk: twisted Galois action 2}
Suppose the hypotheses of \Cref{thm: p2 torsion points correspond to roots of g}. Let $\gamma \in \mathcal{R}_k$ and define
\[
\sigma \bullet \gamma = \begin{cases}
\Phi_k\!\big( \zeta_e^{\,m}\, \sigma(\Phi_k^{-1}(\gamma)) \big), & \text{if } \sigma \bmod \Gal(\overline{\Q}_p/K) = \omega \tau_e^{m},\\[4pt]
\Phi_k\!\big( \zeta_e^{-m}\, \sigma(\Phi_k^{-1}(\gamma)) \big), & \text{if } \sigma \bmod \Gal(\overline{\Q}_p/K) = \tau_e^{m}.
\end{cases}
\]

Then 
\begin{enumerate}
    \item $\Phi_k$ extends to an isomorphism $E[p^k] \rightarrow \mathcal{R}_k$ of $\Gal(\Qpbar/\Qp)$-sets, with respect to the natural action on $E[p^k]$ and the $\bullet$-action on $\mathcal{R}_k$.
    \item If $\sigma \in \Gal(\Qpbar/K)$, then $\sigma \bullet \gamma = \sigma(\gamma)$ is the natural action.
    \item $\sigma \bullet \gamma$ is the unique root $\tilde{\gamma}$ of $\mathcal{R}_k$ which maximises $v(\tilde{\gamma}-\zeta_e^m \sigma(\gamma))$ (resp.~$v(\tilde{\gamma}-\zeta_e^{-m} \sigma(\gamma))$) if $\sigma \bmod{\Gal(\Qpbar/K)} = \omega\tau_e^m$ (resp.~$\sigma \bmod{\Gal(\Qpbar/K)} = \tau_e^m)$.
    \end{enumerate}
\end{proposition}

\begin{proof}
First note that Volkov \cite[p.~126]{Volkov}, starting from the action of Equation \eqref{eq: twisted action}, defines a twisted action of $\Gal(\overline{\Q}_p/\Q_p)$ on $a \in A_\alpha$ by 
\[
\sigma \star a :=
\begin{cases}
   \zeta_e^{\,m}\, \sigma(a), & \text{if } \sigma \bmod \Gal(\overline{\Q}_p/K) = \omega \tau_e^m,\\[3pt]
   \zeta_e^{-m}\, \sigma(a), & \text{if } \sigma \bmod \Gal(\overline{\Q}_p/K) = \tau_e^m,
\end{cases}
\]
where $\sigma(a)$ denotes the natural action of $\Gal(\overline{\Q}_p/\Q_p)$ on $A_\alpha \subset R$. With this action, the natural bijection $A_{c_k}/p^kA_{c_k} \cong E[p^k]$ given by \Cref{lem: Tc vs Tp} becomes an isomorphism of $\Gal(\overline{\Q_p}/\Q_p)$-sets. It then suffices to check that $\Phi_k : A_{c_k} / p^kA_{c_k} \to \mathcal{R}_k$ is $\Gal(\overline{\Q}_p/\Q_p)$-equivariant if $A_{c_k}/p^kA_{c_k}$ is equipped with the $\star$-action and $\mathcal{R}_k$ with the $\bullet$-action.

(i): It is clear by definition that the bijection $\Phi_k: A_{c_k}/p^kA_{c_k} \rightarrow \mathcal{R}_k$ becomes a $\Gal(\overline{\Q}_p/\Q_p)$-set isomorphism via the action $\gamma \mapsto \Phi_k(\sigma \star \Phi_k^{-1}(\gamma))$ for all $\sigma \in \Gal(\overline{\Q}_p/\Q_p)$ and $\gamma \in \mathcal{R}_k$; this is precisely the definition of $\sigma \bullet \gamma$.

(ii): If $\sigma \in \Gal(\Qpbar/K)$, then since $\Phi_k$ is $\Gal(\overline{\Q}_p/K)$-equivariant with the natural action by \Cref{thm: Phi gives pk map} and $\sigma  \star a = \sigma(a)$ for all $a \in A_{\alpha}$ it follows that $\sigma \bullet \gamma=\sigma(\gamma)$ is the the natural action.

(iii): Suppose $\sigma \bmod{\Gal(\Qpbar/K)} = \omega\tau_e^m$; the other case is identical. It suffices to prove that 
\[
v(\Phi_k(\zeta_e^m\sigma(\Phi_k^{-1}(\gamma))) - \zeta_e^m\sigma(\gamma)) > \frac{1}{p^2-1},
\]
since the claim then follows from \Cref{prop: shifted Newton polygon}, so it remains to compute this valuation. Write 
\[
\Phi_k(\zeta_e^m \sigma(\Phi_k^{-1}(\gamma))) - \zeta_e^m \sigma(\gamma)  = \Phi_k(\zeta_e^m \sigma(\Phi_k^{-1}(\gamma))) - \left( \zeta_e^m \sigma(\Phi_k^{-1}(\gamma)) \right)^{(0)} +  \left( \zeta_e^m \sigma(\Phi_k^{-1}(\gamma)) \right)^{(0)} - \zeta_e^m \sigma(\gamma).
\]
Now $v \left( \Phi_k(\zeta_e^m \sigma(\Phi_k^{-1}(\gamma))) - \left( \zeta_e^m \sigma(\Phi_k^{-1}(\gamma)) \right)^{(0)} \right) >  \frac{1}{p^2-1}$ by definition of $\Phi_k$ (\Cref{def: map Phi}). On the other hand, let $a$ be any lift of $\Phi_k^{-1}(\gamma)$ to $A_{c_k}$. Then using the structure of $A_{c_k}$, we have

\[
\left( \zeta_e^m \sigma(\Phi_k^{-1}(\gamma)) \right)^{(0)} - \zeta_e^m \sigma(\gamma) = \zeta_e^m \sigma(a)^{(0)} - \zeta_e^m\sigma(\Phi_k(a)) = \zeta_e^m\sigma \left(a^{(0)} - \Phi_k(a) \right);
\]
since $\zeta_e$ is a unit and $\sigma$ is an isometry, we once again have by \Cref{def: map Phi} that this has valuation strictly greater than $\frac{1}{p^2-1}$. The claimed inequality on the valuation now follows.
\end{proof}

\begin{remark}\label{rmk: alternative division polynomials for small e}
One could remove the semistability defect assumption in \Cref{thm: p2 torsion points correspond to roots of g} as follows. Suppose $e=1$ so $E/\Qp$ has good supersingular reduction. The corresponding filtered $\varphi$-module $D=D_c^\ast(1,0)$ is unique and formally analogous to the $D_\infty$ case (cf.~\Cref{def: D alpha}), except that $D$ is defined over $\Qp$ rather than $\Q_{p^2}$ \cite[p.~117]{Volkov}. Repeating the results in this section verbatim yields a (non-twisted) $\Gal(\overline{\Q}_p/\Q_p)$-equivariant bijection between $E[p^k]$ and the roots of $g_k^\infty(x) := \sum_{n=0}^k (-1)^n p^n x^{p^{2k-2n}}$ whenever $p>\max\{\sqrt{k},3\}$. The $e=2$ case is simply a quadratic twist of this. 

One could then use this approach to determine the $p$-adic image of $E$ (namely, it will fall under \Cref{thm: more precise description padic Galreps}(i) for all $k<p^2$), but this scenario has already been covered by the stronger statement in \Cref{lemma: p-adic image in the twisted crystalline case}.
\end{remark}

\section{Classification of $p$-adic images of elliptic curves over $\Q_p$}\label{sect: integral Tate module}

Recall the group-theoretic classification of $p$-adic Galois images given in \Cref{thm:ellipticcartantower}. One of our principal results of the next few sections is that case (iii) cannot arise.

\begin{theorem}\label{thm: no sharp group over Qp}
Let $p>7$ be a prime and let $E/\Q_p$ be an elliptic curve such that $\NRep{p} \subseteq C_{ns}^+(p)$. Set $G:=\operatorname{Im}\rho_{E,p^\infty}$ and let $e$ be the semistability defect of $E/\Q_p$. 
One of the following holds:
	\begin{enumerate}
		\item $G \subseteq C_{ns}^+$ and $[C_{ns}^+ : G] =[C_{ns}^+(p):G(p)] \in \{1,3\}$, with $G = C_{ns}^+$ if $p \equiv 1 \bmod 3$ or $e \in\{1,2,4\}$;
		\item there exists $n \ge 1$ such that $G \supseteq \op{Id} + p^n M_{2}(\Z_p)$ and $G(p^n) \subseteq C_{ns}^+(p^n)$. In particular, $G$ is the inverse image of $\operatorname{Im} \rho_{E, p^n}$. Moreover, we have 
        \[
        [C_{ns}^+(p^n) : G(p^n)] =[C_{ns}^+(p):G(p)] \in \{1,3\},
        \]
        with $G(p^n) = C_{ns}^+(p^n)$ if $p \equiv 1 \bmod 3$ or $e \in\{1,2,4\}$.
	\end{enumerate}
\end{theorem}

\begin{remark}\label{rmk: uniformity}
Note that case (i) can arise not only in the potential CM case (\Cref{lemma: description in potential CM case}) but also more generally in the potential \textit{formal} CM case; see \cite[\S1.1]{HMRL21} for the definition of formal complex multiplication and \cite[\S3.5]{lubintate} for how to construct such curves. See also \cite[Remark 3.2]{furiolombardo23} as well as \Cref{rmk: index 3} below for examples showing that the index 3 case $[C_{ns}^+(p) : G(p)]=3$ does actually arise for $p \equiv 2 \bmod 3$.
\end{remark}

\subsection{Preliminary reductions for \Cref{thm: no sharp group over Qp}}\label{subsec: preliminary reductions}

We briefly recap the cases of \Cref{thm: no sharp group over Qp} that we have already proved and make some reduction steps, before giving an outline of the remaining proof.

\begin{enumerate}
    \item If the semistability defect $e$ of $E$ is at most $2$, \Cref{thm: no sharp group over Qp} is immediately implied by \Cref{lemma: p-adic image in the twisted crystalline case}, so we may suppose $e \in \{3,4,6\}.$
    \item As the mod $p$ image is contained in $C_{ns}^+(p)$, we have $e \mid p+1$ by \Cref{prop: e divides p plus one} and $E$ has potentially good supersingular reduction and no canonical subgroup of order $p$ by \Cref{cor: supersingular and no canonical subgroup}.
    \item By \Cref{thm: Volkov parametrisation}, we therefore have an isomorphism $V_pE \cong V_{\alpha}$ of Galois modules for some $\alpha \in \mathbb{P}^1(\Qp)$.
    \item The hypotheses and conclusion of \Cref{thm: no sharp group over Qp} are invariant under quadratic twisting (note that by \Cref{thm:ellipticcartantower} the index $[C_{ns}^+(p^n) : G(p^n)]$ is always odd, so it cannot change by a factor of $2$). We may thus suppose $v(\alpha^{-1}) \geq -1$ by \Cref{rmk: alpha role}(i).
    \item Moreover, since $p>7$, we may assume $v(\alpha^{-1}) \neq -1$ since $E$ has no canonical subgroup of order $p$ (see \Cref{rmk: alpha role}(ii) and step (ii) above), so $v(\alpha^{-1}) \geq 0$.
\end{enumerate}

\begin{setup}\label{setup}
    We let $p>3$ be a prime and $E/\Q_p$ be an elliptic curve with potentially supersingular reduction and semistability defect $e \in \{3, 4, 6\}$. We assume that $e \mid p+1$ and let $\alpha$ be as in \Cref{thm: Volkov parametrisation}, so that $V_pE \cong V_\alpha$. We further assume that $v(\alpha^{-1}) \geq 0$.
\end{setup}

\Cref{thm: no sharp group over Qp} will be a consequence of the following more technical but more precise result. Note that this statement holds for all primes $p>3$: it is only in the application to \Cref{thm: no sharp group over Qp} that we need to assume $p>7$.

\begin{theorem}\label{thm: more precise description padic Galreps}
Assume \Cref{setup} and let $k$ be a positive integer such that $p>\sqrt{k}$.
\begin{enumerate}
    \item If $k\leq v(\alpha^{-1})+1$, then the group $\operatorname{Im} \rho_{E, p^k}$ is contained in $C_{ns}^+(p^k)$.
    \item If $k=v(\alpha^{-1})+2$, then the group $\operatorname{Im} \rho_{E, p^\infty}$ contains $\operatorname{Id} + p^{k-1}M_2(\Z_p)$.
\end{enumerate}
In particular, if $p > \max\{7,\sqrt{v(\alpha^{-1})+2}\}$, the conclusion of \Cref{thm: no sharp group over Qp}(ii) holds with $n=v(\alpha^{-1})+1$, and therefore $\operatorname{Im} \rho_{E, p^\infty}$ is the inverse image in $\GL_2(\Z_p)$ of $\operatorname{Im} \rho_{E, p^n} \subset \GL_2(\Z/p^n\Z)$.
\end{theorem}

\begin{remark}
If $k>v(\alpha^{-1})+2$, the conclusion of \Cref{thm: more precise description padic Galreps}(ii) is clearly still valid, so long as $p>\sqrt{v(\alpha^{-1})+2}$.
\end{remark}

Before proving \Cref{thm: more precise description padic Galreps}, we show that it easily implies \Cref{thm: no sharp group over Qp}.

\begin{proof}[Proof of \Cref{thm: no sharp group over Qp} assuming \Cref{thm: more precise description padic Galreps}]
We simply need to rule out case (iii) of \Cref{thm:ellipticcartantower}. By the above discussion we can assume $e \in \{3,4,6\}$, $e \mid p+1$ and $v(\alpha^{-1}) \geq 0$. We apply \Cref{thm: more precise description padic Galreps} with $k=2$: if $v(\alpha^{-1})=0$, \Cref{thm: more precise description padic Galreps}(ii) implies that $G$ contains $\operatorname{Id}+pM_2(\Z_p)$, which is incompatible with \Cref{thm:ellipticcartantower}(iii). On the other hand, if $v(\alpha^{-1}) \geq 1$, then \Cref{thm: more precise description padic Galreps}(i) shows $G(p^2) \subseteq C_{ns}^+(p^2)$, which is again incompatible with \Cref{thm:ellipticcartantower}(iii).
\end{proof}

We prove \Cref{thm: more precise description padic Galreps} in the next two sections. More precisely, to prove (i), in \Cref{sect: large valuation} we show that the image of the mod $p^k$ Galois representation attached to $E$ is the same as for a fixed CM elliptic curve with deformation parameter $\alpha^{-1}=0$ and hence has image contained in $C_{ns}^+(p^k)$. In \Cref{sect: valuation zero}, we prove (ii) by analysing the ramification index of $p$ in the fields $K(E[p^k])$ and $K(E[p^{k-1}])$. Note that the final statement in \Cref{thm: more precise description padic Galreps} follows immediately applying (i) to $k=v(\alpha^{-1})+1$ and (ii) to $k=v(\alpha^{-1})+2$.

\section{Cartan image}\label{sect: large valuation}

In this section, we work with the notation and assumptions of \Cref{setup} and establish \Cref{thm: more precise description padic Galreps}(i), 
which will follow from \Cref{thm: technical result for large valuation} below.
The key result states that when $k$ is sufficiently small, the Galois structure of the roots of $g_k(x)$ (as in \Cref{eq: polynomial g}) is constant in an explicit neighbourhood of $\alpha=\infty$; via \Cref{thm: p2 torsion points correspond to roots of g} we can relate this back to the Galois structure of $E[p^k]$. The relevant polynomial for $\alpha=\infty$ is 
\[
g_k^{\infty}(x) = \sum_{n=0}^{k} (-1)^np^n x^{p^{2k-2n}}.
\]

\begin{theorem}\label{thm: technical result for large valuation}
Suppose $k \leq v(\alpha^{-1}) +1$. The splitting fields of $g_k$ and $g_k^{\infty}$ over $K$ coincide.
\end{theorem}

We first show the following technical lemma.

\begin{lemma}\label{lemma: valuation of g(gamma) and h(beta)}
Let $\gamma_\infty$ be a root of $g_k^\infty$, $\gamma$ a root of $g_k$, and $\zeta$ a $(p^2-1)$-th root of unity. Then:
\begin{itemize}
    \item $v(g_k(\gamma_\infty)) \geq  v(\alpha^{-1}) +1 + \frac{2}{e} + \frac{p}{p^2-1}$;
    \item $v(g_k^\infty(\gamma)) = v(g_k(\zeta \gamma)) \geq  v(\alpha^{-1}) +1 + \frac{2}{e} + \frac{p}{p^2-1}$.
\end{itemize}
Moreover equality holds when $v(\gamma_\infty)$ (resp.~$v(\gamma))$ equals $\frac{1}{p^{2k-2}(p^2-1)}$.
\end{lemma}

\begin{proof}
    Set $g=g_k$ and $h=g_k^\infty$. We have
    \begin{equation}\label{eq: proof of valuation of g(gamma) and h(beta)}
        g(x) = h(x) + \alpha^{-1}\pi_e^2\sum_{n=1}^k (-1)^n p^n x^{p^{2k-2n+1}}.
    \end{equation}
    
    We claim that for $x \in \{\gamma, \gamma_\infty, \zeta\gamma\}$, we have $v(p^n x^{p^{2k-2n+1}}) \geq 1 + \frac{p}{p^2-1}$ with equality precisely when $n=1$ and $v(x)= \frac{1}{p^{2k-2}(p^2-1)}$ and strict inequality otherwise. Indeed, this is clear if $n \geq 2$; the case $n=1$ is a direct computation using that $v(x)\geq \frac{1}{p^{2k-2}(p^2-1)}$ by \Cref{prop: shifted Newton polygon}. Combining this claim with \eqref{eq: proof of valuation of g(gamma) and h(beta)} and the fact that $g(\gamma)=h(\gamma_{\infty})=0$ then completes the proof for the valuations of $g(\gamma_{\infty})$ and $h(\gamma)$.

    Finally, consider $x=\zeta\gamma$. Notice that $\zeta^{p^2} = \zeta$, and so $h(\zeta\gamma) = \zeta h(\gamma)$. Hence
    \begin{align*}
        g(\zeta\gamma) &= \zeta h(\gamma) + \alpha^{-1}\pi_e^2 \zeta^p \sum_{n=1}^k (-1)^n p^n \gamma^{p^{2k-2n+1}} = \zeta g(\gamma) + \alpha^{-1}\pi_e^2 (\zeta^p-\zeta) \sum_{n=1}^k (-1)^n p^n \gamma^{p^{2k-2n+1}} \\
        &= \alpha^{-1}\pi_e^2 (\zeta^p-\zeta) \sum_{n=1}^k (-1)^n p^n \gamma^{p^{2k-2n+1}}.
    \end{align*}
    
    Since $\zeta^p-\zeta$ is a unit, we get $v(g(\zeta\gamma)) = v\left(\alpha^{-1}\pi_e^2 \sum_{n=1}^k (-1)^n p^n \gamma^{p^{2k-2n+1}}\right)=v(g(\gamma)-h(\gamma))=v(h(\gamma))$, where we have used \eqref{eq: proof of valuation of g(gamma) and h(beta)} and the fact that $g(\gamma)=0$ by assumption. This concludes the proof.

\end{proof}

\begin{remark}\label{rmk: g zetabeta congruence}
    Notice that in the case $x=\zeta\gamma$ we actually proved that $g(\zeta\gamma) + \alpha^{-1}\pi_e^2 (\zeta^p-\zeta) p\gamma^{p^{2k-1}}$ has valuation strictly greater than $v(\alpha^{-1}) + 2 + \frac{2}{e}$, which is a stronger statement.
\end{remark}

\begin{proof}[Proof of \Cref{thm: technical result for large valuation}]
Let $\gamma_{\infty}$ be a nonzero root of $g_k^{\infty}$. We show that there exists a root $\gamma$ of $g_k$ such that $v(\gamma_\infty-\gamma)> \frac{1}{p^2-1}$ from which Krasner's lemma gives $K(\gamma_{\infty}) \subseteq K(\gamma)$ since for any other root $\gamma_\infty'$ of $g_k^\infty$ we have $v(\gamma_\infty - \gamma_\infty') \leq \frac{1}{p^2-1}$ by \Cref{prop: shifted Newton polygon}. Since our arguments will be symmetric in $\gamma$ and $\gamma_\infty$, it suffices to prove containment in one direction to get equality of splitting fields. 
Now since $v(\alpha^{-1})\geq k-1$, \Cref{lemma: valuation of g(gamma) and h(beta)} yields
\[
v(g_k(\gamma_\infty)) \geq v(\alpha^{-1}) + 1 + \frac{2}{e} + \frac{p}{p^2-1}
 > k + \frac{1}{p^2-1}.
\]
To conclude, we apply \Cref{lemma: roots distance} with $a=\gamma_\infty$ and $n=0$ to determine the existence of the claimed root $\gamma$ of $g_k$ such that $v(\gamma_\infty-\gamma)> \frac{1}{p^2-1}$.

\end{proof}

We can now prove \Cref{thm: more precise description padic Galreps}(i):

\begin{proof}
Let $G=\operatorname{Im} \rho_{E, p^\infty}$. Note that $k-v(\alpha^{-1})-1 \leq 0$ by assumption, so $\max\{\sqrt{k}, k-v(\alpha^{-1}), 3 \} <p$. Hence we may apply \Cref{thm: p2 torsion points correspond to roots of g} to get that $K(E[p^k])$ is the splitting field over $K$ of the polynomial $g_k(x)$, which coincides with the splitting field $F$ of $g_k^{\infty}(x)$ by \Cref{thm: technical result for large valuation}. By \Cref{cor: splitting field in limit case} there exists a CM elliptic curve $E'/\Q_p$ with $j$-invariant $0$ or $1728$ such that $K(E'[p^k])=F$. By \Cref{lemma: description in potential CM case}(i), the largest power of $p$ that divides $[F : \Q_p]$ is at most $2(k-1)$ (it is easy to show that it is exactly $2(k-1)$, but we don't need this). Therefore
\[
v(\#G(p^k)) = v([F:\Q_p]) \leq 2(k-1),
\]
which is only compatible with the cases where $G(p^k) \subseteq C_{ns}^+(p^k)$ in \Cref{thm:ellipticcartantower}.
\end{proof}

\section{Maximal growth}\label{sect: valuation zero}

As in the previous section, we work with the notation and assumptions of \Cref{setup}.  We now consider the case $k=v(\alpha^{-1})+2$ and prove maximal growth in the following result, from which \Cref{thm: more precise description padic Galreps}(ii) will follow easily. Recall that the field $K=\Q_p(\pi_e, \zeta_e)$ (introduced in \Cref{sect: filtred phi modules}) is the splitting field of $x^e+p$.

\begin{theorem}\label{thm: valuation 0}
    Assume the hypotheses and notation of \Cref{setup} and set $k=v(\alpha^{-1})+2$. Assume furthermore that $p>\sqrt{k}$.
    The field extension $K(E[p^k])/K(E[p^{k-1}])$ has degree $p^4$, hence a fortiori $[\mathbb{Q}_p(E[p^k]) : \mathbb{Q}_p(E[p^{k-1}])]=p^4$.
\end{theorem}

\begin{proof}[Proof of \Cref{thm: more precise description padic Galreps}(ii)]
     Let $G=\operatorname{Im} \rho_{E, p^\infty}$  and suppose by contradiction $G \not \supseteq I + p^{k-1}M_2(\Z_p)$. By \Cref{thm:ellipticcartantower} we then have either $G(p^{k}) \subseteq C_{ns}^+(p^{k})$ with index coprime to $p$, or $k=2$ and $G(p^2) \subseteq G_{ns}^\#(p^2)$. 
    If $G(p^2) \subseteq G_{ns}^\#(p^2)$, by definition of the group $G_{ns}^\#(p^2)$ we have $[\Q_p(E[p^2]) : \Q_p(E[p])] =\frac{\#G(p^2)}{\#G(p)} \mid \#G_{ns}^\#(p^2)=p^3 \cdot 2(p^2-1)$, while \Cref{thm: valuation 0} shows $[\mathbb{Q}_p(E[p^2]) : \mathbb{Q}_p(E[p])]=p^4$. 
    If instead $G(p^{k}) \subseteq C_{ns}^+(p^{k})$, by \Cref{thm: valuation 0} we have $p^4 = [\Q_p(E[p^{k}]) : \Q_p(E[p^{k-1}])] = \frac{|G(p^{k})|}{|G(p^{k-1})|} = \frac{|C_{ns}^+(p^{k})|}{|C_{ns}^+(p^{k-1})|} = p^2$, giving a contradiction.
\end{proof}

For the rest of the section, we set the following notation.
\begin{notation}\label{alpha unit notation}
Let $k=v(\alpha^{-1})+2 \geq 2$ and $g(x)=g_k(x)$ be the polynomial given in \Cref{eq: polynomial g}. We let $F$ be the splitting field of $g$ over $K$ and $\beta,\gamma \in F$ be roots of $g(x)$ with valuation $\frac{1}{p^{2k-2}(p^2-1)}$. Throughout our computations, we write $a \equiv b \bmod p^c$ to mean $v(a-b) \geq c$.
\end{notation}

Observe that since $k=v(\alpha^{-1})+2$, the hypotheses of \Cref{thm: valuation 0} (namely that $p>\sqrt{k}$ and $p>3$) are sufficient to apply \Cref{thm: p2 torsion points correspond to roots of g}(ii) to obtain that $F =K(E[p^k])$. Before embarking on the calculations, we briefly describe our approach to proving \Cref{thm: valuation 0} which is to construct a suitable element $\varepsilon \in F$ with $v(\varepsilon)=\frac{a}{p^{2k}}$ for some integer $a$ coprime to $p$; from this we can directly conclude the proof by a ramification argument.
It remains to describe how to build $\varepsilon$. We produce this from two roots $\beta, \gamma$ of $g$ iteratively as
\begin{equation}\label{eq: epsilons and hs}
    \varepsilon_r = \beta- u_0\gamma^{h_0} - u_1\gamma^{h_1}- \cdots - u_{r-1}\gamma^{h_{r-1}},
\end{equation}
for suitable units $u_i \in \mathcal{O}_F^\times$ and integers $h_i$ such that $v(\beta)<v(\varepsilon_1)<v(\varepsilon_2)<\ldots<v(\varepsilon_r)$ is strictly increasing. In particular, this forces $h_i=\frac{v(\varepsilon_{i})}{v(\gamma)}$. By appropriately choosing the units $u_i$, we show that $\varepsilon_2$ or $\varepsilon_3$ has the required property when $k=2$, whilst if $k>2$, then $\varepsilon_k$ suffices. With the strategy outlined, we now give a proof of \Cref{thm: valuation 0}  modulo the ramification calculation of the next sections.

\begin{proof}[Proof of \Cref{thm: valuation 0}.]
First suppose there exists $\varepsilon \in K(E[p^k])$ such that $v(\varepsilon)=\frac{a}{p^{2k}}$ with $a$ coprime to $p$. Then the extension $K(E[p^k])/\Qp$ has ramification degree divisible by $p^{2k}$.
By \Cref{thm: more precise description padic Galreps}(i) applied to $k-1$, we have $\operatorname{Im}\rho_{E,p^{k-1}} \subseteq C_{ns}^+(p^{k-1})$, so the $p$-part of the ramification index of $K(E[p^{k-1}])/\Qp$ is bounded by $p^{2k-4}$ and hence the ramification degree of $K(E[p^k])/K(E[p^{k-1}])$ is $p^4$. Since the degree is bounded by $p^4$, it follows that $[K(E[p^k]) : K(E[p^{k-1}])]=p^4$.

It remains to exhibit such an element $\varepsilon$. If $k=2$, then $\varepsilon_2$ or $\varepsilon_3$ suffices (Lemma \ref{lem: epsilon2}), whereas if $k>2$, then we use $\varepsilon=\varepsilon_k$ (\Cref{prop: recursive congruence}).
\end{proof}

\subsection{Preliminary estimates}

We note the following version of the freshman's dream:
\begin{lemma}\label{lemma: freshmans dream}
Let $n\geq 1$ be an integer and $a,b \in \overline{\Q}_p$ with $v(a) \geq 0, v(b) \geq 0$. Then
\[
(a+b)^{p^n} \equiv a^{p^n}+b^{p^n} \bmod p^{ 1 + p^n \min\{v(a), v(b)\}}.
\]
\end{lemma}
\begin{proof}
We have $(a+b)^{p^n} - (a^{p^n}+b^{p^n}) = \sum_{i=1}^{p^n-1}\binom{p^n}{i} a^ib^{p^n-i}$. Each binomial coefficient has valuation at least $1$ by \Cref{lemma: binomial valuation}(i), while $v(a^i b^{p^n-i})=i v(a) + (p^n-i) v(b) \geq p^n \min\{ v(a), v(b) \}$.
\end{proof}

\begin{lemma}\label{lemma: g(x+y) = g(x) + g(y)}
    Suppose $p >3$ and recall that $k=v(\alpha^{-1})+2$. Let $x,y \in \overline{\Q}_p$ be elements such that $v(x) \ge \frac{1}{3p^2}$ and $v(y) \ge 0$. 
    We have $g(x+y) \equiv g(x) + g(y) \mod p^{k + \min\{pv(x), 1+v(x)\}}$.
\end{lemma}

\begin{proof}
Set $h(x) = g_k^\infty(x) = \sum_{n=0}^{k} (-1)^np^n x^{p^{2k-2n}}$. By definition we have $$g(x+y) = h(x+y) + \alpha^{-1}\pi_e^2 \sum_{n=1}^k (-1)^n p^n (x+y)^{p^{2k-2n+1}}.$$
Letting $H_{n,i} := (-1)^n p^n \binom{p^{2k-2n}}{i}x^i y^{p^{2k-2n}-i}$, every term in $h(x+y)$ is of the form
\[
(-1)^n p^n (x+y)^{p^{2k-2n}} = (-1)^n p^n x^{p^{2k-2n}} + (-1)^n p^n y^{p^{2k-2n}} + \sum_{i=1}^{p^{2k-2n}-1} H_{n,i}.
\]
Consider a fixed $i \in \{1,\ldots,p^{2k-2n}-1\}$ and let $r \in \Z$ be the unique integer such that $p^{r-1} \le i < p^r$. Note that $1 \le r \le 2k-2n$. We show that for every $0 \le n \le k-1$ and $1 \le i \le p^{2k-2n}-1$ we have $v(H_{n,i}) \ge k+\min\{pv(x), 1 + v(x)\}$. By \Cref{lemma: binomial valuation}(ii) we have
$$v(H_{n,i}) \ge n + (2k-2n-r+1) + iv(x) = k + (k-n) - r + 1 + iv(x).$$
Since $k-n \geq \lceil \frac{r}{2} \rceil$ and $i \geq p^{r-1}$ by definition, we obtain $v(H_{n,i}) \ge k - \lfloor \frac{r}{2} \rfloor + 1 + p^{r-1}v(x)$; this equals $k+1+v(x)$ when $r = 1$. If $r \geq 2$, we claim $v(H_{n,i})$ is greater than $k+pv(x)$. To see this, note that $k- \lfloor \frac{r}{2} \rfloor + 1 + p^{r-1}v(x) = k+pv(x) + (p^{r-1}-p)v(x) - \lfloor \frac{r}{2} \rfloor + 1 \geq k+pv(x) + \frac{p^{r-2}-1}{3p} - \lfloor \frac{r}{2} \rfloor + 1$, so it suffices to prove $\frac{p^{r-2}-1}{3p}  \geq \lfloor \frac{r}{2} \rfloor - 1$, which is easily done by induction.
    
This shows $h(x+y) \equiv h(x) + h(y) \mod p^{k + \min\{pv(x), 1+v(x)\}}$. To conclude, it suffices to prove that $$\alpha^{-1}\pi_e^2 \sum_{n=1}^k (-1)^n p^n (x+y)^{p^{2k-2n+1}} \equiv \alpha^{-1}\pi_e^2 \sum_{n=1}^k (-1)^n p^n (x^{p^{2k-2n+1}} + y^{p^{2k-2n+1}}) \mod p^{k + \min\{pv(x), 1+v(x)\}},$$
    for which it suffices to show that $p^n (x+y)^{p^{2k-2n+1}} \equiv p^n (x^{p^{2k-2n+1}} + y^{p^{2k-2n+1}}) \mod p^{2+\min\{pv(x),1+v(x)\}}$ for $1 \leq n \leq k$ since $v(\alpha^{-1})=k-2.$
    Expanding the power on the left-hand side and cancelling like terms, we are left with showing
\[
p^n \sum_{i=1}^{p^{2k-2n+1}-1} {p^{2k-2n+1} \choose i} x^i y^{p^{2k-2n-1}-i} \equiv 0 \bmod{p^{2+\min\{p v(x), 1+v(x) \}}}.
\]
Note that every summand has valuation at least $1+v(x)$, since each binomial coefficient has valuation at least $1$ by \Cref{lemma: binomial valuation}(i). The statement is now obvious for $n \geq 2$. For $n=1$, we consider separately $i < p$ and $i \geq p$. If $i \geq p$, then the summand has valuation at least $1+iv(x) \geq 1+pv(x)$, whereas if $i <p$, then \Cref{lemma: binomial valuation}(ii) gives $v(\binom{p^{2k-1}}{i}x^i) \geq 2k-1+1-1 + iv(x) \geq 3 +v(x)$ since $k \geq 2$.
\end{proof}

Before proving the next lemma, let us remind the reader of the explicit expression for the condition $g(\gamma)=0$:
\begin{align}\label{eq: g gamma is zero}
    \begin{split}
    0 = g(\gamma) &= \gamma^{p^{2k}}+ \sum_{n=1}^{k} (-1)^np^n (\alpha^{-1}\pi_e^2\gamma^{p^{2k+1-2n}} + \gamma^{p^{2k-2n}}) = h(\gamma) + \alpha^{-1}\pi_e^2 \sum_{n=1}^k (-1)^n p^n \gamma^{p^{2k-2n+1}},
    \end{split}
\end{align}
where $h(x)=\sum_{n=0}^k (-1)^n p^n x^{p^{2k-2n}}$. We will also need the following remark about the $e$-th roots of $p$.

\begin{remark}\label{rmk: 2eth roots of p}
As $e \mid p+1$, the field $\Q_p(\zeta_e)$ is $\Q_{p^2}$, which also contains the $2e$-th roots of unity since $p^2 \equiv 1 \pmod{2e}$ for $e \in \{3,4,6\}$ and all $p>3$.
In particular, the $e$-th roots of $p$ have the form $\zeta\pi_e$, where $\zeta$ satisfies $\zeta^e=-1$, and are all in $K= \Q_p(\zeta_e, \pi_e)$.
\end{remark}

\begin{lemma}\label{lemma: congruences powers of gamma}
We have $\gamma^{p^{2k-2}(p^2-1)} \equiv p(1 + \alpha^{-1} \pi_e^2 \gamma^{p^{2k-2}(p-1)}) \mod p^{2-\frac{1}{p^2}}$. Moreover, there exists an $e$-th root of unity $\xi$ such that
\[
\gamma^{p^{2k-2} \cdot \frac{2(p^2-1)}{e}} \equiv \xi \pi_e^2 \left(1 + \frac{2}{e}\alpha^{-1}\pi_e^2 \gamma^{p^{2k-2}(p-1)} \right) \quad \bmod{p^{1+ \frac{2}{p+1}}}.
\]
\end{lemma}

\begin{proof}
Let $B := v(p^2 \gamma^{p^{2k-4}}) = 2 + \frac{1}{p^2(p^2-1)}$. In \Cref{eq: g gamma is zero}, only $\gamma^{p^{2k}}$ and the summands with $n=1$ can possibly have valuation less than $B$, so we have
$
\gamma^{p^{2k}} - p \left( \alpha^{-1} \pi_e^2 \gamma^{p^{2k-1}} + \gamma^{p^{2k-2}} + \delta \right)=0
$
for some $\delta \in \overline{\Q}_p$ of valuation at least $B-1$. Dividing by $\gamma^{p^{2k-2}}$ we obtain $\gamma^{p^{2k-2}(p^2-1)} = p \left( 1+ \alpha^{-1} \pi_e^2 \gamma^{p^{2k-2}(p-1)}  + \delta' \right)$ with $v(\delta') \geq B-1 - v(\gamma^{p^{2k-2}}) = B-1- \frac{p^{2k-2}}{p^{2k-2}(p^2-1)}=1-\frac{1}{p^2}$, which gives the first part of the statement.

Now notice that the Taylor series for $(1+x)^{2/e}$ has coefficients in $\Z_p$ since $p \nmid e$, so it converges for any $x$ with $v(x) > 0$ and satisfies $v\left((1+x)^{2/e} - \left(1 + \frac{2}{e}x\right)\right) \geq 2 v(x)$. Raising the equality $\gamma^{p^{2k-2}(p^2-1)} = p \left( 1+ \alpha^{-1} \pi_e^2 \gamma^{p^{2k-2}(p-1)}  + \delta' \right)$ to the power $2/e$, and accounting for an unknown $e$-th root of unity $\xi$, we then get $\gamma^{p^{2k-2} \cdot\frac{2(p^2-1)}{e}} = \xi \pi_e^2 \left( 1+ \frac{2}{e} \alpha^{-1} \pi_e^2 \gamma^{p^{2k-2}(p-1)}  + \frac{2}{e}\delta' + \delta'' \right)$ with $v(\pi_e^2) = \frac{2}{e}$, $v(\delta') \geq 1-\frac{1}{p^2}$, and $v(\delta'') \geq 2\min \{ v(\delta'), v\left(\alpha^{-1} \pi_e^2 \gamma^{p^{2k-2}(p-1)}\right) \} \geq \min\{2-\frac{2}{p^2}  ,\frac{4}{e} + \frac{2}{p+1} \}$.
These inequalities easily imply the second congruence in the statement.
\end{proof}

\subsection{Construction of $\varepsilon$}
We now begin to implement the strategy described after \Cref{alpha unit notation}. 
Given a root $\gamma$ of $g_k$ of valuation $\frac{1}{p^{2k-2}(p^2-1)}$ and a $(p^2-1)$-th root of unity $\zeta$, by \Cref{lemma: valuation of g(gamma) and h(beta)} we have $v(g_k(\zeta\gamma)) = v(\alpha^{-1}) + 1 + \frac{2}{e} + \frac{p}{p^2-1} = k-1 +\frac{2}{e} + \frac{p}{p^2-1}$, and hence $k-1+\frac{1}{p^2-1} < v(g_k(\zeta\gamma)) \le k + \frac{1}{p^2-1}$. In particular, by \Cref{lemma: roots distance} there exists another root $\beta$ of $g_k$ such that $v(\beta- \zeta\gamma) = \frac{1}{p^2}\left( \frac{2}{e} + \frac{p}{p^2-1} \right)$.
For the purposes of exposition in the upcoming calculations of $\varepsilon_i$, we provide the reader with a brief table of notation that will be used throughout.

\begin{notation*}
\phantom{~}

\begin{tabular}{cl}
   $g=g_k$ & polynomial given in \Cref{eq: polynomial g} \\
   $\gamma$  & a root of $g$ of valuation $\frac{1}{p^{2k-2}(p^2-1)}$ \\
   $\xi$ & the $e$-th root of unity associated to $\gamma$ arising in \Cref{lemma: congruences powers of gamma} \\
   $\zeta$  & a fixed primitive $(p^2-1)$-th root of unity \\
   $\beta$ & a root of $g$ of valuation $\frac{1}{p^{2k-2}(p^2-1)}$ such that $v(\beta-\zeta\gamma) = \frac{1}{p^2}\left( \frac{2}{e} + \frac{p}{p^2-1} \right)$ \\
   $u$ & the element $u \in \Z_p^\times$ such that $\alpha^{-1} = p^{k-2} u$.
\end{tabular}
\end{notation*}

Throughout the course of our computations, we not only give the valuations of the partial sums $\varepsilon_i$ (and hence determine $h_{i+1}$, see \Cref{eq: epsilons and hs} and the comments following it) but also provide congruence conditions on $\varepsilon_i$ that will facilitate future unit calculations. We start building our partial sums $\varepsilon_i$ by setting $u_0=\zeta$ and $h_0=1$, so $\varepsilon_1=\beta-\zeta\gamma$.

\begin{lemma}\label{lemma: epsilon1}
Suppose $p > 3$. Let $\varepsilon_1=\beta-\zeta\gamma$ and $\alpha^{-1} = u p^{k-2}$ with $u \in \Z_p^\times$. Then $\varepsilon_1$ satisfies
\begin{equation}\label{eq: lemma epsilon1}
     (-1)^k\pi_e^2u(\zeta-\zeta^p)\gamma^{p^{2k-1}} - \varepsilon_1^{p^2} + p\varepsilon_1 \equiv 0 \quad \bmod{p^{1+\frac{2}{ep}+\frac{1}{p^2-1}}}. 
\end{equation}
\end{lemma}

\begin{proof}
    By definition, we have $g(\varepsilon_1 + \zeta\gamma) = 0$ and $v(\varepsilon_1) = \frac{1}{p^2} \left( \frac{2}{e} + \frac{p}{p^2-1} \right) \geq \frac{1}{3p^2}$. We may then apply \Cref{lemma: g(x+y) = g(x) + g(y)} to get
    \begin{equation}\label{eq: auxiliary eq for epsilon1 eq}
        0 = g(\varepsilon_1 + \zeta\gamma) \equiv g(\varepsilon_1) + g(\zeta\gamma) \equiv g(\varepsilon_1) - \alpha^{-1}\pi_e^2 (\zeta^p-\zeta) p\gamma^{p^{2k-1}} \mod p^{k+\frac{2}{ep}+\frac{1}{p^2-1}},
    \end{equation}
    where the last congruence follows from \Cref{rmk: g zetabeta congruence}. On the other hand, for every $r \ge 2$ we have
    $$v(p^{k-r}\varepsilon_1^{p^{2r+1}}) > v(p^{k-r}\varepsilon_1^{p^{2r}}) = k-r + p^{2r-2} \left( \frac{2}{e} + \frac{p}{p^2-1} \right) > k - 2 + \frac{2p^2}{e} + p > k + \frac{2}{ep} + \frac{1}{p^2-1},$$
    and so expanding the definition of $g(\varepsilon_1)$ we find
    $$g(\varepsilon_1) \equiv (-1)^k (p^k \varepsilon_1 - p^{k-1}\varepsilon_1^{p^2}) + (-1)^{k} \alpha^{-1}\pi_e^2(p^k \varepsilon_1^p - p^{k-1} \varepsilon_1^{p^3}) \equiv (-1)^k (p^k \varepsilon_1 - p^{k-1}\varepsilon_1^{p^2}) \mod p^{k+\frac{2}{ep}+\frac{1}{p^2-1}},$$
    where the last congruence holds because $v(p^k\varepsilon_1^p) = k + \frac{2}{ep} + \frac{1}{p^2-1}$ and $v(p^{k-1}\varepsilon_1^{p^3}) = k - 1 + \frac{2p}{e} + \frac{p^2}{p^2-1} > k+1$.
    Writing $\alpha^{-1} = u p^{k-2}$ and dividing \Cref{eq: auxiliary eq for epsilon1 eq} by $(-1)^kp^{k-1}$ concludes the proof.
\end{proof}

We now have to treat separately the cases $k=2$ and $k>2$.

\subsection{The case $k=2$}
We start by focusing on the case $k=2$. Before we proceed with the next step, we first prove some congruences which will ease our calculations. Recall that in this case we have $v(\alpha^{-1}) = 0$.

\begin{lemma}\label{lemma: u1 and u2}
Let $u_1=\xi^{-1}\alpha^{-1}(\zeta-\zeta^p)$, $u_2=-\frac{2}{e} \,\xi^{-2}\alpha^{-2}(\zeta-\zeta^p)$, $h_1=\frac{2(p^2-1)}{e}+p$ and $h_2=\frac{4(p^2-1)}{e}+2p-1$. 
Then
\begin{enumerate}
    \item $(u_1\gamma^{h_1})^{p^2} \equiv \alpha^{-1}\pi_e^2(\zeta-\zeta^p)\gamma^{p^3}\left(1+\tfrac{2}{e}\alpha^{-1}\pi_e^2\gamma^{p^3-p^2}\right) \quad \bmod{p^{1+\frac{1}{p^2-1}}}.$
    \item $(u_2\gamma^{h_2})^{p^2} \equiv -\frac{2}{e}\alpha^{-2}\pi_e^4(\zeta-\zeta^p)\gamma^{2p^3-p^2} \quad \bmod{p^{1+\frac{1}{p^2-1}}}.$
\end{enumerate}
\end{lemma}

\begin{proof}
(i): First note that $u_1^{p^2} \equiv u_1 \bmod{p}$. This follows from the following facts: $\xi^{p^2} = \xi$ since $p^2 \equiv 1 \bmod{e}$;  $(\alpha^{-1})^{p} \equiv \alpha^{-1} \bmod p$ because $\alpha \in \mathbb{Z}_p^\times$; and $(\zeta-\zeta^p)^{p^2} \equiv \zeta-\zeta^p \bmod{p}$ by \Cref{lemma: freshmans dream} since $\zeta^{p^2}=\zeta$.

Now, $h_1p^2=p^3+p^2 \cdot \frac{2(p^2-1)}{e}$, so we can apply \Cref{lemma: congruences powers of gamma} to obtain
\begin{eqnarray*}
    \gamma^{h_1p^2} &\equiv& \xi\pi_e^2 \left(1+\tfrac{2}{e}\alpha^{-1}\pi_e^2\gamma^{p^3-p^2}\right) \gamma^{p^3} \quad \bmod{p^{1+\frac{1}{p^2-1}}}.
\end{eqnarray*}
Combining the two congruences concludes the proof of (i).

(ii): Similarly $u_2^{p^2} \equiv u_2 \bmod{p}$ and since $h_2p^2=2p^2 \cdot \frac{2(p^2-1)}{e} + (2p^3-p^2)$, we similarly obtain from \Cref{lemma: congruences powers of gamma}:
\begin{align*}
    \gamma^{h_2p^2} \ &\equiv \ \xi^2\pi_e^4 \left(1+\frac{2}{e}\alpha^{-1}\pi_e^2\gamma^{p^3-p^2}\right)^2 \gamma^{2p^3-p^2} \ \equiv \ \xi^2\pi_e^4 \gamma^{2p^3-p^2} \mod p^{1+\frac{1}{p^2-1}}
\end{align*}
since the terms suppressed in the final congruence have valuation at least $\frac{6}{e} + \frac{2p^3-p^2}{p^2(p^2-1)} \geq 1 + \frac{1}{p^2-1}$. We conclude as above.
\end{proof}

\begin{lemma}\label{lem: epsilon2}
Let $\varepsilon_2=\varepsilon_1-u_1\gamma^{h_1}$. Then precisely one of the following holds:
\begin{enumerate}
    \item $p^4$ divides the denominator of $v(\varepsilon_2)$; 
    \item $v(\varepsilon_2)=h_2v(\gamma)$ where $h_2=\frac{4(p^2-1)}{e}  +2p-1\in \mathbb{Z}$.
\end{enumerate}
In case (ii), let $u_2=-\frac{2}{e} \,\xi^{-2}\alpha^{-2}(\zeta-\zeta^p)$ and $\varepsilon_3=\varepsilon_2-u_2\gamma^{h_2}$. Then $v(\varepsilon_3)$ has denominator divisible by $p^4$.
\end{lemma}

\begin{proof}
    Replacing $\varepsilon_1$ with $\varepsilon_2+u_1\gamma^{h_1}$ in \Cref{eq: lemma epsilon1} yields:
    \begin{align*}
        0 \ &\equiv \ \alpha^{-1}\pi_e^2(\zeta-\zeta^p)\gamma^{p^3} - (\varepsilon_2+u_1\gamma^{h_1})^{p^2} + p(\varepsilon_2+u_1\gamma^{h_1}) & \mod p^{1+\frac{1}{p^2-1}} \\
        &\equiv \ \alpha^{-1}\pi_e^2(\zeta-\zeta^p)\gamma^{p^3} - \varepsilon_2^{p^2} - (u_1\gamma^{h_1})^{p^2} + p(\varepsilon_2+u_1\gamma^{h_1}) & \mod p^{1+\frac{1}{p^2-1}} \\
        &\equiv \ \alpha^{-1}\pi_e^2(\zeta-\zeta^p)\gamma^{p^3} - \varepsilon_2^{p^2} - \alpha^{-1}\pi_e^2(\zeta-\zeta^p)\gamma^{p^3}\left(1+\tfrac{2}{e}\alpha^{-1}\pi_e^2\gamma^{p^3-p^2}\right) + p(\varepsilon_2+u_1\gamma^{h_1}) & \mod p^{1+\frac{1}{p^2-1}} \\
        &\equiv \ -\tfrac{2}{e}\alpha^{-2}\pi_e^4(\zeta-\zeta^p)\gamma^{2p^3-p^2} - \varepsilon_2^{p^2} +p\varepsilon_2 +pu_1\gamma^{h_1} & \mod p^{1+\frac{1}{p^2-1}}&,
    \end{align*}
    where the second congruence follows from \Cref{lemma: freshmans dream} since $p^2v(\varepsilon_2)\geq p^2v(\gamma) = \frac{1}{p^2-1}$, while the third follows from \Cref{lemma: u1 and u2}(i).
    We consider again the valuations of these terms, which are $\frac{4}{e} {+} \frac{2p-1}{p^2-1}, \, p^2v(\varepsilon_2), \, 1{+}v(\varepsilon_2)$ and $1{+}\frac{1}{p^2}(\frac{2}{e}{+}\frac{p}{p^2-1})$ respectively. Note that $v(\tfrac{2}{e}\alpha^{-2}\pi_e^4(\zeta-\zeta^p)\gamma^{2p^3-p^2}) \neq v(pu_1\gamma^{h_1})$ since only one has denominator coprime to $p$.
    Moreover, if $v(p\varepsilon_2) \leq v(\varepsilon_2^{p^2})$ then $v(\varepsilon_2) \geq \frac{1}{p^2-1}$ which would then imply that both these valuations are greater than $1+\frac{1}{p^2-1}$ leading to $-\tfrac{2}{e}\alpha^{-2}\pi_e^4(\zeta-\zeta^p)\gamma^{2p^3-p^2} +pu_1\gamma^{h_1} \equiv 0 \,\, \bmod{p^{1+\frac{1}{p^2-1}}}$, which is a contradiction since the terms have distinct valuations and $v(\gamma^{h_1}) < \frac{1}{p^2-1}$ for $p > 2$ and $e \geq 3$.
    Hence either:
    \begin{enumerate}
    \item $v(\varepsilon_2^{p^2})=v(pu_1\gamma^{h_1})$ so $v(\varepsilon_2)=\frac{1}{p^2}(1+\frac{1}{p^2}(\frac{2}{e} + \frac{p}{p^2-1}))$ has denominator divisible by $p^4$; or
    \item $v(\varepsilon_2^{p^2})=v(\tfrac{2}{e}\pi_e^4\alpha^{-2}(\zeta-\zeta^p)\gamma^{2p^3-p^2})$ in which case $v(\varepsilon_2)=h_2v(\gamma)$.
    \end{enumerate}
    In case (ii), replacing $\varepsilon_2$ with $\varepsilon_3 + u_2\gamma^{h_2}$ in the above chain of congruences we get
    \begin{align*}
        0 \ 
        &\equiv \ -\tfrac{2}{e}\alpha^{-2}\pi_e^4(\zeta-\zeta^p)\gamma^{2p^3-p^2} - (\varepsilon_3+u_2\gamma^{h_2})^{p^2} +p(\varepsilon_2 +u_1\gamma^{h_1}) & \mod p^{1+\frac{1}{p^2-1}} \\
        &\equiv \ -\tfrac{2}{e}\alpha^{-2}\pi_e^4(\zeta-\zeta^p)\gamma^{2p^3-p^2} - \varepsilon_3^{p^2} -(u_2\gamma^{h_2})^{p^2} +p(\varepsilon_2 +u_1\gamma^{h_1}) & \mod p^{1+\frac{1}{p^2-1}} \\
        &\equiv \ -\varepsilon_3^{p^2} +p(\varepsilon_2 +u_1\gamma^{h_1}) & \mod p^{1+\frac{1}{p^2-1}}&,
    \end{align*}
    where the second congruence follows from \Cref{lemma: freshmans dream} and the last from \Cref{lemma: u1 and u2}(ii).
    Since $v(\varepsilon_2)>v(\gamma^{h_1})$, we have $v(\varepsilon_3^{p^2})=1+v(\gamma^{h_1})$ and hence
    \[
    v(\varepsilon_3)=\frac{1}{p^4} \left(p^2+\frac{2}{e}+\frac{p}{p^2-1} \right)
    \]
    has denominator divisible by $p^4$.
\end{proof}

\subsection{The case $\mathbf{k>2}$.}

Assume now that $k>2$. We set 
\begin{itemize}
    \item $h_1 := p^{2k-4}(p^2-1)\left(\frac{2}{e} + \frac{p}{p^2-1}\right) = \frac{2p^{2k-4}(p^2-1)}{e} + p^{2k-3}$;
    \item $h_{i+1} := p^{2k-4}(p^2-1) + \frac{1}{p^2}h_i$ for $i \ge 1$;
    \item $u_i := (-1)^{k} \xi^{-1} u (\zeta - \zeta^p)$ for $i \geq 1$, where $u := \alpha^{-1}p^{2-k} \in \mathbb{Z}_p^\times$ (note that $u_i = u_1$ for all $i \ge 1$);
    \item $\varepsilon_{i+1} := \varepsilon_i - u_1\gamma^{h_i}$ for $1 \le i \le k-1$.
\end{itemize}
These are the quantities we will use in  \Cref{eq: epsilons and hs}.

\begin{remark}\label{rmk: h_i inequalities}
    Note that $h_i=\frac{1}{p^{2i-2}}(h_1-p^{2k-2}) + p^{2k-2}$ for $i \geq 1$. As $h_1 < p^{2k-2}$, it is immediate that the sequence $h_i$ is strictly increasing and bounded by $p^{2k-2}$. Moreover, $h_1 \in \Z$ and $v(h_i)=2(k-i-1)$ so $h_1,\ldots,h_{k-1}$ are integers, as is $p^2h_k$. In particular, this implies $\varepsilon_i$ are well-defined elements of the splitting field of $g_k$ for all $1 \leq i \leq k$.
\end{remark}

\begin{proposition}\label{prop: recursive congruence}
    Suppose $p>3$. With the notation above, we have
    \[
    v(\varepsilon_i) = h_iv(\gamma) \quad \text{for all } 1 \leq i \leq k.
    \]
    In particular $v(\varepsilon_k)$ has denominator divisible by $p^{2k}$.
\end{proposition}

\begin{proof}
    We first claim that the congruence
    \begin{equation}\label{eq: eps congruence}
    u_1 \gamma^{p^2h_i} - \varepsilon_i^{p^2} + p\varepsilon_i \equiv 0 \mod p^{1+\frac{2}{ep} + \frac{1}{p^2-1}}
    \end{equation}
    holds for every $1 \leq i \leq k$ (note the exponents $p^2h_i$ are integers by \Cref{rmk: h_i inequalities}). We defer the proof of this momentarily, and instead first use it to deduce the valuation of $\varepsilon_i$. 
    By  \Cref{rmk: h_i inequalities}, we have $h_i < p^{2k-2}$ and hence $v(\gamma^{p^2h_i})<1+\frac{1}{p^2-1}$. If $v(\varepsilon_i) \geq \frac{1}{p^2-1}$, then
    \[
    v(\varepsilon_i^{p^2}) \geq v(p\varepsilon_i) \geq 1 + \frac{1}{p^2-1} > v(u_1\gamma^{p^2h_i})
    \]
    which cannot satisfy the congruence \eqref{eq: eps congruence} (noting that $1 + \frac{1}{p^2-1}< 1 + \frac{2}{ep}+\frac{1}{p^2-1}$) and hence $v(\varepsilon_i) < \frac{1}{p^2-1}$. Therefore $v(\varepsilon_i^{p^2})< v(p\varepsilon_i)<1+\frac{1}{p^2-1}$ so to obtain the cancellation we must have $v(\varepsilon_i^{p^2})=v(u_1\gamma^{p^2h_i})$, equivalently $v(\varepsilon_i)=h_iv(\gamma)$.
    In particular, if the congruence is true for $\varepsilon_k$, then since $v(h_k)=-2$ by \Cref{rmk: h_i inequalities}, we have
    \[
    v(v(\varepsilon_k))=v(h_kv(\gamma))=v(h_k)+v(v(\gamma))=-2 -(2k-2) = -2k
    \] 
    so the denominator of $v(\varepsilon_k)$ is divisible by $p^{2k}$ as required.
    
    It remains to prove the claimed congruence \eqref{eq: eps congruence}, which we do by induction on $i \leq k$. Suppose first that $i = 1$. By \Cref{lemma: epsilon1} we have
    \[
    (-1)^k\pi_e^2u(\zeta-\zeta^p)\gamma^{p^{2k-1}} - \varepsilon_1^{p^2} + p\varepsilon_1 \equiv 0 \mod{p^{1+\frac{2}{ep}+\frac{1}{p^2-1}}},
    \]
    hence it suffices to show that 
    \[
    u_1 \gamma^{p^2h_1} := (-1)^k\xi^{-1}u(\zeta-\zeta^p) \gamma^{p^{2k-2} \cdot\frac{2(p^2-1)}{e} + p^{2k-1}} \equiv (-1)^k\pi_e^2u(\zeta-\zeta^p)\gamma^{p^{2k-1}} \mod p^{1+\frac{2}{ep}+\frac{1}{p^2-1}}.
    \]
    The base case now follows from \Cref{lemma: congruences powers of gamma}, noting that $v(\alpha^{-1}) \ge 1$ (as $k>2$), and that $1+ \frac{2}{ep} + \frac{1}{p^2-1} < 1 + \frac{2}{p+1}$.

    We now suppose the congruence holds for $\varepsilon_i$ (so in particular $v(\varepsilon_i)=h_iv(\gamma)$) and prove the induction step: namely that the congruence also holds for $\varepsilon_{i+1}$ whenever $i \leq k-1$. 
    Recall that $\varepsilon_{i+1}=\varepsilon_{i}-u_1\gamma^{h_i}$ by definition. Using that $v(\varepsilon_{i})=h_iv(\gamma)$ and $h_i \geq h_1$ (\Cref{rmk: h_i inequalities}) we have
    \[
    p^2v(\varepsilon_{i+1}) \geq p^2v(\gamma^{h_i}) \geq p^2v(\gamma^{h_1}) \geq \frac{2}{e}+\frac{1}{p^2-1}.
    \]
    The freshman's dream (\Cref{lemma: freshmans dream}) therefore gives us that $\varepsilon_{i}^{p^2} \equiv \varepsilon_{i+1}^{p^2} + u_1^{p^2}\gamma^{p^2h_i} \mod p^{1+\frac{2}{ep}+\frac{1}{p^2-1}}$. Moreover, analogously to the first line of the proof of \Cref{lemma: u1 and u2}, $u_1^{p^2}\equiv u_1 \bmod p$, and hence $u_1^{p^2} \gamma^{p^2 h_i} \equiv u_1 \gamma^{p^2h_i} \mod p^{1+\frac{2}{ep}+\frac{1}{p^2-1}}$ since $v(\gamma^{p^2h_i}) \ge\frac{2}{e}+\frac{1}{p^2-1} > \frac{2}{ep} + \frac{1}{p^2-1}$ by above.
    Rewriting the congruence for $\varepsilon_i$ now yields
    \begin{equation*}
        0 \equiv u_1\gamma^{p^2h_i} - (\varepsilon_{i+1} + u_1\gamma^{h_i})^{p^2} + p\varepsilon_{i+1} + pu_1\gamma^{h_i} \equiv p u_1 \gamma^{h_i} - \varepsilon_{i+1}^{p^2} + p\varepsilon_{i+1} \mod p^{1+\frac{2}{ep}+\frac{1}{p^2-1}}.
    \end{equation*}
    To obtain the claimed congruence for $\varepsilon_{i+1}$, it remains to show that $pu_1 \gamma^{h_i} \equiv u_1\gamma^{p^2h_{i+1}} \mod p^{1+\frac{2}{ep}+\frac{1}{p^2-1}}.$ Since $v(\alpha^{-1}) \geq 1$, we have $p \equiv \gamma^{p^{2k-2}(p^2-1)} \mod p^{2 - \frac{1}{p^2}}$ by \Cref{lemma: congruences powers of gamma}, so substituting this in and using the recurrence relation for $h_i$ proves that $pu_1 \gamma^{h_i} \equiv u_1\gamma^{p^2h_{i+1}} \mod p^{1+\frac{2}{ep}+\frac{1}{p^2-1}}$ and hence also the claimed congruence \eqref{eq: eps congruence} as required.
\end{proof}

\section{Determining $\alpha$}\label{sect: parameter alpha}

One of our main results, \Cref{thm: more precise description padic Galreps}, is formulated in terms of the deformation parameter $\alpha$ associated with an elliptic curve $E$ with potentially supersingular reduction and semistability defect $e \geq 3$. 
In this section, we consider the problem of determining $\alpha$ (or at least its valuation) in order to give a more computationally effective version of \Cref{thm: more precise description padic Galreps} where $\alpha$ is replaced by the $j$-invariant. To the best of our knowledge, this constitutes a novel and explicit link between the concrete Weierstrass model of $E$ and the notoriously opaque world of $p$-adic Hodge theory.

We shall fix the following notation throughout this section.

\begin{notation*}
\phantom{~} \\
\begin{tabular}{cl}
$p>3$ & rational prime \\
$\tilde{E}/\mathbb{F}_p$ & fixed supersingular elliptic curve \\
$e$ & $\in \{3,4,6\}$ such that $e<p-1$ and $e \mid p+1$ \\
$\pi_e$ & $\sqrt[e]{-p}$ \\
$L,\mathcal{O}_L$ & $\Qp(\pi_e)$ and its valuation ring.
\end{tabular}
\end{notation*}

Moreover, since $p>3$, we shall assume all elliptic curves are in short Weierstrass form for convenience. We will study the set of lifts of $\tilde{E}/\mathbb{F}_p$ to $L$, identifying those which can be defined over $\Qp$ with semistability defect $e$ and present an algorithm to compute the associated parameter $\alpha$ in terms of coefficients of formal logarithms (via another parameter $\beta$), building on work of Volkov and Kawachi \cite{Kawachi}. The following definition encapsulates the properties of the elliptic curves we will consider.

\begin{definition}
Let $F/\Qp$ be a totally ramified extension. We say that an elliptic curve $E/F$ is a \emph{lift} of $\tilde{E}$ if $E$ has good reduction with special fibre isomorphic to $\tilde{E}$. We moreover call an elliptic curve $E/\Qp$ a \emph{potential $e$-lift} if $E/\Qp$ has semistability defect $e$ and $E/L$ is a lift of $\tilde{E}$.
\end{definition}

By \Cref{thm: Volkov parametrisation}, to any potential $e$-lift of $\tilde{E}/\F_p$ we can attach a deformation parameter $\alpha \in \mathbb{P}^1(\Q_p)$. Our objective is to describe $\alpha$, and especially its valuation, as explicitly as possible. The result below summarises the key statements regarding $v(\alpha)$. Whilst we do not state it below, one can determine $\alpha$ itself (via $\beta$) using coefficients of the formal logarithm as follows: compute $\beta$ to any given precision with \Cref{lemma: yasuda} and \Cref{prop: beta algorithm}; $\alpha$ is then given in terms of $\beta$ in \Cref{prop: beta parameter}(iv).

\begin{theorem}
\label{thm: valuation alpha overview}
Let $E/\Qp$ be a potential $e$-lift of $\tilde{E}/\mathbb{F}_p$ with $j$-invariant $j$, Volkov parameter $\alpha$ and minimal discriminant $\Delta$. Then:
\begin{enumerate}
    \item $j \in \{0,1728\}$ if and only if $\alpha \in \{0,\infty\}$;
    \item The valuation of $\alpha$ is given by the following table:

\begin{center}
 \begin{tabular}{c|ccc}
     & $e=3$ & $e=4$ & $e=6$ \\ \hline 
     $v(\Delta) < 6$  & $\frac{1}{3}(5-v(j))$ & $\frac{1}{2}(3-v(j-1728))$ & $\frac{1}{3}(4-v(j))$ \\[0.1in]
     $v(\Delta) > 6$ & $\frac{1}{3}(2+v(j))$ & $\frac{1}{2}(1+v(j-1728))$ & $\frac{1}{3}(1+v(j))$
 \end{tabular}
\end{center}

    \item $E$ has a canonical subgroup if and only if $v(\alpha)=1$ if and only if
    $\begin{cases}
        v(j) \in \{1,2\} &\text{if } e \in \{3,6\}; \\
        v(j-1728)=1 &\text{if } e=4. 
    \end{cases}$
\end{enumerate}
\end{theorem}

\begin{proof}
(i) follows upon combining \Cref{prop: beta parameter}(ii) and (iv). Parts (ii) and (iii) are \Cref{proposition: formula valuation alpha}(i) and (ii) respectively.
\end{proof}

\subsection{Lifts and Volkov's $\beta$}

We begin by relating Volkov's parameter $\alpha$ to another parameter $\beta$ which classifies lifts to $L$ of a given supersingular curve $\tilde{E}/\F_p$. Whilst we will discuss all possible lifts later (\Cref{prop: Dieudonne mod}), we temporarily restrict to lifts which are base changes of potential $e$-lifts (so $\alpha$ is defined).

\begin{proposition}\label{prop: beta parameter}
Let $\tilde{E}/\mathbb{F}_p$ be a supersingular elliptic curve. Fix $e \in \{3,4,6\}$ and suppose $e<p-1$ and $e \mid p+1$.
\begin{enumerate}
    \item The set of lifts $E/L$ of $\tilde{E}/\F_p$ which are the base change of potential $e$-lifts is parameterised by
    \[
    \beta \in  \Z_p\pi_e \, \cup \, \Z_p\pi_e^{e-3};
    \]
    denote by $E_\beta / L$ the lift corresponding to $\beta$. 
    \item A lift $E_{\beta}$ has $j$-invariant $0$ or $1728$ if and only if $\beta=0$.
   \item
   For a given $\beta$ as in~(i), consider the potential $e$-lifts whose base change to $L$ is isomorphic to $E_\beta$. 
\begin{enumerate}
    \item If $\beta \neq 0$ and $e \neq 4$, there exists a unique such elliptic curve $E_\beta^{\varepsilon}/\Q_p$, where $\varepsilon \in \{\pm 1\}$, with $\varepsilon = 1$ if and only if $\beta \in \Z_p \pi_e$.

    \item Otherwise, there are exactly two such elliptic curves $E_\beta^{\varepsilon}/\Q_p$ (with $\varepsilon \in \{\pm 1\}$), which are ramified twists of one another.
\end{enumerate}
    In both cases, letting $\Delta$ be the minimal discriminant of $E_\beta^{\varepsilon}/\Q_p$, we have $\varepsilon=1$ if and only if $v(\Delta)<6$, and $\varepsilon=-1$ if and only if $v(\Delta)>6$.
    \item For a potential $e$-lift $E_\beta^{\varepsilon}/\Qp$, set
    \[
    \alpha = \begin{cases}
        \infty & \qquad \text{if } \beta=0 \text{ and }\varepsilon=1; \\
         0&  \qquad \text{if } \beta=0 \text{ and } \varepsilon=-1;
    \end{cases}
    \qquad \text{else set }
    \alpha = \begin{cases}
        -p \cdot (\beta/\pi_e)^{-1} & \qquad \text{if } \varepsilon=1; \\
        -p \cdot \beta/\pi_e^{e-3} &  \qquad \text{if } \varepsilon=-1.
    \end{cases}
    \] 
    The filtered $(\varphi,G_{K/\Qp})$-module corresponding to $E_{\beta}^{\varepsilon}$ is $D_\alpha$  (cf.~\Cref{def: D alpha}).
\end{enumerate}
\end{proposition}

\begin{proof}
(i)-(ii) are contained in \cite[Proposition 6 on p.~91]{Volkov}. (iii) follows from \cite[p.~92, Remarque 1]{Volkov}, except for the determination of $\varepsilon$ in terms of $v(\Delta)$, which will be covered by \Cref{prop: epsilon via discriminant} below. Finally, (iv) is inside the proof in \cite[p.~95]{Volkov}: Volkov describes the relevant filtered modules as 
\[
D_\alpha = \begin{cases}
        D_{\operatorname{pc}}^*(e; 0; -p \cdot (\beta/\pi_e)^{-1}) = D_{-p \cdot (\beta/\pi_e)^{-1}} & \qquad \text{if } \varepsilon=1; \\
        D_{\operatorname{pc}}^*(e; 0; -p \cdot \beta/\pi_e^{e-3}) = D_{-p \cdot \beta/\pi_e^{e-3}} &  \qquad \text{if } \varepsilon=-1. \qedhere
    \end{cases}
\]
\end{proof}

\begin{remark}\label{uniqueness of alpha beta}
Note that different values of $\beta$ may correspond to the same $\alpha$: every $\alpha$ with $v(\alpha)=1$ may arise both from $\beta=-(\alpha/p) \cdot \pi_e^{e-3}$ and $\beta'=-(p/\alpha) \pi_e$. This is since for such a $\beta$ the elliptic curve $E_\beta^{-1}$ has a canonical subgroup of order $p$ (\Cref{rmk: alpha role}(ii)), and so carries a degree $p$ isogeny to another elliptic curve $E'/\Q_p$, which is in general not isomorphic to $E_\beta^\varepsilon$. Using \cite[Remarque~2, p.~93]{Volkov}, which proves the existence of a $\Qp$-isogeny between $E_{u \pi_{e}^{e-3}}$ and $E_{u^{-1} \pi_e}$ for every $u \in \mathbb{Z}_p^\times$ (in our case, $u=-\alpha/p$), one can check that $E' = E_{\beta'}^{+1}$. This isogeny induces an isomorphism of Galois modules $V_pE^{-1}_\beta \cong V_pE^{+1}_{\beta'}$, which is compatible with both curves having the same value of $\alpha$. 
\end{remark}

Whilst \Cref{prop: beta parameter} connects $\alpha$ and $\beta$, it is still unclear how to determine $\beta$ for a given lift. To do this, we will give a more technical description via an integral Dieudonn\'{e} module, which we will then leverage for a direct connection with formal logarithms in the following subsections.

\subsection{Formal logarithms and $E_{\mathrm{CM}}$}

We now aim to relate $\beta$ to the formal logarithm of our lift $E/L$. To enable this, we first fix a particular CM lift of $\tilde{E}/\mathbb{F}_p$ and then recap formal logarithms. Recall from \Cref{prop: e divides p plus one} that $j(E)$ is congruent to $0$ or $1728$ modulo $p$, according to whether $e \in \{3,6\}$ or $e=4$.

\begin{definition}
We fix a lift $E_{\mathrm{CM}}/\Qp$ of $\tilde{E}$ with good supersingular reduction and $j(E_{\mathrm{CM}}) \in \{0,1728\}$. As $\tilde{E}$ is in short Weierstrass form, we may assume the same for $E_{\mathrm{CM}}$ and write
\[
E_{\mathrm{CM}}:\quad
\begin{cases}
y^2=x^3+A_{\mathrm{CM}} x & \text{if } e=4,\\
y^2=x^3+B_{\mathrm{CM}}  & \text{if } e \in \{3,6\},
\end{cases}
\]
for suitable $A_{\mathrm{CM}},B_{\mathrm{CM}} \in \Z_p$.
\end{definition}

\begin{definition}\label{def: formal log}
Let $F$ be an algebraic extension of $\Q_p$ and let $E/F$ be an elliptic curve in short Weierstrass form. Let $\omega=\frac{dx}{2y}$ be the standard invariant differential of $E$ and take $t=-x/y$ as parameter at the origin. The choice of $t$ induces an identification of the completed local ring of $E$ at the origin with $\mathcal{O}_F[[t]]$, and hence of the formal group of $E$ with the formal spectrum of $\mathcal{O}_F[[t]]$. We may then write (the pullback to $\mathcal{O}_F[[t]]$ of) $\omega$ as $g(t)\,dt$ with $g(t)\in 1 +t \mathcal{O}_F[[t]]$.
The formal logarithm of $\hat{E}$ is the unique formal power series
$\log_{\widehat{E}}(t)\in F[[t]]$ that solves the equations
\[
\log_{\widehat{E}}(0)=0\qquad \text{and} \qquad \frac{d}{dt}\log_{\widehat{E}}(t)=g(t).
\]
In particular, $\log_{\widehat{E}}(t)=t+O(t^2)$.
\end{definition}

\begin{definition}
    Take $E=E_{\mathrm{CM}}/\Q_p$ in the previous definition. We shall write $\Gamma$ and $\log_\Gamma \in \Q_p[[t]]$ for the formal group of $E$ over $\Z_p$ and the formal logarithm of $E$ respectively, in line with the notation of \cite{Kawachi}.
\end{definition}

We collect here two results about formal logarithms that will be useful later.
\begin{lemma}[Yasuda, {\cite[Corollary 6(a)]{Yasuda}}]\label{lemma: yasuda}
Let $R$ be a $\Q$-algebra and $E/R$ be given by $y^2=x^3+Ax+B$ such that $4A^3+27B^2 \in R^{\times}$.
Let $t=-x/y$ and $\log_{\widehat E}(t) \in R[[t]]$ be the corresponding formal logarithm. Then
\[
\log_{\widehat E}(t)=\sum_{m,n\ge 0}
{2m+3n \choose m+2n, m, n} \,A^mB^n\,
\frac{t^{4m+6n+1}}{4m+6n+1}.
\]
\end{lemma}

\begin{lemma}\label{lemma: odd-killed}
Write $\log_\Gamma(t)=\sum_{n\ge 0} d_{\mathrm{CM},n} t^{n}$. Then:
\begin{enumerate}
\item $d_{\mathrm{CM}, p^{2k+1}}=0$ for all $k \ge 0$.
\item $v(d_{\mathrm{CM},p^{2k}}) = -k$ for all $k \ge 0$.
\item If $\zeta_e$ is an $e$-th root of unity, then $\log_\Gamma(\zeta_e t)=\zeta_e \log_\Gamma(t)$.
\end{enumerate}
\end{lemma}

\begin{proof}
We treat the case $e \in \{3,6\}$; the case $e=4$ is analogous. Now $E_{\mathrm{CM}}/\Qp: y^2=x^3+B_{\mathrm{CM}}$ and $p \equiv -1 \bmod{6}$ since $E_{\mathrm{CM}}$ has supersingular reduction. By \Cref{lemma: yasuda} we have $\log_\Gamma(t) \in t\Qp[[t^6]].$ (i) and (iii) now follow from this, noting that $p \equiv -1 \bmod{6}$ and $e \mid 6$ respectively.

(ii)  As $E_{\mathrm{CM}}$ has good supersingular reduction, $\log_\Gamma$ is of type $T^2+p$ (cf.\ the first line of the proof of \cite[Proposition 2.1.1]{Kawachi}), so the valuations follow from \cite[Lemma 2.1.1]{Kawachi}.
\end{proof}

\subsection{The connection with formal logarithms}\label{subsect: connection with formal logs}

For a power series ring $R[[x_0,\ldots,x_n]]$, we will write $R[[x_0,\ldots,x_n]]_0$ for the ideal of series with zero constant coefficient.

\begin{definition}\label{def: MH Zp}
Following Fontaine \cite[p.~166]{Fontaine-p-divisible} and Kawachi \cite[p.~695]{Kawachi}, we set
\[
    MH_{\Z_p}(\Gamma)=\left\{ f \in \Q_p[[t]]_0 : \ \frac{d}{dt}f \in \Z_p[[t]] , \quad f(x)+f(y)-f(\Gamma(x,y))\in p\Z_p[[x,y]]_0\right\} / p\Z_p[[t]]_0,
\]
where $\Gamma(x,y)$ is the formal group law of the formal group $\Gamma$. We will write $[f(t)]$ for the class of $f(t) \in \Q_p[[t]]_0$ in $MH_{\Z_p}(\Gamma)$.
\end{definition}

\begin{remark}\label{rmk: log in MH}
The prototypical example of an element of $MH_{\Z_p}(\Gamma)$ is $f=[\log_{\Gamma}(t)]$, where we take $t=-x/y$ (cf. \Cref{def: formal log}); indeed $\log_{\Gamma}(x) + \log_{\Gamma}(y) - \log_{\Gamma}(\Gamma(x,y))=0$ by definition. 
\end{remark}

\begin{lemma}\label{lemma: adapted generator}
Let $M$ be the Dieudonné module of the $p$-divisible group $\tilde{E}(p)$ of $\tilde{E}$. The following hold.
\begin{enumerate}
    \item There is a continuous isomorphism of $\varphi$-modules $w : M \xrightarrow{\sim} MH_{\Z_p}(\Gamma)$, where $\varphi$ acts on $MH_{\Z_p}(\Gamma)$ by the formula
    \[
    \varphi([f(t)])=[f(t^p)].
    \]
    \item There exists a generator $e_1 \in M$ as a $\Z_p[\varphi]$-module such that $w(e_1)=[\log_\Gamma(t)]$. The elements $e_1, e_2 := \varphi(e_1)$ form a $\Z_p$-basis of $M$, and
    \[
    w(e_1) = [\log_\Gamma(t)], \quad w(e_2)=\varphi([\log_\Gamma(t)])=[\log_\Gamma(t^p)]
    \]
    is a $\Z_p$-basis of $MH_{\Z_p}(\Gamma)$. Moreover, $e_1$ is ``adapted'' to the automorphism group of $\tilde{E}$ in the sense of \cite[p.~91]{Volkov}, that is, $e_1 \otimes 1 \in M \otimes_{\Z_p} \Z_{p^2}$ is an eigenvector for the action of the automorphisms of $(\tilde{E})_{\F_{p^2}}$.
\end{enumerate}
\end{lemma}

\begin{proof}
Since $\tilde E$ is supersingular, \cite[p.~86]{VolkovArticle} gives that not only is $M$ free of rank $2$ over $\Z_p$, but it is also free of rank $1$ over $\Z_p[\varphi]$ where $\varphi$ satisfies $\varphi^2+p=0$ and every $x \in M \setminus \varphi M$ is a $\Z_p[\varphi]$-generator. As recalled in \cite[p.~695]{Kawachi}, \cite[III, Proposition 6.5]{Fontaine-p-divisible} gives a continuous isomorphism $w:M\xrightarrow{\sim} MH_{\Z_p}(\Gamma)$ of $\Z_p[\Frob]$-modules, where the Frobenius $\Frob$ acts by transport of structure. An easy calculation using the formulae in \cite[p.~695]{Kawachi} yields that the $\varphi$-action on $MH_{\Z_p}(\Gamma)$ is given by $\varphi([f(t)])=[f(t^p)]$; this shows (i).

Fix an automorphism of order $e$ of $\tilde{E}_{\F_{p^2}} = (E_{\mathrm{CM}})_{\F_{p^2}}$. It lifts uniquely to $\op{Aut}_{\Z_{p^2}}(E_{\mathrm{CM}}) \cong \Z[\zeta_e]^{\times}$, which (up to replacing it with its inverse) induces $t \mapsto \zeta_e t$ on the completion of the local ring at $0$ of $(E_{\mathrm{CM}})_{\Z_{p^2}}$. By \Cref{lemma: odd-killed}(iii) we thus see that $[\log_\Gamma(t)] \in MH_{\Z_p}(\Gamma) \otimes \Z_{p^2}$ is an eigenvector for the action of this automorphism (under the identification induced by $w$). Moreover, $[\log_\Gamma(t)] \not\in \varphi(MH_{\Z_p}(\Gamma))$, because any $[f(t^p)] \in \varphi(MH_{\Z_p}(\Gamma))$ satisfies $f(t^p) \equiv 0 \bmod ({p, t^2})$, while $\log_{\Gamma}(t) \equiv t \bmod ({p, t^2})$. Since every element of $M \setminus \varphi M$ is a $\Z_p[\varphi]$-generator, we choose the $\Z_p[\varphi]$-generator $e_1:=w^{-1}([\log_{\Gamma}(t)])$ hence $M=\Z_pe_1 \oplus \Z_p\varphi(e_1)$. To complete the proof of (ii), finally observe that $e_1 \otimes 1$ is an eigenvector for the action of the automorphisms of $(\tilde{E})_{\F_{p^2}} = E_{\mathrm{CM}} \times_{\Z_p} \F_{p^2}$, so $e_1$ is ``adapted".
\end{proof}

\begin{remark}
    The result of this lemma is related to \cite[Chapitre III, Proposition 4.3]{Fontaine-p-divisible}. When $M$ is the Dieudonn\'{e} module of the $p$-divisible group $\tilde{E}(p)$, Fontaine shows that the quotient $M/\varphi M$ is canonically isomorphic to the cotangent space of $\tilde{E}(p)$, which is a $1$-dimensional vector space generated by any non-zero invariant differential. In our description, we identify a distinguished line in $M$ itself: via the isomorphism $w$, it corresponds to the line generated by the class of the formal logarithm, whose differential is precisely the invariant differential.
\end{remark}

\begin{proposition}\label{prop: Dieudonne mod}
Let $\tilde{E}/\mathbb{F}_p$ be a supersingular curve and let $M$ be the Dieudonn\'{e} module associated to the $p$-divisible group $\tilde{E}(p)$ of $\tilde{E}/\mathbb{F}_p$. Define
\[
\mathcal{M} := M \otimes_{\Z_p} \mathcal{O}_L + \varphi M \otimes_{\Z_p} \pi_e^{1-e}\mathcal{O}_L
\]
and set $\mathbf e_1:=e_1\otimes 1$, $\mathbf e_2:=e_2\otimes \pi_e^{1-e}$, where $e_1$ is the generator described in \Cref{lemma: adapted generator}.
Then:
\begin{enumerate}
\item $\mathbf{e}_1, \mathbf{e}_2$ is an $\mathcal{O}_L$-basis of $\mathcal{M}$;
\item every lift  $E'/\mathcal{O}_L$ of $\tilde{E}$ corresponds to an $\mathcal{O}_L$-line $\mathcal{L}(\beta) = (\mathbf{e}_1 + \beta\mathbf{e}_2)\mathcal{O}_L \subset \mathcal{M}$  with $\beta \in \mathcal{O}_L$;
\item conversely, every $\mathcal{O}_L$-line $\mathcal{L}(\beta)$ for $\beta \in \mathcal{O}_L$ corresponds to some lift $E'/L$ of $\tilde{E}$;
\item in the notation of \Cref{prop: beta parameter}(i), $\mathcal{L}(\beta)$ corresponds to $E_\beta$.
\end{enumerate}
\end{proposition}

\begin{proof}
Part (i) is obvious from the definitions, using that $\pi_e^{-e}=-p^{-1}$; see also \cite[beginning of \S2.2]{Kawachi}. For the rest of the statement, see \cite[\S3.3.1.2 and Proposition 4, p.~83]{Volkov}.
\end{proof}

\begin{remark}\label{rmk: Kawachi vs Volkov}
We follow Volkov's choice of basis for $\mathcal{M}$ here, but note that Kawachi adopts a slightly different choice: they take the basis $\mathbf{e}_1^K=\mathbf{e_1}$ and $\mathbf{e}_2^K=e_2 \otimes \frac{\pi_e}{p} =e_2 \otimes -\pi_e^{1-e} = -\mathbf{e}_2$. Following our arguments through with Kawachi's convention, one arrives at a $\beta$-parameter which is the negative of Volkov's (and our) $\beta$-parameter.
\end{remark}

We next construct an isomorphism from our integral module $\mathcal{M}$ to an analogous version of $MH_{\Z_p}$ for curves over $L$. In what follows, we will work both with lifts $E/L$ and potential $e$-lifts $E/\Q_p$. In either case, we fix an integral model with good reduction for $E_L$
\begin{equation}\label{eq: integral model EL}
    E_L/L :y^2=x^3+Ax+B, \quad \text{where } \begin{cases}
        v(A)=0 & \text{ if } e=4; \\
        v(B)=0 & \text{ if } e \in \{3,6\}
    \end{cases}
\end{equation}
which reduces modulo $(\pi_e)$ to $\tilde{E}$.

\begin{definition}\label{def: MHOL} 
Let $E/L$ be a lift of $\tilde{E}$ with formal group law $\hat{E}$ (with respect to the fixed integral model of good reduction \eqref{eq: integral model EL}). Following \cite[p.~698]{Kawachi} and \cite[IV, \S 4.1]{Fontaine-p-divisible}, we define
\[
MH_{\mathcal{O}_L}(\hat{E}) := \left\{f \in L[[t]]_0 : \frac{d}{dt}f \in \mathcal{O}_L[[t]], \quad f(x)+f(y) -f(\hat{E}(x,y)) \in \pi_e\mathcal{O}_L[[x,y]]_0 \right\}/\pi_e\mathcal{O}_L[[t]]_0,
\]
{and write $[f(t)]$ for the class of $f(t) \in L[[t]]_0$ in $MH_{\mathcal{O}_L}(\hat{E})$.}
\end{definition}

\begin{remark}\label{rmk: MH indep of lift}
\phantom{~}
\begin{itemize}
\item When $e\leq p-1$, $MH_{\mathcal{O}_L}(\hat{E})$ only depends on $\tilde{E}$ and is independent of the choice of lift $E$; see \cite[IV, Proposition 4.1]{Fontaine-p-divisible} and the last line of \cite[IV.3]{Fontaine-p-divisible}. We will use $MH_{\mathcal{O}_L}((\hat{E}_{\mathrm{CM}})_{L})$, which for simplicity we denote $MH_{\mathcal{O}_L}(\Gamma)$.
\item Analogously to \Cref{rmk: log in MH}, $[\log_{\hat{E}}(t)]$ (with parameter $t$ induced by the choice of model \eqref{eq: integral model EL}) is an element of $MH_{\mathcal{O}_L}(\hat{E})$. 
\end{itemize}
\end{remark}

\begin{theorem}\label{thm: MH iso}
Recall the module $\mathcal{M}$ from \Cref{prop: Dieudonne mod} (with its basis $\mathbf{e}_1,\mathbf{e}_2$) that parameterises lifts of $\tilde{E}/\mathbb{F}_p$. Then:
\begin{enumerate}
    \item there is a continuous isomorphism of $\mathcal{O}_L$-modules $\tilde{w}: \mathcal{M} \xrightarrow{\sim} MH_{\mathcal{O}_L}(\Gamma)$;
    \item we have
    \[
    \tilde{w}(\mathbf{e}_1)=[\log_\Gamma(t)], \quad \tilde{w}(\mathbf{e}_2)= -[\tfrac{\pi_e}{p}\log_\Gamma(t^p)];
    \]
    \item if the line $\mathcal{L}(\beta) \subseteq \mathcal{M}$ corresponds to the lift $E/L$ (cf. \Cref{prop: Dieudonne mod}(iii)), then $\tilde{w}(\mathcal{L}(\beta))=\mathcal{O}_L \cdot [\log_{\hat E}(t)]$.
\end{enumerate}
\end{theorem}

\begin{proof}
(i)-(ii): The isomorphism $\isom$ is constructed in \cite[Lemma 2.2.3 and top of p.~699]{Kawachi} and the images of $\mathbf{e}_1, \mathbf{e}_2$ follow from \cite[Lemma 2.2.2]{Kawachi} (accounting for the sign change discussed in \Cref{rmk: Kawachi vs Volkov}). (iii) is \cite[Lemma 2.2.4]{Kawachi}.
\end{proof}

If $E/L$ is a lift of $\tilde{E}/\mathbb{F}_p$ then by \Cref{rmk: MH indep of lift} we get a class $[\log_{\hat{E}}(t)] \in MH_{\mathcal{O}_L}(\Gamma)$, which can be written as a linear combination of $[\log_\Gamma(t)]$ and $[\tfrac{\pi_e}{p}\log_{\Gamma}(t^p)]$ by \Cref{thm: MH iso}. We can lift this to get a relation of the corresponding power series which we now exploit.

\begin{corollary}\label{cor: linear combination formal logarithms}
Let $E/L$ be a lift of $\tilde{E}/\mathbb{F}_p$. Write
\begin{equation}\label{eq: formal log rel}
    \log_{\hat{E}}(t) \equiv a\log_{\Gamma}(t) + b\tfrac{\pi_e}{p}\log_{\Gamma}(t^p) \bmod{\pi_e\mathcal{O}_L[[t]]}.
\end{equation}
Then $a$ is a unit in $\mathcal{O}_L$ and the lift $E$ corresponds under $\tilde{w}$ to $\mathcal{L}(\beta) \subseteq \mathcal{M}$, where $\beta=-\frac{b}{a} \in \mathcal{O}_L$.
\end{corollary}

\begin{proof}
Since the formal logarithms have the form $\log_{\hat{E}}(t)=t+O(t^2)$, we can compare the $t$-coefficient to get $a \equiv 1 \bmod{\pi_e}$; in particular $a$ is a unit.
Now applying $\isom^{-1}$ to \eqref{eq: formal log rel} we have $\isom^{-1}([\log_{\hat{E}}(t)])=a\mathbf{e}_1 - b\mathbf{e}_2$
and hence by \Cref{thm: MH iso}(iii)
\[
\mathcal{L}(\beta) = \isom^{-1}(\mathcal{O}_L \cdot [\log_{\hat{E}}(t)]) = \mathcal{O}_L \cdot (a\mathbf{e}_1 - b\mathbf{e}_2) = \mathcal{O}_L \cdot (\mathbf{e}_1 - \tfrac{b}{a}\mathbf{e}_2),
\]
so the $\beta$-parameter is precisely $-\frac{b}{a}$ by definition of $\mathcal{L}(\beta)$.
\end{proof}

\subsection{Algorithmic determination of $\beta$}\label{subsect: beta algorithm}

Having related the $\beta$-parameter (and hence also $\alpha$) to the formal logarithm, we now use the explicit power series of $\log_{\hat{E}}$ and $\log_{\Gamma}$ to determine $\beta$ to any given precision.

\begin{proposition}\label{prop: beta algorithm}
Let $E/L$ be a lift of $\tilde{E}/\mathbb{F}_p$ and write $\log_{\hat{E}}(t)= \sum_{n \geq 0} d_nt^n$. For all $k \geq 0$, the $\beta$-parameter of $E$ satisfies 
\[
\beta \equiv - \frac{p}{\pi_e}\frac{d_{p^{2k+1}}}{d_{p^{2k}}} \bmod \pi_e^{ke+1}.
\]
In particular, its valuation is given by the formula
$
v(\beta)=\lim_{k\to\infty}\left(v(d_{p^{2k+1}})+(k+1)-\frac{1}{e}\right).
$
\end{proposition}

\begin{proof}
Throughout the proof, a congruence $x \equiv y \pmod{\pi_e^m \mathcal{O}_L}$ for $x, y \in L$ (not necessarily in $\mathcal{O}_L$) means as usual that the difference $x-y$ belongs to $\pi_e^m \mathcal{O}_L$.
By examining coefficients of \eqref{eq: formal log rel}, we have
\[
d_r = ad_{\mathrm{CM},r} + b\tfrac{\pi_e}{p}d_{\mathrm{CM},r/p} \,\, \bmod{\pi_e\mathcal{O}_L}, \qquad \text{where } d_{\mathrm{CM},r/p} =0 \text{ if } p\nmid r. 
\]
Taking $r$ to be a power of $p$ and using \Cref{lemma: odd-killed}(i) and (ii), we find that for all $k \geq 0$
\[
d_{p^{2k}} \equiv ap^{-k}u_{p^{2k}} \bmod{\pi_e}, \quad d_{p^{2k+1}} \equiv b\tfrac{\pi_e}{p}p^{-k}u_{p^{2k}} \bmod{\pi_e}
\]
where $u_{p^{2k}}:=d_{\mathrm{CM},p^{2k}} \cdot p^k \in \Z_p^{\times}$; equivalently
\[
a \equiv d_{p^{2k}}p^{k}u_{p^{2k}}^{-1} \bmod{\pi_e^{ke+1}}, \quad b \equiv d_{p^{2k+1}}\tfrac{p}{\pi_e}p^{k}u_{p^{2k}}^{-1} \bmod{\pi_e^{(k+1)e}}.
\]
Since $(k+1)e \geq ke+1$ and $a \in \mathcal{O}_L^{\times}$ (\Cref{cor: linear combination formal logarithms}), the congruence on $\beta$ follows. The congruence on $b$ also shows that $v(b) = \lim_{k \to \infty} \left(v(d_{p^{2k+1}}) + k+1-\frac{1}{e}\right)$; since $a$ is a unit, the valuation formula in the statement is an immediate consequence.
\end{proof}

\begin{remark}
The congruence $d_{p^{2k}} \equiv a p^{-k} u_{p^{2k}} \bmod{\pi_e}$ established in the course of the previous proof, combined with the fact that $a$ is a unit, shows in particular that $v(d_{p^{2k}})=-k$ for all $k \geq 0$. Similarly, since $v(b) \geq 0$ we obtain $v(d_{p^{2k+1}}) \geq \frac{1}{e} -(k+1)$.
We can then rewrite the congruence in the statement of \Cref{prop: beta algorithm} as $\beta \equiv -\frac{b'}{a'} \bmod{\pi_e^{ke+1}}$, where the numerator $b'=\frac{p^{k+1}}{\pi_e} \cdot d_{p^{2k+1}}$ is in $\mathcal{O}_L$ and the denominator $a'=p^k d_{p^{2k}}$ is in $\mathcal{O}_L^\times$. Since we take a congruence modulo $\pi_e^{ke+1}$, it suffices to compute $a',b'$ modulo $\pi_e^{ke+1}$. We discuss how to do this efficiently, starting with $b'$. Using \Cref{lemma: yasuda}, we may write
\[
b'=\frac{p^{k+1}}{\pi_e} d_{p^{2k+1}} = \frac{1}{p^{k} \pi_e} \sum_{\substack{m,n\ge 0 \\ 4m+6n+1 = p^{2k+1}}}
{2m+3n \choose m+2n, m, n} \,A^mB^n.
\]
The factor in front has valuation $-(k+\frac{1}{e})$, while the $(m,n)$ summand has valuation at least $m\,v(A)+n\,v(B)$. Since one of $A,B$ has valuation at least $\frac{1}{e}$ (using the model \eqref{eq: integral model EL}), only summands with $m \leq 2ke+1$ (if $v(A)>0$) or
$n \leq  2ke+1$ (if $v(B)>0$) can contribute modulo $\pi_e^{ke+1}$. 
In particular, the number of relevant summands grows linearly, rather than exponentially, in $k$. A similar comment applies to $a'$, which may be computed using the formula of \Cref{lemma: yasuda}, keeping only the $2ke$ summands of lowest valuation.
\end{remark}

\begin{remark}\label{rmk: invariance under strict isomorphisms}
Let $u\in 1+\pi_e\mathcal{O}_L$, and let $E'$ be obtained from $E$ by the scaling $(x,y)\mapsto(u^2x,u^3y)$, so that
$E' : (y')^2=(x')^3+A u^4 x'+B u^6$. It is easy to check that $[\log_{\widehat{E}'}]=\frac{1}{u}[\log_{\widehat{E}}]$ in $MH_{\mathcal{O}_L}(\widehat{E})$. By \Cref{thm: MH iso}(iii), the invariant $\beta$ depends only on the line generated by $[\log_{\widehat{E}}]$, and hence this equality shows that we obtain the same $\beta$ using either model. In other words, $\beta$ is invariant under strict isomorphisms, namely isomorphisms inducing the identity on the special fibre. This invariance is, of course, already implicit in Fontaine's theory.
\end{remark}

\begin{remark}\label{rmk: Kawachi error}
    The statement of \cite[Proposition 2.2.1]{Kawachi} may suggest a different way to compute $\beta$, by looking only at the coefficient of $t^p$ in $\log_{\widehat{E}}(t)$ and using that $b_p, \frac{\beta \pi_e}{p}$ are $\Z_p$-linearly independent. However, the proof of this proposition appears to be incorrect, since it deduces an exact equality of coefficients from an equality in $MH_{\mathcal{O}_L}(\widehat{E})$, where functions are only defined modulo $\pi_e \mathcal{O}_L[[t]]_0$. In particular, it is easy to check that \cite[Proposition 2.2.1]{Kawachi} leads to a formula for $\beta$ that is not invariant under strict isomorphisms.
\end{remark}

\subsection{Valuation of $\beta$}
Let $E/\Q_p$ be an elliptic curve with semistability defect $e \geq 3$.
Our purpose in this section is to give explicit formulae expressing the valuations of the $\alpha$ and $\beta$ parameters of $E$ in terms of its $j$-invariant. The following is the key result.

\begin{proposition}\label{prop: valuation beta}
Let $E/\Qp$ be a potential $e$-lift of $\tilde{E}/\mathbb{F}_p$ with $j(E)=j$. The valuation of the $\beta$-parameter of $E_L$ (cf.~\Cref{prop: Dieudonne mod}) is given by
\[
v(\beta) = 
\begin{cases}
\begin{array}{lll}
\frac{1}{3}v(j) - \frac{1}{e} & \quad & \text{ if $e \in \{3,6\}$} \\
\frac{1}{2}v(j-1728) - \frac{1}{e} & \quad & \text{ if $e =4$.}
\end{array}
\end{cases}
\]
\end{proposition}

In order to prove \Cref{prop: valuation beta}, we will need the following technical estimate on the coefficients of the formal logarithm of $E_L$ corresponding to the choice of model \eqref{eq: integral model EL}.

\begin{proposition}\label{prop: valuation coefficients formal log}
Let $\log_{\hat{E}_L}(t)=\sum_{r \geq 1} d_r t^r$ be the formal logarithm associated with the good reduction model \eqref{eq: integral model EL} of $E_L$. Suppose that $e \mid p+1$ and $e<p-1$. Then for every $k\ge 0$ we have
\[
v\bigl(d_{p^{2k+1}}\bigr)=-(k+1)+v(A) \text{ for } e\in\{3,6\},
\; \text{ respectively } \;
v\bigl(d_{p^{2k+1}}\bigr)=-(k+1)+v(B) \text{ for }e=4.
\]
\end{proposition}

\begin{proof}
Recall that by \Cref{lemma: yasuda} we have
\begin{equation}\label{eq: dp2k1-mn-sum}
d_{p^{2k+1}}
=
p^{-(2k+1)}
\sum_{\substack{m,n\ge 0\\ 2m+3n=\frac{p^{2k+1}-1}{2}}}
\binom{\frac{p^{2k+1}-1}{2}}{m,\;n,\;m+2n}\,A^mB^n.
\end{equation}
Our strategy is to show that there is a unique term of minimal valuation -- this occurs when $n=1$ (resp.~$m=1$) when $e=4$ (resp.~$e \in \{3,6\}$) -- and then compute its valuation. We only treat the case $e=4$ as the case $e \in \{3,6\}$ is analogous. 

Suppose then that $e=4$ so $v(B)>v(A)=0$ and set $N=\frac{p^{2k+1}-1}{2}$ which is odd since $p \equiv -1 \bmod{4}$. The equality $2m+3n=N$ implies $n$ is odd and the admissible options are $(m_j,n_j):=(m_0-3j,1+2j)$ for $0\leq j \leq \lfloor m_0/3 \rfloor$, where $m_0=\frac{N-3}{2}$.
For $j$ in this interval, write $T_j:=\binom{N}{m_j,\;n_j,\;m_j+2n_j}$; by definition we have
\[
T_j = \frac{m_0!(m_0+2)!}{(m_0-3j)!(m_0+j+2)!} \cdot \frac{1}{(1+2j)!}T_0.
\]
We claim that $v(T_j) \geq v(\frac{1}{(1+2j)!}T_0)$. We show this by induction. The base case $j=0$ is trivial, and for the induction step it suffices to show $v(T_{j+1} \cdot (3+2j)!) \geq v(T_j \cdot (1+2j)!)$, which is implied by $v(m_0-3j-2) \geq v(m_0+j+3)$. The identity
\[
3(m_0+j+3) + (m_0-3j-2)= p^{2k+1}
\]
shows that, since both summands on the left-hand side are strictly between $0$ and $p^{2k+1}$, they must have the same valuation, which completes the induction.

Hence, if $j>0$, then since $v(A)=0$, we have 
\[
v(T_jA^{m_j}B^{1+2j})  \geq v(T_0 A^{m_0} B) +2jv(B) - v((1+2j)!) > v(T_0A^{m_0}B^{n_0})
\]
where the last inequality follows using the bounds $v(B) \geq \frac{1}{e}=\frac{1}{4}$ and $v((1+2j)!) < \frac{1+2j}{p-1}$. This shows that \eqref{eq: dp2k1-mn-sum} has a unique summand of minimal valuation, corresponding to $j=0$, that is, $(m,n)=(m_0, 1)$.

Therefore $v(d_{p^{2k+1}})=-(2k+1) + v(T_0)+v(B)$ so it remains to prove $v(T_0)=k$. Now by expansion, we see $T_0=N\binom{N-1}{m_0}$ with $v(N)=0$ and so we compute the valuation of the binomial coefficient using Kummer's theorem: letting $S(n)$ be the sum of the $p$-adic digits of $n$, we have $v \left( \binom{N-1}{m_0} \right)=\frac{1}{p-1}(S(m_0)+S(N+1-m_0)-S(N-1))$.
One can check that the base $p$ expansions are:
\begin{eqnarray*}
    N-1 &=& \tfrac{p-3}{2} \cdot p^0+\sum_{i=1}^{2k} \tfrac{p-1}{2} \cdot p^i; \\
    N-1-m_0 &=& \tfrac{p+1}{4}\cdot p^0 + \sum_{i=1}^k \tfrac{p-3}{4} \cdot p^{2i} + \sum_{i=1}^k (p-\tfrac{p+1}{4}) \cdot p^{2i+1}; \\
    m_0 &=& (\tfrac{p+1}{4} -2)\cdot p^0 + \sum_{i=1}^k \tfrac{p-3}{4} \cdot p^{2i} + \sum_{i=1}^k (p-\tfrac{p+1}{4}) \cdot p^{2i+1}.
\end{eqnarray*}

From this, it is straightforward to compute the sum of the $p$-adic digits of $N-1, N-1-m_0$, and $m_0$, and Kummer's theorem gives $v(T_0)=k$ as desired.
\end{proof}

\begin{remark}\label{rmk: relation to Hasse invariant}
    Applying \Cref{prop: valuation coefficients formal log} with $k=0$ gives in particular
    \[
    v(pd_p)=
    \begin{cases}
        v(A), \qquad \text{if $e \in \{3,6\}$} \\
        v(B), \qquad \text{if $e=4$}.
    \end{cases}
    \]
    Combining \Cref{lemma: yasuda} and the development of $(x^3+Ax+B)^{\frac{p^n-1}{2}}$ using the multinomial theorem, one sees that the product $p^n d_{p^n}$ coincides with the coefficient of $x^{p^n-1}$ in $\left( x^3+Ax+B \right)^{\frac{p^n-1}{2}}$. For $n=1$, the reduction modulo $p$ of the coefficient of $x^{p-1}$ in $\left( x^3+Ax+B \right)^{\frac{p-1}{2}}$, that is $pd_p$, is known as the \textit{Hasse invariant} of the elliptic curve $E$. This suggests that the ratios $p d_{p^{2k+1}}/d_{p^{2k}}$ should have an interpretation as higher analogues of the Hasse invariant. Note that by \Cref{prop: beta algorithm} these ratios converge $p$-adically to $-\pi_e \beta$.
\end{remark}

\begin{corollary}\label{cor: formula for valuation beta}
    With the good reduction model of $E/L$ given in \eqref{eq: integral model EL}, we have 
    \[
    v(\beta) = -\frac{1}{e} + \begin{cases}
        v(A), \qquad \text{ if }e \in \{3,6\}\\
        v(B), \qquad \text{ if }e =4
    \end{cases}
    \]
    Letting $H \in \mathcal{O}_L$ be the coefficient of $x^{p-1}$ in $(x^3+Ax+B)^{(p-1)/2}$, we can rewrite this as
    \[
    v(\beta) = v(H) - \frac{1}{e}.
    \]
\end{corollary}

\begin{proof}
    From \Cref{prop: beta algorithm} we have $v(\beta) = \lim_{k \to \infty} \left(v(d_{p^{2k+1}}) + (k+1) - \frac{1}{e}\right)$.
    \Cref{prop: valuation coefficients formal log} shows that $v(d_{p^{2k+1}}) + (k+1) - \frac{1}{e}$ is independent of $k$ and equal to the formula given in the statement. The equality $H=pd_p$ and hence the formula involving $H$ follow from \Cref{rmk: relation to Hasse invariant}.
\end{proof}

We can now prove our formula for $v(\beta)$.

\begin{proof}[Proof of \Cref{prop: valuation beta}]
\Cref{cor: formula for valuation beta} gives $v(\beta)=v(A)-\frac{1}{e}$ (resp.~$v(\beta)=v(B)-\frac{1}{e}$) when $e \in \{3,6\}$ (resp.~$e=4$), so it remains to express the valuation of the coefficients via the $j$-invariant, noting that the discriminant $\Delta$ is a unit by assumption as the model of $E_L$ given by \eqref{eq: integral model EL} has good reduction.
The identities
\[
j = -1728 \cdot \frac{(4A)^3}{\Delta} \qquad \text{and} \qquad j - 1728 = 1728 \cdot \frac{16 \cdot 27B^2}{\Delta}
\]
now immediately give $v(A)=\frac{1}{3}v(j)$ and $v(B)=\frac{1}{2}v(j-1728)$ as claimed.
\end{proof}

\subsection{Volkov's $\alpha$ and the $p$-adic image}

With everything in place, we are finally in a position to completely describe $\alpha$ itself for a potential $e$-lift. A priori, it appears that one should compute $\beta$ in order to determine the sign $\varepsilon$, especially in the case $e=4$ where knowing $v(\beta)$ isn't sufficient. We first show that this can be bypassed via the discriminant and conclude the proof of \Cref{prop: beta parameter}(iii).

\begin{proposition}\label{prop: epsilon via discriminant} 
Let $E/\Q_p$ be a potential $e$-lift with minimal discriminant $\Delta$ and write $E \cong E^\varepsilon_\beta$ as in \Cref{prop: beta parameter}(iii), for some $\beta \in \mathbb{Z}_p \pi_e \cup \mathbb{Z}_p \pi_e^{e-3}$ and some $\varepsilon \in \{ \pm 1 \}$. Then $\varepsilon=1$ if and only if $v(\Delta)<6$. 
In particular, the Volkov parameter is $\alpha = \begin{cases}
\begin{array}{cc}
    -p \cdot (\beta/\pi_e)^{-1} & \text{ if } v(\Delta)<6 \\
    -p \cdot \beta/\pi_e^{e-3} & \text{ if } v(\Delta)>6.
    \end{array}
\end{cases}$
\end{proposition}

\begin{proof}
   We begin by recalling the action of $\Gal(K/\Q_{p^2})$ on the special fibre of a good reduction model (equivalently, Néron model) of $E_K$; see also \cite[p.~497]{MR236190}.
    The Galois group $\Gal(K/\Q_{p^2})$ acts on $E_K = E_{\Q_{p^2}} \times_{\Q_{p^2}} K$ via its action on $K$. By functoriality of the Néron model, it also acts on the Néron model $\mathcal{E}$ of $E_K$, and the map $\mathcal{E} \to \operatorname{Spec} \mathcal{O}_K$ is compatible with the Galois action on both schemes. Since $\operatorname{Gal}(K/\mathbb{Q}_{p^2})$ acts trivially on the residue field $\mathcal{O}_K/(\pi_e) \cong \mathbb{F}_{p^2}$, it acts on the special fibre of $\mathcal{E}$, which may be canonically identified with $(\tilde{E})_{\mathbb{F}_{p^2}}$. 

    By following Volkov's argument \cite[p.~95]{Volkov}, one sees that $\varepsilon$ is determined as follows: taking $\tau_e \in \Gal(K/\Q_{p^2})$ (where $\tau_e$ is defined by $\tau_e(\pi_e)=\zeta_e\pi_e$) in the above construction, we obtain an induced automorphism of $(\tilde{E})_{\mathbb{F}_{p^2}}$, which in turn acts on its Dieudonné module. The sign $\varepsilon$ is defined by the equality $\tau_e (e_1 \otimes 1) = \zeta_e^{\varepsilon} (e_1 \otimes 1)$, where $e_1$ is the adapted generator of \Cref{lemma: adapted generator}. Via the isomorphism of \Cref{thm: MH iso}, it thus suffices to understand the action of $\tau_e$ on $\log_\Gamma(t)$. We now do this in terms of a minimal model $y^2 = x^3+Ax+B$ of $E/\Q_p$ with discriminant $\Delta$.
    
    We give details in the case $e=4$ and briefly sketch the differences for the remaining cases $e \in \{3, 6\}$. Note that when $e=4$ the valuation of $\Delta$ can only be $3$ or $9$ (see \cite[Table 15.1]{MR2514094}), and that these two cases correspond to $v(A)=1$ and $v(A)=3$ respectively (with $v(B^2) > v(A^3)$ in both cases).

    Let $u := \pi_e^{-v(A)}$. A good reduction model of $E_K$ is given by $\mathcal{E} : z^2 = w^3 + Au^4 w + Bu^6$, with the isomorphism given by 
    \[
    \begin{array}{cccc}
    \psi :& E_K & \to & \mathcal{E} \\
     & (x,y) & \mapsto & ( u^2x, u^3y).
    \end{array}
    \]
    The action of $\tau_e$ on the special fibre of $\mathcal{E}$ may be computed by reducing modulo $\pi_e$ the composition
\[
\begin{array}{cccc}
{}^{\tau_e} \psi \circ \psi^{-1} : & \mathcal{E}  & \to & {}^{\tau_e} \mathcal{E} \\
& (w, z) & \mapsto & (\zeta_e^{-2v(A)}w, \zeta_e^{-3v(A)}z),
\end{array}
\]
where ${}^{\tau_e} \mathcal{E} : z^2 = w^3 + Au^4 w + B\zeta_e^{-6} u^6$. Note that the special fibres of $\mathcal{E}$ and ${}^{\tau_e}\mathcal{E}$ coincide, because $v(Bu^6)=v(B\zeta_e^{-6} u^6)=v(B) - \frac{6}{4}v(A)>0$. Thus, $\tau_e$ induces $(w,z) \mapsto (\zeta_e^{-2v(A)}w, \zeta_e^{-3v(A)}z)$ on the special fibre $(\tilde{E})_{\mathbb{F}_{p^2}}$, and hence sends the parameter $t = -w/z$ to $-\zeta_e^{-2v(A)}w / \zeta_e^{-3v(A)}z =\zeta_e^{v(A)} t$. 

By \Cref{lemma: odd-killed}(iii), we find that $\tau_e$ acts on $\log_\Gamma(t)$ (hence on $e_1 \otimes 1$) as $\tau_e \log_\Gamma(t) = \zeta_e^{v(A)} \log_\Gamma(t)$, so we obtain $\varepsilon \equiv v(A) \bmod{4}$. This concludes the proof, because we have $v(A)=1$ if and only if $v(\Delta)<6$, and similarly $v(A)=3 \equiv -1 \bmod{4}$ if and only if $v(\Delta)>6$.

Finally, in the case $e=6$ (resp.~$e=3$) one takes $u:=\pi_e^{-v(B)}$ (resp.~$u:=\pi_e^{-v(B)/2}$), and by a similar calculation finds that $\tau_e$ acts on $\log_\Gamma(t)$ as multiplication by $\zeta_e^{v(B)}$ (resp.~$\zeta_e^{-v(B)}$). Since $v(\Delta)=2v(B)$ can only be $2$ or $10$ (resp.~$4$ or $8$), we find in all cases that the eigenvalue of $\tau_e$ acting on $e_1 \otimes 1$ is $\zeta_e$ if $v(\Delta)<6$ and $\zeta_e^{-1}$ if $v(\Delta)>6$. \end{proof}

\begin{remark}
    The result of the previous proposition agrees with \cite[Remarque on p.~133]{Volkov}, which however does not explicitly give the value of $\varepsilon$ in the case $v(\alpha)=1$. 
    As mentioned in \textit{loc.~cit.}, one can deduce the value of $\varepsilon$ (unless $v(\alpha)=1$) by using Kraus' description of the mod $p$ representation via fundamental characters \cite[\S 2.3.2, Proposition 2 and Lemme 2]{KrausPoidsConducteur}.

\end{remark}

We next determine the valuation of $\alpha$ purely in terms of invariants that can be easily computed from a Weierstrass model.

\begin{proposition}\label{proposition: formula valuation alpha}
Let $E/\Qp$ be a potential $e$-lift of $\tilde{E}/\mathbb{F}_p$ with $j(E)=j$, Volkov parameter $\alpha$ and minimal discriminant $\Delta$. Then:
\begin{enumerate}
\item Suppose $j \neq 0,1728$. The valuation of $\alpha$ is given by the following table: 
    
\begin{center}
 \begin{tabular}{c|ccc}
     & $e=3$ & $e=4$ & $e=6$ \\ \hline 
     $v(\Delta) < 6$  & $\frac{1}{3}(5-v(j))$ & $\frac{1}{2}(3-v(j-1728))$ & $\frac{1}{3}(4-v(j))$ \\[0.1in]
     $v(\Delta) > 6$ & $\frac{1}{3}(2+v(j))$ & $\frac{1}{2}(1+v(j-1728))$ & $\frac{1}{3}(1+v(j))$
 \end{tabular}
\end{center}
\item $E$ has a canonical subgroup if and only if $v(\alpha)=1$, if and only if $v(j) \in \{1,2\}$ (resp.~$v(j-1728)=1$) if $e \in \{3,6\}$ (resp.~$e=4$).
\item Up to possibly replacing $E$ by a ramified quadratic twist, we have
\[
v(\alpha^{-1})= \begin{cases}
\frac{1}{3}(v(j)-4) &\text{ if } e \in \{3,6\} \text{ and } v(j) \equiv 1 \bmod{3}; \\
\frac{1}{3}(v(j)-5) &\text{ if } e \in \{3,6\} \text{ and } v(j) \equiv -1 \bmod{3}; \\
\frac{1}{2}(v(j-1728)-3) &\text{ if } e=4.
\end{cases}
\]
Moreover, if $E$ does not have a canonical subgroup, then the formulae above all satisfy $v(\alpha^{-1}) \geq 0$.
\end{enumerate}
\end{proposition}

\begin{proof}
Note that $v(\beta)$ is given by \Cref{prop: valuation beta}. A straightforward case analysis with \Cref{prop: beta parameter}(iii)-(iv) then gives $v(\alpha)$ for (i); (ii) similarly follows since having a canonical subgroup is equivalent to $v(\alpha)=1$ (\Cref{rmk: canonical subgp}). 
For (iii), observe that taking a ramified quadratic twist connects the pairs $\{\frac{1}{3}(5-v(j)), \frac{1}{3}(1+v(j))\}$, $\{\frac{1}{3}(2+v(j)),\frac{1}{3}(4-v(j))\}$ and $\{\frac{1}{2}(3-v(j-1728)),\frac{1}{2}(1+v(j-1728))\}$. In each case, the chosen options satisfy $v(\alpha^{-1}) \geq - 1$ since $j \in \Z_p$ and $v(\alpha^{-1})$ is an integer (note this also forces the specified congruence condition on $v(j) \bmod{3}$). If $E$ doesn't have a canonical subgroup, then neither does any quadratic twist, and hence the indicated choice satisfies $v(\alpha^{-1}) \geq 0$. 
\end{proof}

We now finally apply the description of $v(\alpha)$ given in \Cref{proposition: formula valuation alpha} to complete the proof of \Cref{introthm: more precise description padic Galreps}.

\begin{corollary}\label{cor: explicit p-adic image}
    Let $p>7$ be a prime and let $E/\Q_p$ be an elliptic curve. Suppose that $\operatorname{Im}\rho_{E,p} \subseteq C_{ns}^+(p)$ and that $E$ has semistability defect $e \in \{3,4,6\}$. Set $j=j(E)$ and let $n_0$ be the positive integer
    \[
    n_0 := \begin{cases}
    \begin{array}{lll}
    \!\left\lfloor \frac{1}{3}v(j) \right\rfloor & \quad & \text{if $e \in \{3,6\}$} \\
    \!\left\lfloor \frac{1}{2}v(j-1728) \right\rfloor & \quad & \text{if $e = 4$.}
    \end{array}
    \end{cases}
    \]
    Let $\pi_{n_0} : \GL_2(\Z_p) \to \GL_2(\Z/p^{n_0}\Z)$ be the canonical projection. If $p>\sqrt{n_0+1}$, then 
\[
\operatorname{Im}\rho_{E,p^{n_0}} \subseteq C_{ns}^+(p^{n_0}) \text{ with index dividing $3$,} \qquad\qquad \text{and} \,\, \operatorname{Im}\rho_{E,p^\infty} = \pi_{n_0}^{-1}(\operatorname{Im}\rho_{E,p^{n_0}}).
\]
Moreover, if $e=4$ or $p \not\equiv 2,5 \bmod 9$, then $\operatorname{Im}\rho_{E,p^{n_0}} = C_{ns}^+(p^{n_0})$. 
\end{corollary}

\begin{proof}
Note that $E$ does not have a canonical subgroup by \Cref{cor: supersingular and no canonical subgroup}, hence $n_0>0$ and $v(\alpha) \neq 1$ by \Cref{proposition: formula valuation alpha}(ii). By reduction step (iv) in \Cref{subsec: preliminary reductions}, we are free to replace $E$ by a quadratic twist. After possibly doing so, we may apply \Cref{proposition: formula valuation alpha}(iii) to get an explicit formula for $v(\alpha)$ in terms of $j$. It is immediate to check that this formula may be written as $v(\alpha^{-1})=n_0-1$ with $n_0 \geq 1$ as in the statement, hence by \Cref{thm: more precise description padic Galreps} we have $\operatorname{Im}\rho_{E,p^{n_0}} \subseteq C_{ns}^+(p^{n_0})$ and $\operatorname{Im}\rho_{E,p^\infty} = \pi_{n_0}^{-1}(\operatorname{Im}\rho_{E,p^{n_0}})$.

It remains to verify the claim on the index $[C_{ns}^+(p^{n_0}):\operatorname{Im}\rho_{E,p^{n_0}}]=[C_{ns}^+(p) : \operatorname{Im} \rho_{E, p}]$. To do this, we consider the action of inertia on $E[p] \otimes_{\mathbb{F}_p} \overline{\mathbb{F}}_p$ in terms of the level $2$ fundamental character $\psi_2$. Volkov \cite[p.~129, p.~131]{Volkov} shows the action is via $\{\psi_2^{1-\frac{p^2-1}{e}},\psi_2^{p+\frac{p^2-1}{e}} \}$ if $v(\alpha)\leq 0$ (resp.~$\{\psi_2^{1+\frac{p^2-1}{e}},\psi_2^{p-\frac{p^2-1}{e}} \}$ if $v(\alpha)\geq 2$). Now the index is $3$ if and only if both exponents $1 \mp \frac{p^2-1}{e}, p \pm \frac{p^2-1}{e}$ are divisible by $3$; a quick check shows that if $e=4$ or $p \not \equiv 2,5 \bmod{9}$ then $1 \mp \frac{p^2-1}{e} \not\equiv 0 \bmod 3$ which completes the proof.
\end{proof}

\begin{remark}\label{rmk: index 3}
The proof of the previous corollary also shows how to determine whether the index $[C_{ns}^+(p^{n_0}) : \operatorname{Im}\rho_{E,p^{n_0}}]$ is $1$ or $3$ in the missing cases when $e \in \{3,6\}$ and $p \equiv 2,5 \bmod{9}$ via studying the exponents of the fundamental characters in the representation $E[p] \otimes \overline{\F}_p$.

Let $\Delta$ be a minimal discriminant of $E$. Using that $v(\alpha) \leq 0$ if and only if $v(\Delta)<6$ (\Cref{proposition: formula valuation alpha}(ii)), and recalling that $v(\Delta) \in \{\frac{12}{e},\frac{12(e-1)}{e}\}$ (cf.~\Cref{lemma: quadratic twist acquires good reduction over a small extension}), the index is $3$ precisely in the following cases:
\[
 p \equiv 2 \bmod{9} \,\,\, \text{ and } v(\Delta) \in \{4,10\}; \qquad\qquad \text{or } \,\,\, p \equiv 5 \bmod{9} \,\,\, \text{ and } v(\Delta) \in \{2,8\}.
\]

Moreover, all four of these cases occur. For $p \equiv 2, 5 \bmod{9}$, consider the elliptic curves $E_i/\Q_p : y^2 = x^3 + p^i$ for $i \in \{1, 2, 4, 5\}$. These have potentially good supersingular reduction, minimal discriminant $p^{2i}$, and semistability defect $e = 6/\gcd(2,i)$. Comparing with the conditions described above, we deduce that for every prime $p \equiv 2,5 \bmod{9}$ and each $e \in \{3,6\}$, the case $[C_{ns}^+(p) : \operatorname{Im}\rho_{E,p}] = 3$ does indeed occur.

We also remark that $C_{ns}^+(p)$ has a unique index $3$ subgroup up to conjugacy, so knowledge of the index suffices to describe the Galois image.
\end{remark}

\subsection{Examples}
 
We now present several examples illustrating the computation of $\alpha$ and $\beta$, as well as their application to our $p$-adic image results. Throughout the examples, we write $\alpha_i, \beta_i$ for the $\alpha$- and $\beta$-parameters of the elliptic curve $E_i/\Q_p$.

\begin{example}
Let $p=11$ and let $E_1/\Q_{p}: y^2=x^3+p^3x+p^2$. This has semistability defect $e=3$ with minimal model $E_1/L: y^2=x^3+p\pi_3^2x+1$. One can then compute the following coefficients of $\log_{\hat{E}_1}$:
    \begin{center}
      \begin{tabular}{ccc}
        $ d_{p} = 20\pi_e^2$, & $d_{p^2}= \frac{59003}{p} + O(p^4)$, & $d_{p^3}= \frac{-62940\pi_e^2}{p} + O(p^4)$, \\
         & $d_{p^4}= \frac{370910}{p^2} + O(p^4)$, & $d_{p^5}= \frac{-859443\pi_e^2}{p^2} + O(p^4)$.
    \end{tabular}   
    \end{center}
    Using \Cref{prop: beta algorithm}, we hence get:
    \[
    \beta_1 \equiv -\frac{p}{\pi_e} \frac{d_{p^3}}{d_{p^2}} \equiv \frac{62940p\pi_e}{59003} \equiv 0 \bmod{\pi_e^4} \qquad \text{ and } \quad \beta_1 \equiv -\frac{p}{\pi_e} \frac{d_{p^5}}{d_{p^4}} \equiv 2p\pi_e \bmod{\pi_e^7};
    \]
    in particular $v(\beta_1)=\frac{4}{3}$. By \Cref{prop: beta parameter}(i) we must have $\beta_1 \in \mathbb{Z}_{11} \pi_e$ and hence $\varepsilon_1=+1$ (\Cref{prop: beta parameter}(iii)). Applying \Cref{prop: beta parameter}(iv), we find that $\alpha_1 = 5 + O(p) \in \Z_p$ and in particular $v(\alpha_1^{-1})=0$. 

    \Cref{thm: more precise description padic Galreps} then shows that the conclusion of \Cref{thm: no sharp group over Qp}(ii) holds with $n=1$. More explicitly, the image of the $p$-adic Galois representation of $E$ is the inverse image in $\GL_2(\Z_p)$ of its mod $p$ representation, which is index $3$ subgroup of $G_1$ of $C_{ns}^+(p)$ by \Cref{rmk: index 3}. Therefore $\op{Im} \rho_{E_1,p^{\infty}} = \pi_1^{-1}(G_1)$.
\end{example}

\begin{example}
Let $p>5$ be a prime with $p \equiv 2 \bmod{3}$. Let $E_2/\Qp : y^2 = x^3 + p^4 x + p^2$, which has potentially supersingular reduction with semistability defect $e=3$ and $j$-invariant $1728 \cdot \frac{p^8}{p^8 + 27/4}$. Using \Cref{prop: valuation beta} yields $v(\beta_2)=7/3$. As in the previous example we have $\varepsilon=1$, from which it follows that $v(\alpha_2^{-1})=1$. \Cref{thm: more precise description padic Galreps} hence shows that the mod $p^2$ representation of $E_2$ has image contained in the normaliser of a non-split Cartan, while $\operatorname{Im} \rho_{E_2, p^\infty}$ is the inverse image in $\GL_2(\Z_p)$ of $\operatorname{Im} \rho_{E_2, p^2}$.
\end{example}

\begin{example}\label{ex: ambiguous alpha 2}
Let $p>7$ be a prime with $p \equiv 3 \bmod{4}$ and let $E_3/\Q_p: y^2=x^3+px+p^2$. Then $E_3$ has semistability defect $e=4$ and satisfies $j(E_3)-1728 = -1728 \cdot \frac{p}{p + 4/27}$.
By \Cref{prop: valuation beta}, we have $v(\beta_3)=\frac{1}{4}$ and by \Cref{thm: valuation alpha overview}(ii) we have $v(\alpha_3^{-1})=-1$ (for both twists; the curve $E_3$ has minimal discriminant of valuation $3$ and hence $\varepsilon=+1$ by \Cref{prop: epsilon via discriminant}). This is a case not covered by \Cref{thm: more precise description padic Galreps}: $E_3$ has a canonical subgroup (see \Cref{rmk: alpha role}(ii) or use $v(j-1728)=1$ with \Cref{thm: valuation alpha overview}(iii)) and so the mod $p$ image is not contained in $C_{ns}^+(p)$ since $p>7$ (see \Cref{cor: supersingular and no canonical subgroup}).
\end{example}

\section{Application to elliptic curves over $\Q$}\label{sect: adelic index}

\subsection{Description of the $p$-adic Galois image}\label{subsec: global description p adic image}

As a consequence of \Cref{thm: no sharp group over Qp} and \cite[Theorem 1.6]{furio24} we obtain the following structure theorem for the global image of the $p$-adic representations attached to non-CM elliptic curves over $\Q$. Note that, unlike \Cref{thm: no sharp group over Qp}, this is a statement about representations of the absolute Galois group of $\mathbb{Q}$ rather than of $\Qp$.

\begin{theorem}\label{thm: no sharp groups over Q}
    Let $E/\Q$ be a non-CM elliptic curve and $p > 7$ be a prime. Consider the mod $p$ and $p$-adic Galois representations
\[
\rho_{E, p} : \op{Gal}(\overline{\Q}/\Q) \to \GL_2(\F_p) \quad \text{and} \quad \rho_{E, p^\infty} : \op{Gal}(\overline{\Q}/\Q) \to \GL_2(\Z_p).
\] 
    Suppose that $\op{Im} \rho_{E, p} \subseteq C_{ns}^+(p)$. Then there exists $n \geq 1$ such that 
    \[
    \op{Im} \rho_{E, p^\infty} = \pi_n^{-1} \left(C_{ns}^+(p^n)\right),
    \]
where $\pi_n : \GL_2(\Z_p) \to \GL_2(\Z/p^n\Z)$ is the canonical projection.
\end{theorem}

\begin{proof}
We first apply \cite[Theorem 1.6]{furio24}. Since $p>7$, either: $\op{Im} \rho_{E,p^{\infty}} = \pi_n^{-1}(C_{ns}^+(p^n))$; or $\op{Im} \rho_{E,p^2} \cong C_{ns}^+(p) \ltimes \{ \op{Id} + p \left(\begin{smallmatrix} a & b\varepsilon \\ -b & c \end{smallmatrix}\right) : a,b,c \in \mathbb{F}_p \}.$ For contradiction, suppose the second option holds. 
We now examine the local image $G:=\rho_{E,p^{\infty}}(\op{Gal}(\Qpbar/\Qp)) \subseteq \op{Im} \rho_{E,p^{\infty}}$, in particular the group $\ker(G(p^2) \rightarrow G(p))$. Recall the definition and notation of the irreducible submodules of $M_{2}(\mathbb{F}_p)$ given in \Cref{lemma: conjugacy action of Cnsp}(iii). From the description above, one sees directly that $\ker(G(p^2) \rightarrow G(p)) \subseteq V_1 \oplus V_3$.
On the other hand, \Cref{thm: no sharp group over Qp} states that either $\ker(G(p^2) \to G(p)) = V_1 \oplus V_2$ (in case (i), and in case (ii) if $n>2$), or $\ker(G(p^2) \to G(p)) = V_1 \oplus V_2 \oplus V_3$ (otherwise). In either case, we have a contradiction that completes the proof.
\end{proof}

\begin{remark}
\Cref{rmk: uniformity} shows that the image of the local Galois group $\op{Gal}(\Qpbar/\Qp)$ via $\rho_{E, p^\infty}$ is not necessarily open in $\GL_2(\Z_p)$. This implies that for elliptic curves over $\Q$, the value of $n$ in \Cref{thm: no sharp groups over Q} cannot be bounded by purely $p$-adic arguments.
\end{remark}

\subsection{Adelic bounds} 
We now give an improved version of \cite[Theorem 1.5]{furio24}. The key of the improvement consists in the fact that, thanks to \Cref{thm: no sharp groups over Q}, we are able to use a sharper bound on the index of the $p$-adic representations: if $n$ is the smallest integer such that $\operatorname{Im}\rho_{E,p^\infty} \supset \op{Id} + p^nM_2(\Z_p)$, then $[\GL_2(\Z_p) : \operatorname{Im}\rho_{E,p^\infty}] \le \frac{1}{2}p^{2n}$, while in \cite{furio24} the bound $[\GL_2(\Z_p) : \operatorname{Im}\rho_{E,p^\infty}] \le \frac{1}{2}p^{3n}$ was used.

\begin{theorem}\label{thm: adelic bound}
    Let $E/\Q$ be a non-CM elliptic curve and let $\operatorname{h}(j(E))$ be the logarithmic Weil height of its $j$-invariant. Let $\rho_E=\prod_p \rho_{E,p^\infty} : \Gal(\overline{\Q}/\Q) \rightarrow \GL_2(\widehat{\Z})$. We have
    \[
    [\GL_2(\widehat{\Z}) : \operatorname{Im} \rho_E] < 3.7 \cdot 10^{19} (\operatorname{h}(j(E)) + 480)^{3.11}.
    \]
    Moreover, 
    \begin{align*}
		[\GL_2(\widehat{\Z}) : \operatorname{Im}\rho_E] < 2.8 \cdot 10^{20} (\operatorname{h}(j(E)) + 270)^{2 + 3.251 \cdot \delta(12 \operatorname{h}(j(E)))}.
    \end{align*}
    where $\delta: (-0.75,\infty) \rightarrow \mathbb{R}$ is defined as 
		$\delta(x) := \frac{1}{\log(\log(x+40) + 7.6) - 0.903}.$

    In particular, we have $[\GL_2(\widehat{\Z}) : \operatorname{Im}\rho_E] < \operatorname{h}(j(E))^{2+O\left(\frac{1}{\log\log\operatorname{h}(j(E))}\right)}$ as $\operatorname{h}(j(E))$ tends to $\infty$.
\end{theorem}

Before proving this theorem, we state and prove two lemmas. The first exploits \Cref{thm: no sharp groups over Q} to improve \cite[Corollary 6.8]{furio24}, whereas the second consolidates several of the results of \cite{furio24}, which we will then apply in the proof of \Cref{thm: adelic bound}.

\begin{lemma}\label{lemma: p-adic indices}
    Let $E/\Q$ be a non-CM elliptic curve and let $p$ be an odd prime such that $\operatorname{Im}\rho_{E,p} \subseteq C_{ns}^+(p)$, with equality holding in the case $p=3$. Let $n \ge 1$ be the largest integer for which $\operatorname{Im}\rho_{E,p^n} \subseteq C_{ns}^+(p^n)$. If $j(E) \ne 2^{4} \cdot 3^{2} \cdot 5^{7} \cdot 23^{3}$, we have
	\begin{align*}
        [\GL_2(\Z_3) : \operatorname{Im}\rho_{E,3^\infty}] &\le 3^{2n} &\text{if } p=3, \\
        [\GL_2(\Z_5) : \operatorname{Im}\rho_{E,5^\infty}] &\le \max\{2 \cdot 5^{2n-1}, 30\} \le \frac{6}{5} \cdot 5^{2n} &\text{if } p=5,\\
        [\GL_2(\Z_7) : \operatorname{Im}\rho_{E,7^\infty}] &\le \max\{3 \cdot 7^{2n-1}, 147\} \le 3 \cdot 7^{2n} &\text{if } p=7, \\
        [\GL_2(\Z_p) : \operatorname{Im}\rho_{E,p^\infty}] &= \frac{p-1}{2p} \cdot p^{2n} &\text{if } p > 7.
	\end{align*}
\end{lemma}

\begin{proof}
    For $p \in \{3,7\}$ the conclusion follows immediately from \cite[Proposition 6.7]{furio24}. If $p>7$, then by \Cref{thm: no sharp groups over Q}, $[\GL_2(\Z_p) : \operatorname{Im}\rho_{E,p^\infty}] = [\op{GL}_2(\Z/p^n\Z):C_{ns}^+(p^n)] = \frac{p^{2n}-p^{2n-1}}{2} =\frac{p-1}{2p} \cdot p^{2n}$.

    If $p=5$, then $\operatorname{Im}\rho_{E,p^2} \not\subseteq G_{ns}^\#(p^2)$ by \cite[Theorem 1.6]{rszb22} since $j(E) \ne 2^{4} \cdot 3^{2} \cdot 5^{7} \cdot 23^{3}$; the result now follows from the classification in \cite[Theorem 6.5]{furio24}.
\end{proof}

\begin{remark}
    Note that in \Cref{lemma: p-adic indices} we invoke the result of \cite[Proposition 6.7]{furio24} for $p \in \{3,7\}$. In these cases, however, it is now known that the $p$-adic indices are uniformly bounded (see \cite[Theorem 1.4]{balakrishnan25} and \cite[Corollary 1.5]{furio20257adicgaloisrepresentationselliptic}). We nonetheless retain the above formulation as it is more convenient for our purposes, and the use of the uniform bound would not yield any significant improvements in \Cref{thm: adelic bound}.
\end{remark}

\begin{lemma}\label{lemma: furio known}
Let $E/\Q$ be a non-CM elliptic curve with $j$-invariant $j(E)$ and stable Faltings height $h_{\mathcal{F}}(E)$. Define

\begin{align*}
\mathcal{J} =& \{ -11 \, {\cdot} \, 131^3,  \:\:\, -2^{-1} \, {\cdot} \, 17^2\, {\cdot} \, 101^3, \:\:\, -2^{-17}\, {\cdot} \, 17 \, {\cdot} \, 373^3, \:\:\, -7 \, {\cdot} \, 137^3 \, {\cdot} \, 2083^3, \:\:\, 2^4\, {\cdot} \, 3^{-13} \, {\cdot} \, 5 \, {\cdot} \, 13^4 \, {\cdot} \, 17^3, \:\:\, -7 \, {\cdot} \, 11^3, \:\:\, \\
& \:\:\, {-}11^2, -2^{12} \, {\cdot} \, 3^{-13} \, {\cdot} \, 5^3 \, {\cdot} \, 11 \, {\cdot} \, 13^4, \:\:\, 2^{18} \, {\cdot} \, 3^3 \, {\cdot} \, 5^{-13} \, {\cdot} \, 13^4 \, {\cdot} \, 61^{-13} \, {\cdot} \, 127^3 \, {\cdot} \, 139^3 \, {\cdot} \, 157^3 \, {\cdot} \, 283^3 \, {\cdot} \, 929, \:\:\, 2^4 \, {\cdot} \, 3^2 \, {\cdot} \, 5^7 \, {\cdot} \, 23^3 \}. 
\end{align*}

\noindent One of the following holds:
\begin{itemize}
    \item[(A)] there exists a prime $p>13$ such that $\op{Im} \rho_{E,p}=C_{ns}^+(p)$;
     \item[(B)] $\rho_{E,p}$ is surjective for all primes $p>13$;
    \item[(C)] $j(E) \in \mathcal{J}$.
\end{itemize}

\noindent Moreover,
\begin{itemize}
    \item if case {\normalfont(B)} holds, then 
    \[
    [\op{GL}_2(\widehat{\Z}): \op{Im} \rho_E] < \begin{cases}
    8.4 \cdot 10^{22} \cdot (h_{\mathcal{F}}(E)+40)^{2.616}, \\
    4 \cdot 10^{22} \cdot (h_{\mathcal{F}}(E)+40)^{1.814 \cdot \delta(h_{\mathcal{F}}(E))}(h_{\mathcal{F}}(E)+22.5)^2;
    \end{cases}
    \]
    \item if case {\normalfont(C)} holds, then $[\op{GL}_2(\widehat{\Z}): \op{Im} \rho_E] \leq 2736$.
\end{itemize}
\end{lemma}

\begin{proof}
    The classification is \cite[Proposition 7.15]{furio24}. If $j(E) \in \mathcal{J}$ and $j(E) \neq 2^4 \cdot 3^2 \cdot 5^7 \cdot 23^3$, then the index follows directly from \cite[Proposition 7.12]{furio24}. If $j(E)=2^4 \cdot 3^2 \cdot 5^7 \cdot 23^3$, then we note that the adelic index only depends on the $\overline{\Q}$-isomorphism class by \cite[Corollary 2.3]{zywina15index} and use the LMFDB (namely the curve \href{https://beta.lmfdb.org/EllipticCurve/Q/396900/b/1}{396900.b1}) to find that the index is $200 \leq 2736$. In case (B), the bounds are given in part (B) of \cite[Proof of Theorem 7.1]{furio24}.
\end{proof}

\begin{proof}[Proof of \Cref{thm: adelic bound}]
    It suffices to bound the adelic index in case (A) of \Cref{lemma: furio known}. Define $\mathcal{C}_{ns}$ as the set of primes $p>5$ for which $\operatorname{Im}\rho_{E,p} \subseteq C_{ns}^+(p)$ and let $\mathcal{C}=\mathcal{C}_{ns} \cup \{5\}$. Let $\beta$ be the number of primes $p \in \mathcal{C}_{ns}$ for which $E$ has bad reduction at $p$.
    For every $p \geq 3$ let $n_p$ be the maximal integer $n \geq 0$ for which $\operatorname{Im}\rho_{E,p^n} \subseteq C_{ns}^+(p^n)$. Note that $n_p=0$ if $\operatorname{Im} \rho_{E, p}$ is not contained in $C_{ns}^+(p)$.
    Set $\Lambda := \prod_{p>2} p^{n_p}$,
    \[
    \op{Ind}(p):=[\op{GL}_2(\Z_p):\op{Im} \rho_{E,p^{\infty}}] \qquad \text{and} \qquad \operatorname{Ind}_S(p) = [ \SL_2(\Z_p) : \rho_{E, p^\infty}(\Gal(\overline{\Q} / \Q^{\operatorname{ab}})) ].
    \]
    As we assume to be in case (A), the proof of \cite[Lemma 7.18]{furio24} (in particular Equation (7.7)) gives
    \begin{equation*}
        [\GL_2(\widehat{\Z}) : \operatorname{Im}\rho_E] \le 24^2 \Delta_7 \cdot 2^{|\mathcal{C}_{ns}|} \cdot 3^\beta \cdot \operatorname{Ind}_S(2) \cdot \operatorname{Ind}_S(3) \cdot \prod_{p \in \mathcal{C}} \op{Ind}(p),
    \end{equation*}
    where $\Delta_7$ is defined as $1$ if $7 \notin \mathcal{C}$, as $8$ if $7 \in \mathcal{C}$ and $E$ has good reduction at $7$, and as $\frac{8}{3}$ if $7 \in \mathcal{C}$ and $E$ has bad reduction at $7$.

    We can now follow the proof of \cite[Lemma 7.18]{furio24}, in particular the part \textit{Bounding the $p$-adic indices}, and bound $\op{Ind}(p)$ for $p \in \mathcal{C}$ and $\op{Ind}_S(p)$ for $p \in \{2,3\}$. 
    
    If $p=2$, then the arguments in \textit{loc.~cit.}~show we may assume that $\op{Ind}_S(2) \leq 32$. Similarly, if $p=3,5$ and $\op{Im} \rho_{E, p} \not \subseteq C_{ns}^+(p)$, then we find $\op{Ind}_S(3) \leq 27$ and $\op{Ind}(5) \leq 5$; if instead $\op{Im} \rho_{E, p} \subseteq C_{ns}^+(p)$, using \Cref{lemma: p-adic indices} we have $\op{Ind}_S(3) \le 2 \op{Ind}(3) \le 2 \cdot 3^{2n_3}$ and $\op{Ind}(5) \le \frac{6}{5} \cdot 5^{2n_5}$. In all cases, this gives $\op{Ind}_S(3) \leq 27 \cdot 3^{2n_3}$ and $\op{Ind}(5) \leq 5 \cdot 5^{2n_5}$. For the primes in $\mathcal{C}_{ns}$, we can directly apply \Cref{lemma: p-adic indices}, obtaining $\op{Ind}(7) \leq \frac{6}{2} \cdot 7^{2n_7}$ and $\op{Ind}(p) \leq  \frac{1}{2}p^{2n_p}$ for $p > 7$. Replacing all these bounds in the inequality above, we obtain
	\begin{align}\label{eq: boundnonsuq}
		[\GL_2(\widehat{\Z}) : \operatorname{Im}\rho_E] &\leq (2^{12} \cdot 3^6 \cdot 5 ) \cdot 3^\beta \cdot \Delta_7 \cdot \Lambda^2 \leq 1.5 \cdot 10^7 \cdot \Delta_7 \cdot 3^{\beta} \cdot \Lambda^2.
	\end{align}
    
    Observe that this inequality is similar to that of \cite[Lemma 7.18]{furio24}, however the exponent of $\Lambda$ is decreased by $1$ and the constant is multiplied by $6$.  To conclude, it suffices to mimic part (A) of the proof of \cite[Theorem 7.1]{furio24}, doing the calculations with $\Lambda^2$ in place of $\Lambda^3$ with the improved bound in \Cref{eq: boundnonsuq}. 
    We report here the new inequalities obtained. Let $\Fheight(E)$ be the stable Faltings height of $E$. Recall that $\Fheight(E) > -0.75$ (see \cite[Remark 2.4]{furio24}). If $j \notin \Z$ then
    \begin{equation}
    [\GL_2(\widehat{\Z}) : \operatorname{Im}\rho_E] < \begin{cases}
        6.2 \cdot 10^{10} (\Fheight(E)+1.5)^{2.38}, \\
        5 \cdot 10^{13} (\Fheight(E) +40)^{1.437 \cdot \delta(\Fheight(E))} (\Fheight(E) + 1.5)^2;
    \end{cases}
    \end{equation}
    while for $j \in \Z$ we have
    \begin{equation}
        [\GL_2(\widehat{\Z}) : \operatorname{Im}\rho_E] < \begin{cases}
        2.72 \cdot 10^{17} (\Fheight(E) + 40)^{3.11}, \\
        2.35 \cdot 10^{16} (\Fheight(E) +40)^{3.251 \cdot \delta(\Fheight(E))} (\Fheight(E) + 22.5)^2.
    \end{cases}
    \end{equation}
    Combining these bounds with those of \Cref{lemma: furio known}, we obtain in any case that
    \begin{align*}
        [\GL_2(\widehat{\Z}) : \operatorname{Im} \rho_E] &< \begin{cases} 8.4 \cdot 10^{22} (\Fheight(E) + 40)^{3.11} \\
		4 \cdot 10^{22} (\Fheight(E) + 22.5)^{2 + 3.251 \cdot \delta(\Fheight(E))}. \end{cases}
    \end{align*}
    By \cite[Theorem 2.3]{furio24} we have that $\Fheight(E) \le \frac{1}{12}\operatorname{h}(j(E))$, and applying this bound in the inequalities above, we conclude the proof.
\end{proof}

\bibliographystyle{alpha}
\bibliography{biblio}

\end{document}